\pdfoutput=1
\documentclass[11pt,a4paper]{amsart}
\usepackage[margin=1in]{geometry}

\usepackage[T1]{fontenc}
\usepackage{textcase}
\usepackage[dvipsnames]{xcolor}
\usepackage{microtype}
\usepackage{fnpct}

\usepackage{amsmath}
\usepackage{amssymb}
\usepackage{eucal}
\usepackage{mathrsfs}
\usepackage{xfrac}
\usepackage{tikz-cd}
\renewcommand{\injlim}{\varinjlim}
\renewcommand{\projlim}{\varprojlim}

\usepackage{enumitem}
\usepackage{booktabs}
\usepackage[pdfusetitle,colorlinks]{hyperref}
\hypersetup{bookmarksdepth=2,pdfencoding=unicode,allcolors=MidnightBlue}
\usepackage{zref-clever}
\zcsetup{abbrev=false,cap=true,nameinlink=false,sort=false}
\newcommand{\cref}[1]{\zcref{#1}}
\newcommand{\Cref}[1]{\zcref[S]{#1}}
\zcsetup{pairsep={ and~},lastsep={, and~}}
\zcRefTypeSetup{equation}{Name-sg=,Name-pl=,refbounds={(,,,)}}
\zcRefTypeSetup{item}{Name-sg=,Name-pl=,refbounds={(,,,)}}
\newlist{conenum}{enumerate}{1}
\setlist[conenum,1]{label=(\roman*),ref=\roman*}
\zcRefTypeSetup{conenumi}{Name-sg=,Name-pl=,refbounds={(,,,)}}

\NewDocumentCommand{\newzctheorem}{momo}{\IfValueTF{#4}
  {\newtheorem{#1}{#3}[#4]}
  {\IfValueTF{#2}
    {\AddToHook{env/#1/begin}{\zcsetup{countertype={#2=#1}}}\newtheorem{#1}[#2]{#3}}
    {\newtheorem{#1}{#3}}}}
\numberwithin{equation}{section}
\theoremstyle{plain}
\newtheorem{Theorem}{Theorem}

\zcRefTypeSetup{Theorem}{Name-sg=Theorem,Name-pl=Theorems}
\newzctheorem{theorem}[equation]{Theorem}
\newzctheorem{proposition}[equation]{Proposition}
\newzctheorem{lemma}[equation]{Lemma}
\newzctheorem{corollary}[equation]{Corollary}
\theoremstyle{definition}
\newzctheorem{definition}[equation]{Definition}
\newzctheorem{assumption}[equation]{Assumption}
\newzctheorem{example}[equation]{Example}
\newzctheorem{question}[equation]{Question}
\theoremstyle{remark}
\newzctheorem{remark}[equation]{Remark}
\newzctheorem{slogan}[equation]{Slogan}

\let\oldAA\AA\let\AA\relax
\let\oldL\L\let\L\relax
\let\oldSS\SS\let\SS\relax

\newcommand{\NN}{\mathbf{N}}
\newcommand{\FF}{\mathbf{F}}
\newcommand{\ZZ}{\mathbf{Z}}
\newcommand{\QQ}{\mathbf{Q}}
\newcommand{\RR}{\mathbf{R}}
\newcommand{\SS}{\mathbf{S}}
\newcommand{\CC}{\mathbf{C}}

\newcommand{\AA}{\mathbb{A}}
\newcommand{\BB}{\mathbb{B}}
\newcommand{\DD}{\mathbb{D}}
\newcommand{\GG}{\mathbb{G}}

\newcommand{\PP}{\mathbb{P}}

\newcommand{\E}{\mathrm{E}}

\newcommand{\KAS}{\textnormal{KAS}}

\newcommand{\lgt}{\textnormal{lgt}}
\newcommand{\alg}{\textnormal{alg}}

\newcommand{\Nis}{\textnormal{Nis}}
\newcommand{\et}{\textnormal{ét}}
\newcommand{\fin}{\textnormal{fin}}
\newcommand{\mot}{\textnormal{mot}}
\newcommand{\hyp}{\textnormal{hyp}}

\newcommand{\st}{\textnormal{st}}
\newcommand{\fet}{\textnormal{fét}}
\newcommand{\iso}{\textnormal{iso}}
\newcommand{\tfp}{\textnormal{tfp}}
\newcommand{\all}{\textnormal{all}}

\newcommand{\L}{\operatorname{L}}
\newcommand{\R}{\operatorname{R}}

\newcommand{\CAlg}{\operatorname{CAlg}}
\newcommand{\D}[1][]{\mathop{\kern0pt#1\mathrm{D}}\nolimits}
\newcommand{\End}{\operatorname{End}}
\newcommand{\Fun}{\operatorname{Fun}}

\newcommand{\Gest}{\operatorname{Gest}}

\newcommand{\Hom}{\operatorname{Hom}}
\newcommand{\Ind}{\operatorname{Ind}}

\newcommand{\Map}{\operatorname{Map}}
\newcommand{\Mod}{\operatorname{Mod}}
\newcommand{\Mot}[1][]{\mathop{\kern0pt#1\mathrm{Mot}}\nolimits}
\newcommand{\Nm}{\operatorname{Nm}}
\newcommand{\PShv}[1][]{\mathop{\kern0pt#1\mathrm{PShv}}\nolimits}
\newcommand{\Pic}{\operatorname{Pic}}
\newcommand{\Pro}{\operatorname{Pro}}
\newcommand{\SH}[1][]{\mathop{\kern0pt#1\mathrm{SH}}\nolimits}
\newcommand{\Shv}[1][]{\mathop{\kern0pt#1\mathrm{Shv}}\nolimits}
\newcommand{\Span}[1][]{\mathop{\kern0pt#1\mathrm{Span}}\nolimits}
\newcommand{\Sparc}{\operatorname{Sparc}}

\newcommand{\Spec}{\operatorname{Spec}}
\newcommand{\Sp}{\operatorname{M}}
\newcommand{\Tot}{\operatorname{Tot}}
\newcommand{\awn}{\operatorname{awn}}

\newcommand{\dR}{\operatorname{dR}}

\newcommand{\fib}{\operatorname{fib}}
\newcommand{\id}{\operatorname{id}}

\newcommand{\op}{\operatorname{op}}
\newcommand{\pr}{\operatorname{pr}}

\newcommand{\ucpl}{\operatorname{ucpl}}

\newcommand{\lsm}{\textnormal{lax-\(\otimes\)}}

\newcommand{\cat}[1]{\mathcal{#1}}
\newcommand{\Cat}[1]{\mathsf{#1}}
\newcommand{\Cls}[1]{\mathscr{#1}}
\newcommand{\idl}[1]{\mathfrak{#1}}

\newcommand{\X}{\mathord{-}}
\newcommand{\unit}{\mathbf{1}}

\title{Berkovich \texorpdfstring{\(2\)}{2}-motives and normed ring stacks}
\author{Ko Aoki}
\address{Max Planck Institute for Mathematics,
  Vivatsgasse 7, 53111 Bonn, Germany
}
\email{aoki@mpim-bonn.mpg.de}
\date{\today}

\begin{document}

\begin{abstract}
  The de~Rham stack construction of Simpson shows that
  D-modules are quasicoherent sheaves on a modified geometry.
  Drinfeld furthermore
  introduced the ring stack perspective (aka transmutation),
  which asserts that a coefficient theory
  is determined by a ring stack.
  Scholze proposed relating this idea
  to motivic realizations
  using \((\infty,2)\)-categorical language.

  In this work,
  we formulate and prove a precise version of this principle:
  The presentable category of kernels of motivic homotopy theory
  is the linearly symmetric monoidal \((\infty,2)\)-category
  that is freely generated
  by a homologically trivial smooth sutured ring stack.
  We also prove the étale version of this statement,
  reducing étale descent to the Kummer and Artin–Schreier conditions.
  Lastly,
  we prove an analytic version
  connecting Scholze’s Berkovich motives and ring stacks with an absolute value.
  This is useful to construct
  motivic realization functors in analytic geometry,
  such as the Habiro and Hyodo–Kato realizations.
\end{abstract}

\maketitle
\setcounter{tocdepth}{1}
\tableofcontents

\section{Introduction}\label{s:intro}

In this paper,
an \(n\)-category
means an \((\infty,n)\)-category.

This paper is devoted
to the study of \(2\)-motives.
Most importantly,
we prove universal properties.
We explain the ideas motivating this concept in \cref{ss:id},
state the main results in \cref{ss:intro},
and discuss applications in \cref{ss:app}.

\subsection{Ideas}\label{ss:id}

In arithmetic geometry,
varieties are studied using various cohomology theories.
For a variety~\(X\)
over a field~\(k\),
one might study
its Betti cohomology,
de Rham cohomology,
étale cohomology,
crystalline cohomology,
etc.
These theories are interconnected by comparison isomorphisms,
yet they are constructed in entirely different ways.
To govern this complexity, Grothendieck envisioned the concept of \emph{motives},
positing the existence of a universal, “absolute” cohomology theory.
This would take the form of an abelian category of motives~\(\Mot(k)\)\footnote{
  In this paper, the notation \(\Mot\)
  (or \(\Mot[2]\))
  does not have a global definition
  and simply denotes some version of (\(2\)-)motives.
},
such that any variety \(X\) has a corresponding motive~\([X]\).
Each known cohomology theory \(H^{*}\) would then arise
as a \emph{realization functor} from \(\Mot(k)\) to the category of vector spaces.
This concept sought to capture the cohomological essence of a variety,
an object that would explain all its different cohomological “manifestations.”
While Grothendieck’s original vision of \(\Mot(k)\) has remained conjectural,
it was later substantially realized by Voevodsky~\cite{Voevodsky00},
whose construction of the triangulated category of motives \(\operatorname{DM}(k)\)
led to his celebrated proof of the Milnor conjecture.
More relevant
to the present work
is the stable motivic homotopy theory~\(\SH(k)\) 
of Morel–Voevodsky~\cite{MorelVoevodsky99}.

In parallel, Grothendieck, with key contributions from Verdier,
considered six operations.
This is the essential framework
for working with these cohomology theories.
It describes the fundamental structural properties
such as functoriality, Poincaré duality, and the Künneth formula,
that a well-behaved “coefficient theory” (like the category of étale sheaves or D-modules)
must have.
Nowadays,
these six operations
can be packaged as a functor
from spans (aka correspondences).
This perspective was famously explained in
a volume by
Gaitsgory–Rozenblyum~\cite{GaitsgoryRozenblyum17},
where this idea of using spans was attributed to Lurie.

This brings us to the modern context.
If Grothendieck’s motives classify
cohomology theories,
it is natural to consider classifying coefficient theories instead.
Drew–Gallauer~\cite{DrewGallauer22} showed 
that \(\SH\) is a universal six-functor formalism.
Scholze~\cite{ScholzeM}
instead considered this via categorification:
There should be a \(2\)-category
of \(2\)-motives that is a universal recipient
of six-functor formalisms.
In this framework, the \(2\)-motive~\([X]\)
captures the sheaf-theoretic essence of~\(X\).
The category of \(2\)-motives
is thus a \(2\)-category in which varieties determine objects
and operations on sheaves determine morphisms.
This unified picture, however, raises a technical
problem regarding the precise mathematical nature of this \(2\)-category.
It should look like the category of linear categories.
Stefanich~\cite{StefanichP} was the first to provide a foundation
for such a theory.
Our previous work~\cite{n-pr} showed how to achieve this without reference to a universe.
The present work builds directly on that foundation.
We see a precise formulation
of this idea in \cref{main_1}.

Even though we have a good packaging of coefficient theories,
the challenge
of systematically constructing these coefficient theories still remains.
A breakthrough came from the study of de~Rham cohomology.
The natural coefficient theory for de~Rham cohomology is the theory of D-modules
due to Sato.
In a crucial insight,
Simpson~\cite{Simpson96} realized that D-modules on a variety~\(X\)
over~\(\QQ\)
could be reinterpreted as certain quasicoherent sheaves on a new geometric object
called the \emph{de~Rham stack}~\(X^{\dR}\).
This was the first indication that
various coefficient categories
are simply quasicoherent sheaves on some modified geometry associated to~\(X\).
This idea was further refined by Drinfeld~\cite{Drinfeld22}.
What he emphasized was the ring stack perspective,
which is now called \emph{transmutation},
a term coined by Bhatt (cf.~\cite[Remark~2.3.8]{FGauges}).
The idea is that an entire coefficient theory can be generated
from a single piece of data;
e.g., for D-modules, \((\AA^{1})^{\dR}\),
viewed as a ring stack.
By specifying what \(\AA^{1}\) maps to,
one can “transmute” the base geometry of varieties into a new geometry,
and the quasicoherent sheaves on that new geometry give you the coefficient theory.

Given this,
Scholze~\cite{ScholzeM}
further proposed
that \(\Mot[2](\ZZ)\)
should classify ring stacks
satisfying certain conditions.
In this paper,
we prove a precise version of this as~\cref{main_2}.
In terms of
the gestalt theory of
Scholze–Stefanich~\cite{Gestalten},
we can say that
it provides a moduli description
of the gestalt associated to stable motivic homotopy theory.

Furthermore,
we prove certain variants of this theorem.
Notably,
Scholze~\cite{ScholzeB}
introduced
the category of \emph{Berkovich motives}
with applications to local Langlands in mind.
It is a certain analytic version of motivic homotopy theory.
He constructed six operations for them
and proved several desirable properties.
One drawback of this theory
is that it is difficult to construct realization functors,
since the definition already involves arc~descent.
In \cref{main_4},
we characterize
this using ring stacks with an absolute value.
This facilitates the construction of realization functors
by first establishing \(2\)-categorical realizations.

\subsection{Results}\label{ss:intro}

Let \(\Cat{QProj}\)
denote the category of quasiprojective static schemes
over~\(\ZZ\)
(see \cref{xddj7g} for this choice).
We take \(\SH\) to be the Morel–Voevodsky stable motivic homotopy theory
(see \cref{ss:sh}).
This admits six operations due to Ayoub~\cite{Ayoub07a,Ayoub07b}.
Applying the machinery of Liu–Zheng~\cite{LiuZheng},
we obtain
a lax symmetric monoidal functor
\({\SH}\colon\Span(\Cat{QProj})\to\Cat{Pr}_{\st}\).

\begin{definition}\label{xagyaj}
  Let \(S\) be a divisorial noetherian static scheme\footnote{This assumption is included solely to
    ensure a unique notion of quasiprojectivity.
  }.
  We write \(\SH[2](S)\)
  for the presentably symmetric monoidal \(2\)-category of kernels
  (see \cref{ss:ker})
  of
  \({\SH}\colon\Span(\Cat{QProj}_{S})\to\Cat{Pr}_{\st}\).
  Informally,
  it is freely generated
  by the image of a symmetric monoidal functor
  \([\X]\colon\Cat{QProj}_{S}^{\op}\to\SH[2](S)\),
  with the image being described as follows:
  \begin{itemize}
    \item
      The mapping category from~\([X]\) to~\([Y]\)
      is \(\SH(X\times Y)\).
    \item
      The identity \(\id_{[X]}\) is \(d_{!}\unit\in\SH(X\times X)\)
      (which is equivalent to~\(d_{*}\unit\) in this case),
      where \(d\colon X\to X\times X\) is the diagonal.
    \item
      For \(M\colon[X]\to[Y]\)
      and \(N\colon[Y]\to[Z]\),
      its composite is
      \((\pr_{X,Z})_{!}(\pr_{X,Y}^{*}M\otimes\pr_{Y,Z}^{*}N)\),
      where \(\pr\) denotes the corresponding projection
      from \(X\times Y\times Z\).
  \end{itemize}
\end{definition}

First,
we characterize
this via
its universal property:

\begin{Theorem}\label{main_1}
  The functor
  \([\X]\colon(\Cat{QProj}_{S})^{\op}\to\SH[2](S)\)
  is universal among
  symmetric monoidal functors
  to stable presentably symmetric monoidal \(2\)-categories
  satisfying the following axioms:
  \begin{conenum}
    \item\label{i:a_sm}
      Consider a smooth morphism \(f\colon Y\to X\).
      Then \(f^{*}\colon[X]\to[Y]\) admits 
      a left adjoint~\(f_{\natural}\).
      Moreover,
      the Beck–Chevalley transformation
      \(f'_{\natural}q^{*}\to p^{*}f_{\natural}\)
      is an equivalence
      for any pullback square
      \begin{equation}
        \label{e:iljxp}
        \begin{tikzcd}
          Y'\ar[r,"q"]\ar[d,"f'"']&
          Y\ar[d,"f"]\\
          X'\ar[r,"p"]&
          X\rlap.
        \end{tikzcd}
      \end{equation}
    \item\label{i:a_exc}
      For a closed subvariety \(i\colon Z\hookrightarrow X\)
      with complement~\(U\),
      the diagram
      \([U]\gets[X]\to[Z]\) is a recollement
      (see \cref{xpyo8y}).
      In particular,
      \(i^{*}\colon[X]\to[Z]\) admits a right adjoint~\(i_{*}\).
    \item\label{i:a_hi}
      For the projection \(f\colon\AA^{1}_{X}\to X\),
      the counit \(f_{\natural}f^{*}\to{\id}\)
      is an equivalence.
    \item\label{i:a_tate}
      For the projection \(f\colon\AA^{1}_{X}\to X\)
      and its zero section \(s\colon X\to\AA^{1}_{X}\),
      the morphism
      \(f_{\natural}s_{*}\)
      is an autoequivalence of~\([X]\).
  \end{conenum}
\end{Theorem}

\begin{remark}\label{xddj7g}
  Our choice of \(\Cat{QProj}_{S}\)
  in \cref{main_1}
  is not important;
  we can consider,
  e.g.,
  the category of static schemes separated of finite type over~\(S\)
  to obtain the same result.
  One advantage of this \(2\)-categorical packaging
  is that
  we do not need to specify
  which category of geometries
  to work with.
\end{remark}

We now specialize to the case \(S=\Spec\ZZ\).
The following is the characterization via
the language of ring stacks:

\begin{Theorem}\label{main_2}
  The stable presentably symmetric monoidal \(2\)-category
  \(\SH[2](\ZZ)\)
  is freely generated
by a homologically trivial smooth sutured ring stack.
\end{Theorem}

\begin{remark}\label{xawvvy}
  By \cref{st_sm},
  we can equivalently state \cref{main_2}
  as saying that
  \(\SH[2](\ZZ)\)
  is freely generated
  by a stable homologically trivial weakly suave ring stack.
  Now,
  “stable,”
  “sutured,”
  “homologically trivial,”
  and “weakly suave”
  correspond to \cref{i:a_tate,i:a_exc,i:a_hi,i:a_sm}
  in \cref{main_1},
  respectively.
  One notable point is that
  all these conditions are only about~\(\AA^{1}\).
  In particular,
  in \cref{main_2},
  we do not mention
  general schemes at all.
\end{remark}

We then consider its étale variant,
where \(\SH[2]_{\et}\)
denotes the presentable \(2\)-category of kernels
for \(\SH_{\et}\),
the étale sheafified version of \(\SH\) (see \cref{ss:shet}):

\begin{Theorem}\label{main_3}
  The stable presentably symmetric monoidal \(2\)-category
  \(\SH[2]_{\et}(\ZZ)\)
  is freely generated
  by a homologically trivial smooth sutured ring stack
  that is Kummer and Artin–Schreier
  in the following sense:
  \begin{description}
    \item[Kummer]
      The morphism
      \(R[\sfrac1l]^{\times}\to R[\sfrac1l]^{\times}\)
      given by \(x\mapsto x^{l}\)
      is a cover
      for any prime~\(l\).
    \item[Artin–Schreier]
      The morphism
      \(R/p\to R/p\)
      given by \(x\mapsto x^{p}-x\)
      is a cover
      for any prime~\(p\).
  \end{description}
\end{Theorem}

In the statement above,
the notion of \emph{cover}
is important,
and we study it extensively
in \cref{s:des}
under the name of \emph{descent}.

\begin{remark}\label{xidgyh}
  \Cref{main_3}
  shows that there is an idempotent algebra~\(E\)
  in \(\SH[2](\ZZ)\)
  such that
  \(\Mod_{E}(\SH[2](\ZZ))\) is equivalent to \(\SH[2]_{\et}(\ZZ)\).
\end{remark}

We then move on to the analytic situation.
Scholze~\cite{ScholzeB} introduced
Berkovich motives with integral coefficients \(\D_{\mot}(\X;\ZZ)\)
and constructed six operations for them.
We consider
Berkovich motives with spherical coefficients
\(\D_{\mot}(\X;\SS)\);
see \cref{ss:d_mot}.
Here
we consider
uniform Banach rings topologically of finite presentation
as the class of “varieties”
to obtain the presentable \(2\)-category
of kernels,
for which we write \(\D[2]_{\mot}(\ZZ;\SS)\).
We then characterize this as follows:

\begin{Theorem}\label{main_4}
  The stable presentably symmetric monoidal \(2\)-category
  \(\D[2]_{\mot}(\ZZ;\SS)\)
  is freely generated by
  a Kummer–Artin–Schreier homologically trivial smooth ring stack
  with an absolute value
  whose open unit disk is homologically trivial.
\end{Theorem}

\begin{remark}\label{xb4idi}
  Note that \(\D[2]_{\mot}(\ZZ;\SS)\)
  admits a class for any seminormed ring.
  In this situation,
  these classes are proper,
  as shown in \cref{x2a9p8}.
  This means that,
  unlike Scholze’s Berkovich motives,
  proper base change holds
  without any finiteness assumption.
\end{remark}

\subsection{Applications}\label{ss:app}

An immediate application is
the construction of motivic realization functors.
To do this,
we first consider the \(2\)-categorical realization
using our universality results.
Then, by taking \(\End(\unit)\),
we obtain the desired realization \(1\)-functor.
The following is an example of this:

\begin{theorem}\label{x847gj}
  The analytic Habiro stack of Scholze~\cite{ScholzeH}
  and
  the Hyodo–Kato stack
  of Anschütz–Bosco–Le~Bras–Rodríguez~Camargo–Scholze~\cite{ABLBRCS}
  determine realization functors.
\end{theorem}

Another application
is about the nature of these \(2\)-categories.
For example, we see the following:

\begin{theorem}\label{xa0wo6}
  Let \(\Mot[2](\ZZ)\)
  be any of \(\SH[2](\ZZ)\), \(\SH[2]_{\et}(\ZZ)\),
  or \(\D[2]_{\mot}(\ZZ;\SS)\).
  Then it is a \(1\)-truncated object in \(\CAlg(2\Cat{Pr})^{\op}\),
  i.e.,
  the tautological morphism
  \begin{equation*}
    \Mot[2](\ZZ)^{\otimes S^{2}}\to\Mot[2](\ZZ)
  \end{equation*}
  is an equivalence.
\end{theorem}

\begin{proof}
  For an object \(\cat{C}\in\CAlg(2\Cat{Pr})\),
  we have
  \begin{equation*}
    \Map_{\CAlg(2\Cat{Pr})}\bigl(\Mot[2](\ZZ)^{\otimes S^{2}},\cat{C}\bigr)
    \simeq
    \Map\bigl(S^{2},\Map_{\CAlg(2\Cat{Pr})}(\Mot[2](\ZZ),\cat{C})\bigr).
  \end{equation*}
  Therefore,
  it suffices to show that
  \(\Map_{\CAlg(2\Cat{Pr})}(\Mot[2](\ZZ),\cat{C})\)
  is \(1\)-truncated.
  To show this, it suffices to show that
  every object is \(0\)-truncated,
  i.e., that
  the diagonal morphism
  \begin{equation*}
    R\to\Map(S^{1},R)
  \end{equation*}
  is an equivalence
  for a ring stack~\(R\)
  satisfying the conditions of \cref{main_2}.
  This follows from \cref{xc95y9}.
\end{proof}

\begin{remark}\label{x152wm}
\Cref{xa0wo6} for \(\Mot[2](\ZZ)=\SH[2](\ZZ)\)
  can also be deduced
  directly from \cref{main_1}.
\end{remark}

This approach
enables us to show certain properties
without examining the \(2\)-category directly.
The proof of the following will appear in~\cite{n-rig-1},
where we will introduce the notion of \(n\)-rigidity:

\begin{theorem}[\cite{n-rig-1}]\label{xrxwab}
  The stable presentably symmetric monoidal
  \(2\)-categories
  \(\D[2]_{\mot}(\ZZ;\SS[\sfrac12])\)
  and
  \(\D[2]_{\mot}(O;\SS)\)
  for any totally imaginary number ring~\(O\)
  are \(2\)-rigid.
\end{theorem}

In the context of the geometry
of gestalts by Scholze–Stefanich (cf.~\cite{Gestalten}),
which treats presentable categorical spectra (see~\cite[Remark~2.10]{n-pr})
as a building block of geometry,
this means that
the gestalt of Berkovich motives
is almost proper.

\subsection*{Outline}\label{ss:ol}

We begin in \cref{s:idem}
by reviewing some basic tools in higher category theory,
which are essential for the subsequent developments.
In \cref{s:dgs}, we use these tools to prove \cref{main_1}.  
Next, we introduce some general machinery.
\Cref{s:stk} covers basic notions of stacks,
indispensable for stating our main theorems.
\Cref{s:des} develops the theory of descent,
a fundamental tool for understanding presentably symmetric monoidal \(n\)-categories.
We illustrate these concepts with examples in \cref{s:d_ex}.  
With this preparation, we turn to the proofs of the main theorems.
In \cref{s:main_alg,s:main_et}, we prove \cref{main_2,main_3},
respectively.
Finally, we consider the analytic setting.
\Cref{s:ber} studies Berkovich geometry,
which provides the framework for the results of \cref{s:main_ber},
where we prove \cref{main_4}.

\subsection*{Acknowledgments}\label{ss:ack}

I thank Peter Scholze
for helpful discussions
from the very early stages of this project.
I thank
Adam Dauser,
Hadrian Heine,
Juan~Esteban Rodríguez~Camargo,
and
Germán Stefanich for useful conversations while writing this paper.
I thank Peter Scholze for valuable comments on a draft of this paper.
I thank the Max Planck Institute for Mathematics.

\section{Some \texorpdfstring{\(2\)}{2}-category theory}\label{s:idem}

We recall some \(2\)-category theory
needed in this paper.
In \cref{ss:corr},
we review  spans.
In \cref{ss:ker},
we review \(2\)-categories of kernels.
In \cref{ss:bc_hex},
we commute a hexagon
consisting of Beck–Chevalley transformations.
In \cref{ss:exc},
we do some \(2\)-categorical homological algebra.

We assume familiarity
with the basics of presentable \(n\)-category theory;
see~\cite{StefanichP} or~\cite{n-pr}.
What is crucial in this paper is that,
as demonstrated in~\cite[Section~3]{n-pr},
presentable \(2\)-category theory
behaves well.

\subsection{Spans}\label{ss:corr}

For a \(1\)-category~\(\cat{G}\) with finite limits,
we obtain
the \(n\)-category of iterated spans \(\Span[n](\cat{G})\) of Haugseng~\cite{Haugseng18}
(see also~\cite[Section~1.5]{Haugseng18} for prior works).
In this paper,
the variants
\(\Span_{E}(\cat{G})\)
and
\(\Span[2]_{E;P,J}(\cat{G})\) play important roles,
where \(\cat{G}\) is a \(1\)-category
and
\(J\), \(P\), and~\(E\) are wide subcategories of~\(\cat{G}\)
stable under base change
satisfying \(J\subset E\supset P\).
Concretely,
\(\Span[2]_{E;P,J}(\cat{G})\) is described as follows;
see~\cite[Construction~4.12]{CnossenLenzLinskens} for the precise construction:
\begin{itemize}
  \item
    Objects are those of~\(\cat{G}\).
    \item
    A \(1\)-morphism from~\(X\) to~\(X'\)
    is a span \(X\gets Y\to X'\) in~\(\cat{G}\)
    such that \(Y\to X'\)
    is in~\(E\).
    \item
    A \(2\)-morphism from~\(X\gets Y\to X'\)
    to \(X\gets Y'\to X'\)
    is a diagram
    \begin{equation}
      \label{e:g9up8}
      \begin{tikzcd}
        {}&
        Y\ar[rd]\ar[ld]&
        {}\\
        X&
        Z\ar[r]\ar[l]\ar[u]\ar[d]&
        X'\\
        {}&
        Y'\ar[ur]\ar[ul]&
        {}
      \end{tikzcd}
    \end{equation}
    in~\(\cat{G}\)
    such that \(Z\to Y\) is in~\(P\) and \(Z\to Y'\) is in~\(J\).
\end{itemize}
By setting~\(J\) and~\(P\) to be the class of all equivalences,
we obtain~\(\Span_{E}(\cat{G})=\Span[2]_{E;\all,\all}(\cat{G})\),
which is a \(1\)-category.
Note the inclusions
\(\cat{G}^{\op}\to\Span_{E}(\cat{G})\to\Span[2]_{E;P,J}(\cat{G})\).
In~\cite{CnossenLenzLinskens},
the universality of
this category was obtained under suitable assumptions.
Here we recall from~\cite[Theorem~B]{CnossenLenzLinskens}
the version
that considers the cartesian symmetric monoidal structure on~\(\cat{G}\):

\begin{theorem}[Cnossen–Lenz–Linskens]\label{cll}
  Let \(\cat{G}\) be a \(1\)-category with finite products.
  Consider its subcategories~\(J\), \(P\), and~\(E\)
  satisfying the following:
  \begin{itemize}
    \item
      Both~\(J\) and~\(P\) are closed under base change
      and satisfy cancellation.
    \item
      A morphism is in~\(E\)
      if and only if it can be written as \(pj\)
      for \(j\in J\) and \(p\in P\).
    \item
      Every morphism in \(J\cap P\)
      is \(n\)-truncated for some~\(n\) (depending on the morphism).
  \end{itemize}
  Then, for any symmetric monoidal \(2\)-category~\(\cat{C}\),
  the restriction functor
  \begin{equation*}
    \Fun^{\lsm}(\Span[2]_{E;P,J}(\cat{G}),\cat{C})
    \to
    \Fun^{\lsm}(\cat{G}^{\op},\cat{C})
  \end{equation*}
  is a subcategory inclusion.
  \begin{itemize}
    \item
      Objects are spanned by \((J,P)\)-biadjointable functors~\(D\)
      satisfying the projection formula:
      The morphism \(j^{*}=D(j)\) for any \(j\in J\)
      admits a left adjoint~\(j_{\natural}\)
      satisfying the projection formula and the Beck–Chevalley condition.
      Similarly, \(p^{*}=D(p)\) for any \(p\in P\)
      admits a right adjoint~\(p_{*}\) satisfying the dual conditions.
      Moreover,
      these adjoints satisfy the double Beck–Chevalley condition;
      i.e., for any pullback square
      \begin{equation*}
        \begin{tikzcd}
          Y'\ar[r,"k"]\ar[d,"p'"']&
          Y\ar[d,"p"]\\
          X'\ar[r,"j"]&
          X
        \end{tikzcd}
      \end{equation*}
      such that \(j\in J\) and \(p\in P\),
      the induced morphism \(j_{\natural}p'_{*}\to p_{*}k_{\natural}\)
      is an equivalence.
    \item
      Morphisms \(\alpha\colon D\to D'\)
      are natural transformations
      compatible with the formations of~\(j_{\natural}\) and~\(p_{*}\)
      above:
      For \(j\colon Y\to X\) in~\(J\),
      the morphism \(D'(j)^{\L}\alpha_{X}\to\alpha_{Y}D(j)^{\L}\)
      is an equivalence
      and for \(p\colon Y\to X\) in~\(P\),
      the morphism \(\alpha_{X}D(p)^{\R}\to D'(p)^{\R}\alpha_{Y}\)
      is an equivalence.
  \end{itemize}
  Here, \(\Fun^{\lsm}\)
  denotes the category of lax symmetric monoidal functors.
\end{theorem}

\begin{remark}\label{xw0us7}
  We use \cref{cll}
  throughout this paper,
  even where weaker results might suffice.
  We refer the reader to~\cite[Section~1]{CnossenLenzLinskens}
  for historical accounts
  of other previous approaches.
  We can often compare this with
  a previous construction,
  such as in Liu–Zheng’s work~\cite{LiuZheng},
  by using
  the uniqueness theorem of Dauser–Kuijper~\cite{DauserKuijper}.
\end{remark}

\subsection{Kernels}\label{ss:ker}

Consider a category with finite limits~\(\cat{G}\)
and a lax symmetric monoidal functor
\(D\colon\Span(\cat{G})\to\Cat{Cat}\).
The symmetric monoidal \(2\)-category we associate to~\(D\)
was used in~\cite[Section~2.3]{FarguesScholze}
based on\footnote{See~\cite[Section~3]{HansenScholze23}
  and~\cite[Remark~1.3.11]{HeyerMann}
  about the relation.
} Lu–Zheng’s work~\cite{LuZheng22}.
It has the following features:
\begin{itemize}
  \item
    Objects are those of~\(\cat{G}\).
    We write \([X]\) for the object corresponding to~\(X\in\cat{G}\).
  \item
    The mapping category from~\([X]\) to~\([Y]\)
    is \(D(X\times Y)\).
  \item
    The identity is \(d_{!}\unit\),
    where \(d\colon X\to X\times X\) is the diagonal.
  \item
    For \(F\colon[X]\to[Y]\)
    and \(G\colon[Y]\to[Z]\),
    its composite is
    \((\pr_{X,Z})_{!}(\pr_{X,Y}^{*}F\otimes\pr_{Y,Z}^{*}G)\),
    where \(\pr\) denotes the corresponding projection
    from \(X\times Y\times Z\).
\end{itemize}
The following precise formulation
was learned from Dauser:

\begin{definition}\label{ker_s}
  Let \(\cat{G}\) be a category with finite limits
  and \(D\colon\Span(\cat{G})\to\Cat{Cat}\)
  a lax symmetric monoidal functor.
  Then,
  since \(\Span(\cat{G})\) is closed,
  it is enriched over itself.
  Base changing along~\(D\),
  we obtain a \(\Cat{Cat}\)-enrichment on~\(\Span(\cat{G})\).
  We call this the \emph{\(2\)-category of kernels}.
\end{definition}

In the presentable situation,
we can make another definition
of a similar category.
The following definition was learned from Dauser:

\begin{definition}\label{ker_p}
  Let \(\cat{G}\) be a category with finite limits
  and \(D\colon\Span(\cat{G})\to\Cat{Pr}\)
  a lax symmetric monoidal functor,
  which we regard
  as a commutative algebra object
  of \(\Fun(\Span(\cat{G}),\Cat{Pr})\)
  with the Day convolution symmetric monoidal structure.
  We consider
  \(\Mod_{D}(\Fun(\Span(\cat{G}),\Cat{Pr}))\),
  which we refer to as
  the \emph{presentable \(2\)-category of kernels}.
\end{definition}

An advantage of this definition is that
it avoids the use of enriched category theory.
Nevertheless,
these two constructions can be compared by the following:

\begin{theorem}\label{xchyfq}
  Let \(\cat{S}\) be a small closed symmetric monoidal category
  and \(\cat{V}\) an object of \(\CAlg(\Cat{Pr})\).
  Let \(D\colon\cat{S}\to\cat{V}\) be a lax symmetric monoidal functor.
  Then there is a canonical equivalence
  \begin{equation*}
    \PShv^{\cat{V}}(\widetilde{\cat{S}}^{\op})
    \simeq
    \Mod_{D}(\Fun(\cat{S},\cat{V})),
  \end{equation*}
  where \(\widetilde{\cat{S}}\) is the \(\cat{V}\)-enriched category
  obtained by base changing the self-enrichment of~\(\cat{S}\)
  along~\(D\).
\end{theorem}

\begin{corollary}\label{xfa81g}
  In the situation of \cref{ker_p},
  we choose a regular cardinal~\(\kappa\)
  so that \(D\) factors through \(\Cat{Pr}^{\kappa}\).
  By the procedure in \cref{ker_s},
  we obtain \(\Cat{Pr}^{\kappa}\)-enriched category~\(\cat{K}\).
  Then \(\PShv^{\Cat{Pr}^{\kappa}}(\cat{K})\)
  is equivalent to \(\Mod_{D}(\Fun(\Span(\cat{G}),\Cat{Pr}^{\kappa}))\).
\end{corollary}

Heine helped with the following proof:

\begin{proof}[Proof of \cref{xchyfq}]
  We write \(U\colon\Cat{Ani}\to\cat{V}\) for the unit functor.
  Since \(U(\Map_{\cat{S}}(\unit,\X))\)
  is the unit in \(\CAlg(\Fun(\cat{S},\cat{V}))\),
  we obtain a morphism \(U(\Map_{\cat{S}}(\unit,\X))\to D\) there.
  By base changing the self-enrichment of~\(\cat{S}\) along this morphism,
  we obtain
  a \(\cat{V}\)-functor
  \(\cat{S}\to\widetilde{\cat{S}}\),
  where the source is enriched trivially.
  We write \(f\) for its opposite.
  This induces an adjunction \(f_{!}\dashv f^{*}\)
  in \(\Mod_{\cat{V}}(\Cat{Pr})\).
  Note that
  \(f^{*}\colon\PShv^{\cat{V}}(\widetilde{\cat{S}}^{\op})
  \to\PShv^{\cat{V}}(\cat{S}^{\op})=\Fun(\cat{S},\cat{V})\) is monadic.
  It suffices to show that this \(\cat{V}\)-linear monad~\(f^{*}f_{!}\)
  is identified with~\(D\).
  First,
  by definition,
  \(f^{*}\unit\) is identified with~\(D\),
  from which we obtain a morphism of monads \(D\to f^{*}f_{!}\).
By linearity,
  it suffices to verify this equivalence on representables,
  which is straightforward.
\end{proof}

\subsection{On double Beck–Chevalley transformations}\label{ss:bc_hex}

We prove the following
about cubes in a \(2\)-category:

\begin{proposition}\label{hexagon}
  Let \(\cat{G}=[1]\times[1]\) be a square category
  labeled as in \cref{e:iljxp}
  and \(\cat{C}\) be a \(2\)-category.
  We have a functor \(D\colon\cat{G}^{\op}\to\cat{C}\)
  such that the square
  \begin{equation*}
    \begin{tikzcd}
      D(X)\ar[r,"D(p)"]\ar[d,"D(f)"']&
      D(X')\ar[d,"D(f')"]\\
      D(Y)\ar[r,"D(q)"]&
      D(Y')
    \end{tikzcd}
  \end{equation*}
  is vertically left adjointable
  and horizontally right adjointable.
  We have another diagram \(D'\) satisfying the same condition
  and a natural transformation \(\alpha\colon D\to D'\).
  Then the diagram
  \begin{equation*}
    \begin{tikzcd}[column sep=small]
      {}&
      \alpha_{X}D(f)^{\L}D(q)^{\R}\ar[r]&
      \alpha_{X}D(p)^{\R}D(f')^{\L}\ar[dr]&
      {}\\
      D'(f)^{\L}\alpha_{Y}D(q)^{\R}\ar[ur]\ar[dr]&
      {}&
      {}&
      D'(p)^{\R}\alpha_{X'}D(f')^{\L}\\
      {}&
      D'(f)^{\L}D'(q)^{\R}\alpha_{Y'}\ar[r]&
      D'(p)^{\R}D'(f')^{\L}\alpha_{Y'}\ar[ur]&
      {}
    \end{tikzcd}
  \end{equation*}
  consisting of
  four Beck–Chevalley
  and two double Beck–Chevalley transformations
  commutes.
\end{proposition}

We break the symmetry in the following proof;
e.g., it is unclear whether our chosen witness coincides
with the one obtained by breaking the symmetry
in the other way:

\begin{proof}
By pasting Beck–Chevalley transformations,
  we reduce the diagram to
  \begin{equation*}
    \begin{tikzcd}
      D'(f)^{\L}\alpha_{Y}D(q)^{\R}\ar[r]\ar[d]&
      \alpha_{X}D(p)^{\R}D(f')^{\L}\ar[d]\\
      D'(f)^{\L}D'(q)^{\R}\alpha_{Y'}\ar[r]&
      D'(p)^{\R}\alpha_{X'}D(f')^{\L}\rlap,
    \end{tikzcd}
  \end{equation*}
  which consists of four Beck–Chevalley transformations.
We then unpack the rows;
  we are reduced to commuting the right square of
  \begin{equation*}
    \begin{tikzcd}
      \alpha_{Y}D(q)^{\R}\ar[r]\ar[d]&
      \alpha_{Y}D(q)^{\R}D(f')D(f')^{\L}\ar[r,"\simeq"]\ar[d]&
      D'(f)\alpha_{X}D(p)^{\R}D(f')^{\L}\ar[d]\\
      D'(q)^{\R}\alpha_{Y'}\ar[r]&
      D'(q)^{\R}\alpha_{Y'}D(f')D(f')^{\L}\ar[r,"\simeq"]&
      D'(f)D'(p)^{\R}\alpha_{X'}D(f')^{\L}\rlap.
    \end{tikzcd}
  \end{equation*}
  By removing \(D(f')^{\L}\), now we have to commute
  \begin{equation*}
    \begin{tikzcd}
      \alpha_{Y}D(q)^{\R}D(f')\ar[d]&
      \alpha_{Y}D(f)D(p)^{\R}\ar[r,"\simeq"]\ar[l,"\simeq"']&
      D'(f)\alpha_{X}D(p)^{\R}\ar[d]\\
      D'(q)^{\R}\alpha_{Y'}D(f')\ar[r,"\simeq"]&
      D'(q)^{\R}D'(f')\alpha_{X'}&
      D'(f)D'(p)^{\R}\alpha_{X'}\ar[l,"\simeq"']\rlap.
    \end{tikzcd}
  \end{equation*}
By pasting Beck–Chevalley transformations again,
  this simplifies to
  \begin{equation*}
    \begin{tikzcd}
      \alpha_{Y}D(f)D(p)^{\R}\ar[r,"\simeq"]\ar[d]&
      D'(f)\alpha_{X}D(p)^{\R}\ar[d]\\
      D'(q)^{\R}\alpha_{Y}D(f')\ar[r,"\simeq"]&
      D'(q)^{\R}D'(f')\alpha_{X'}\rlap,
    \end{tikzcd}
  \end{equation*}
  which commutes.
\end{proof}

\subsection{Recollements and excision squares}\label{ss:exc}

Here we perform some \(2\)-categorical homological algebra.

\begin{definition}\label{xpyo8y}
  We call
  a diagram
  \begin{equation*}
    \begin{tikzcd}
      C_{U}&
      C_{X}\ar[l,"j^{*}"']\ar[r,"i^{*}"]&
      C_{Z}
    \end{tikzcd}
  \end{equation*}
  in a stable presentable \(2\)-category
  a \emph{recollement}
  if the following conditions are satisfied:
  \begin{conenum}
    \item\label{i:u0xy6}
      The morphism~\(j^{*}\) admits a left adjoint~\(j_{!}\),
      and
      the counit \(j_{!}j^{*}\to{\id}\) is an equivalence.
    \item\label{i:6y2xc}
      The morphism~\(i^{*}\) admits a right adjoint~\(i_{*}\),
      and the unit \({\id}\to i_{*}i^{*}\) is an equivalence.
    \item\label{i:m2bvr}
      The composite
      \(j^{*}i_{*}\) is zero\footnote{Note that this is a condition, not a datum.
      } (or equivalently, \(i^{*}j_{!}\) is zero)
      and the induced square
      \begin{equation*}
        \begin{tikzcd}
          j_{!}j^{*}\ar[r]\ar[d]&
          {\id}\ar[d]\\
          0\ar[r]&
          i_{*}i^{*}
        \end{tikzcd}
      \end{equation*}
      in \(\End(C_{X})\) is a pushout.
  \end{conenum}
\end{definition}

We perform some \(2\)-categorical homological algebra.

\begin{proposition}\label{xwog77}
In a stable presentable \(2\)-category,
  we consider a commutative diagram
  \begin{equation*}
    \begin{tikzcd}
      C_{U}\ar[d,"g^{*}"']&
      C_{X}\ar[l,"j^{*}"']\ar[r,"i^{*}"]\ar[d,"f^{*}"]&
      C_{Z}\ar[d,equal]\\
      C_{V}&
      C_{Y}\ar[l,"l^{*}"']\ar[r,"k^{*}"]&
      C_{Z}
    \end{tikzcd}
  \end{equation*}
  satisfying the following conditions:
  \begin{conenum}
    \item\label{i:rec_x}
      The top row is a recollement.
    \item\label{i:rec_y}
      The bottom row is a recollement.
    \item\label{i:ladj_i}
      The right square is vertically left adjointable;
      i.e., \(f^{*}\) admits a left adjoint~\(f_{\natural}\),
      and the Beck–Chevalley transformation
      \(k^{*}\to i^{*}f_{\natural}\) is an equivalence.
  \end{conenum}
  Then the left square is cartesian.
\end{proposition}

\begin{proof}
First,
  \(g_{\natural}=j^{*}f_{\natural}l_{!}\)
  determines a left adjoint of~\(g^{*}\).
  With this definition,
  the left square is also vertically left adjointable.

  We apply \(\Hom(D,\X)\)
  for an arbitrary object~\(D\)
  to reduce this to the case of \(\Cat{Pr}_{\st}\).
  We write \(\Hom(D,C_{\X})\) as~\(\cat{C}_{\X}\).
  We write \(\Map_{\X}\)
  for \(\Map_{\cat{C}_{\X}}\).

  We first prove that
  the induced functor from~\(\cat{C}_{X}\)
  to the pullback is fully faithful.
  This is equivalent to the square
  \begin{equation*}
    \begin{tikzcd}
      \Map_{U}(j^{*}C,j^{*}D)\ar[d]&
      \Map_{X}(C,D)\ar[l]\ar[d]\\
      \Map_{V}(l^{*}f^{*}C,l^{*}f^{*}D)&
      \Map_{Y}(f^{*}C,f^{*}D)\ar[l]
    \end{tikzcd}
  \end{equation*}
  being a pullback
  for objects~\(C\) and~\(D\) of~\(\cat{C}_{X}\).
  This is equivalent to the assertion that the square
  \begin{equation*}
    \begin{tikzcd}
      j_{!}j^{*}C\ar[r]&
      C\\
      f_{\natural}l_{!}l^{*}f^{*}C\ar[r]\ar[u]&
      f_{\natural}f^{*}C\ar[u]
    \end{tikzcd}
  \end{equation*}
  is a pushout.
  This can be checked by looking at the cofibers of rows.

  Then we prove that the induced functor is essentially surjective.
  We consider
  \(C_{U}\in\cat{C}_{U}\)
  and \(C_{Y}\in\cat{C}_{Y}\)
  with an equivalence \(g^{*}(C_{U})\simeq l^{*}(C_{Y})\).
We obtain a map
  \begin{equation*}
    i_{*}k^{*}(C_{Y})[-1]
    \overset{\text{\cref{i:ladj_i}}}{\simeq}
    i_{*}i^{*}f_{\natural}(C_{Y})[-1]
    \xrightarrow{\text{\cref{i:rec_x}}}
    j_{!}j^{*}f_{\natural}(C_{Y})
    \simeq
    j_{!}g_{\natural}l^{*}(C_{Y})
    \simeq
    j_{!}g_{\natural}g^{*}(C_{U})
    \to
    j_{!}(C_{U})
  \end{equation*}
  and its cofiber is the desired object.
\end{proof}

We note the following degenerate case:

\begin{corollary}\label{xrsyv9}
  For a recollement in \cref{xpyo8y},
  \begin{equation*}
    \begin{tikzcd}
      0\ar[d]&
      C_{Z}\ar[l]\ar[d,"i_{*}"]\\
      C_{U}&
      C_{X}\ar[l,"j^{*}"']
    \end{tikzcd}
  \end{equation*}
  is a pullback square.
\end{corollary}

\begin{remark}\label{xqrpo7}
  \Cref{xrsyv9} means that
  a recollement is determined by
  \(j^{*}\colon C_{X}\to C_{U}\)
  (or \(i^{*}\colon C_{X}\to C_{Z}\),
  since
  we can obtain~\(C_{U}\) is the kernel
  of~\(i^{*}\)
  by passing to the left adjoint).
\end{remark}

\begin{definition}\label{x64hfp}
  Given \cref{xqrpo7},
  we call the left square in \cref{xwog77}
  an \emph{excision square}.
\end{definition}

\begin{remark}\label{xibj1p}
In \cref{xpyo8y,x64hfp},
  we have considered the noncommutative situation.
  We can define the commutative version
  when the diagrams lie in \(\CAlg(\cat{C})\),
  simply by considering the same notion in \(\Mod_{C_{X}}(\cat{C})\).
\end{remark}

\section{Algebraic \texorpdfstring{\(2\)}{2}-motives}\label{s:dgs}

In \cref{ss:sh},
we review \(\SH\) and necessary facts.
In \cref{ss:cc_dg}, we prove \cref{main_1}.
In \cref{ss:mot_pb}, we prove the (relative) Künneth formula.

\subsection{A review of algebraic motivic spectra}\label{ss:sh}

The following was introduced in~\cite{MorelVoevodsky99}:

\begin{definition}[Morel–Voevodsky]\label{x00w04}
  Let \(X\) be a quasicompact quasiseparated static scheme.
  We write \(\Cat{Sm}_{X}\) for
  the category of finitely presented smooth \(X\)-schemes.
  The \emph{category of motivic spectra} over~\(X\)
  is defined as
  \begin{equation}
    \label{e:u0u27}
    \SH(X)
    =
    \Shv_{\Nis}(\Cat{Sm}_{X};\Cat{Sp})_{\AA^{1}}
    [(\Sigma^{\infty}\PP^{1})^{\otimes-1}],
  \end{equation}
  where
  \(\Shv_{\Nis}(\Cat{Sm}_{X};\Cat{Sp})_{\AA^{1}}\)
  denotes the full subcategory spanned by \(\AA^{1}\)-invariant Nisnevich sheaves.
  See~\cite[Section~2.1]{Robalo15}
  for the inversion procedure.
\end{definition}

\begin{remark}\label{xhx6bu}
  More precisely,
  what was considered in~\cite{MorelVoevodsky99} was
  hypercomplete sheaves unlike \cref{x00w04}.
  Note that \(\Shv_{\Nis}(\Cat{Sm}_{X})\)
  is hypercomplete when \(X\) is of finite type over~\(\ZZ\)
  (cf.~\cite[Theorem~3.18]{ClausenMathew21}).
  We prefer to work with this version,
  since
  \(\SH\colon\Cat{Ring}\to\CAlg(\Cat{Pr})\)
  preserves filtered colimits
  with this definition.
\end{remark}

Note the obvious functoriality;
\(\SH\) itself is a contravariant functor
to \(\CAlg(\Cat{Pr}_{\st})\).
Hence, we have four operations.
Smooth base change is straightforward from the definition.
Proper base change for~\(\SH\)
was established in Ayoub’s thesis~\cite{Ayoub07a,Ayoub07b}\footnote{
  Voevodsky first stated this without proof.
  Röndigs also independently worked on it, which remains unpublished.
}.
More precisely,
he demonstrated
in~\cite[Scholie~1.4.2]{Ayoub07a}
that a specific set of axioms
implies proper base change.
He also considered the symmetric monoidal version
in~\cite[Section~2.3]{Ayoub07a}.
We state both versions here:

\begin{theorem}[Ayoub]\label{ayoub}
  Let \(S\) be a divisorial noetherian static scheme.
  We write \(\Cat{QProj}_{S}\)
  for the category of static schemes quasiprojective over~\(S\).
Let \(D\colon(\Cat{QProj}_{S})^{\op}\to\Cat{Pr}_{\st}\)\footnote{More generally,
    he worked in the setting of triangulated categories.
  } be a functor.
  We assume that it satisfies the following:
  \begin{conenum}
    \item\label{i:y_nsm}
      Smooth morphisms satisfy base change;
      i.e., when \(f\colon Y\to X\) is smooth,
      then \(f^{*}\colon D(X)\to D(Y)\) admits
      a left adjoint~\(f_{\natural}\)
      such that
      the Beck–Chevalley transformation
      \(f'_{\natural}q^{*}\to p^{*}f_{\natural}\)
      is an equivalence for any pullback square~\cref{e:iljxp}.
    \item\label{i:y_exc}
      For any closed immersion \(i\colon Z\hookrightarrow X\)
      with its complement \(j\colon U\hookrightarrow X\),
      the diagram \(D(U)\gets D(X)\to D(Z)\) is a recollement\footnote{Morel–Voevodsky~\cite[Theorem~2.21]{MorelVoevodsky99}
        proved that \(\SH\) satisfies this property.
      };
      see \cref{xhypzi} below.
    \item\label{i:y_hi}
      For the tautological map \(f\colon\AA^{1}_{X}\to X\),
      the functor~\(f^{*}\colon D(X)\to D(\AA^{1}_{X})\) is fully faithful.
    \item\label{i:y_tate}
      In the situation of~\cref{i:y_hi},
      the functor \(f_{\natural}s_{*}\colon D(X)\to D(X)\)
      is an equivalence,
      where \(s\colon X\to\AA^{1}_{X}\) is the zero section.
  \end{conenum}
  Then projective morphisms satisfy base change;
  i.e., when \(p\colon X'\to X\) is projective,
  \(p^{*}\colon D(X)\to D(X')\)
  admits a right adjoint~\(p_{*}\) in~\(\Cat{Pr}_{\st}\)
  such that
  the Beck–Chevalley transformation
  \(f'_{\natural}q^{*}\to p^{*}f_{\natural}\)
  is an equivalence for any pullback square~\cref{e:iljxp}.

  Moreover,
  we consider a lax symmetric monoidal functor
  \(D\colon(\Cat{QProj}_{S})^{\op}\to\Cat{Pr}_{\st}\)
  satisfying the above conditions,
  where~\cref{i:y_nsm} is replaced with
  the following:
  \begin{conenum}[resume]
    \item\label{i:y_msm}
      Smooth morphisms satisfy base change
      and the projection formula;
      i.e., \cref{i:y_nsm} holds
      and moreover \(f_{\natural}\)
      is \(D(X)\)-linear.
  \end{conenum}
  In this case,
  the projection formula holds for projective morphisms;
  i.e., \(p_{*}\) is also \(D(X)\)-linear
  for any projective morphism~\(p\colon Y\to X\).
\end{theorem}

\begin{remark}\label{xhypzi}
For \cref{ayoub},
  Ayoub originally formulated \cref{i:y_exc}
  as \(D(\emptyset)=0\),
  the full faithfulness of~\(i_{*}\),
  and the conservativity
  of \((j^{*},i^{*})\colon D(X)\to D(U)\oplus D(Z)\).
  The equivalence follows from
  applying the smooth base change to \(U\to X\) along~\(i\).
\end{remark}

\begin{remark}\label{xflfxz}
  To prove \cref{ayoub},
  Ayoub first constructed shriek functoriality
  on the level of triangulated categories.
  Then he used it to prove
  the proper base change for~\(\PP^{n}\).
  The point is that we cannot use
  the usual recognition principle based on compactifications
  to obtain the shriek functoriality,
  since properness is what we wish to show.
\end{remark}

\begin{remark}\label{x86fao}
  At least the first part of \cref{ayoub} is valid
  when \(\Cat{Pr}_{\st}\)
  is replaced with a general stable presentable \(2\)-category.
  We prove this in \cref{aminus} below.
\end{remark}

Drew–Gallauer~\cite{DrewGallauer22} identified
\(\SH\) as the universal six-functor formalism:

\begin{theorem}[Drew–Gallauer]\label{drga}
  We use \(S\) and \(\Cat{QProj}_{S}\)
  as in \cref{ayoub} (but see \cref{xlsbil} below).
We consider the following \emph{nonfull} subcategory
  of \(\CAlg(\Fun((\Cat{QProj}_{S})^{\op},\Cat{Pr}_{\st}))\):
  \begin{itemize}
    \item
      Objects are
      lax symmetric monoidal functors \(D\colon(\Cat{QProj}_{S})^{\op}\to\Cat{Pr}_{\st}\)
      satisfying~\cref{i:y_exc,i:y_hi,i:y_tate,i:y_msm}
      of \cref{ayoub}.
    \item
      Morphisms are those natural transformations \(\alpha\colon D\to D'\)
      such that for any smooth morphism \(f\colon Y\to X\),
      the canonical morphism
      \(f_{\natural}\alpha_{Y}\to \alpha_{X}f_{\natural}\) is an equivalence.
  \end{itemize}
  Then \(\SH\) is initial in this category.
\end{theorem}

\begin{remark}\label{xlsbil}
  More precisely,
  Drew–Gallauer~\cite{DrewGallauer22}
  considered the category of schemes of finite type over
  a finite-dimensional noetherian base,
  unlike \cref{drga}.
  In our situation,
  their proof a~priori produces
  for~\(X\)
  the category
  in~\cref{e:u0u27}
  with \(\Cat{Sm}_{X}\)
  being replaced
  by its full subcategory spanned by those
  quasiprojective over~\(X\),
  but it coincides with~\(\SH(X)\)
  by the standard basis argument.
  Note that the same remark applies
  (and hence the same result holds)
  when we replace \(\Cat{QProj}_{S}\)
  with the full subcategory spanned by affines.
\end{remark}

\subsection{Algebraic \texorpdfstring{\(2\)}{2}-motives}\label{ss:cc_dg}

In the proof,
we need the following enhancement of \cref{ayoub}:

\begin{proposition}\label{aplus}
  In the situation of \cref{ayoub},
  suppose that
  we have a natural transformation \(\alpha\colon D\to D'\)
  between functors satisfying the conditions.
  If \(\alpha\) is compatible with smooth base change
  (cf.~\cref{drga}),
  it is also compatible with projective base change.
\end{proposition}

\begin{proof}
  We split this into the cases
  of closed immersions
  and \(\PP^{n}_{X}\to X\).
  The former case follows from \cref{i:a_exc}
  and smooth base change for open immersions.
  Therefore,
  it suffices to consider the latter case.

  We now consider a general smooth projective morphism
  \(f\colon Y\to X\).
  We form the diagram
  \begin{equation*}
    \begin{tikzcd}
      Y\ar[r,"d"]&
      Y\times_{X}Y\ar[r,"p"]\ar[d,"q"']&
      Y\ar[d,"f"]\\
      {}&
      Y\ar[r,"f"]&
      X\rlap.
    \end{tikzcd}
  \end{equation*}
  Then for~\(D\) (and similarly for~\(D'\)),
  we consider the natural transformation
  \begin{equation*}
    f_{\natural}
    \simeq f_{\natural}p_{*}d_{*}
    \to f_{*}q_{\natural}d_{*},
  \end{equation*}
  where the second morphism
  is induced by the double Beck–Chevalley transformation.
  Ayoub proved that
  this natural transformation is an equivalence
  and \(q_{\natural}d_{*}\) is invertible.
  By the commutative diagram
\begin{equation*}
    \begin{tikzcd}
      {}&
      {}&
      f_{\natural}p_{*}d_{*}\alpha_{Y}\ar[r,"\simeq"]&
      f_{*}q_{\natural}d_{*}\alpha_{Y}\\
      f_{\natural}\alpha_{Y}\ar[r,"\simeq"]\ar[d,"\simeq"']\ar[urr,"\simeq",bend left=15]&
      f_{\natural}\alpha_{Y}p_{*}d_{*}\ar[r]\ar[d,"\simeq"]&
      f_{\natural}p_{*}\alpha_{Y\times_{X}Y}d_{*}\ar[r]\ar[u,"\simeq"]&
      f_{*}q_{\natural}\alpha_{Y\times_{X}Y}d_{*}\ar[d,"\simeq"]\ar[u,"\simeq"']\\
      \alpha_{X}f_{\natural}\ar[r,"\simeq"]&
      \alpha_{X}f_{\natural}p_{*}d_{*}\ar[r,"\simeq"]&
      \alpha_{X}f_{*}q_{\natural}d_{*}\ar[r]&
      f_{*}\alpha_{Y}q_{\natural}d_{*}\rlap,
    \end{tikzcd}
  \end{equation*}
  where the hexagon commutes by \cref{hexagon},
  the bottom right horizontal morphism
  \(\alpha_{X}f_{*}q_{\natural}d_{*}\to f_{*}\alpha_{Y}q_{\natural}d_{*}\)
  is also an equivalence.
  Since \(q_{\natural}d_{*}\) is invertible,
  the Beck–Chevalley transformation
  \(f_{\natural}\alpha_{Y}\to \alpha_{X}f_{\natural}\)
  under consideration is an equivalence.
\end{proof}

\begin{corollary}\label{aminus}
  The first part of \cref{ayoub} remains valid
  when \(\Cat{Pr}_{\st}\)
  is replaced with
  any stable presentable \(2\)-category~\(\cat{C}\).
\end{corollary}

\begin{proof}
  We use \cref{ayoub} and Yoneda.
  We can verify adjointability
  using \cref{aplus}.
\end{proof}

We now prove \cref{main_1}:

\begin{proof}[Proof of \cref{main_1}]
In this proof,
  we write \(\Mot[2](S)\)
  for the universal target.
  We wish to prove that
  the morphism
  \(F\colon\Mot[2](S)\to\SH[2](S)\)
  given by the universal property
  is an equivalence.

  We first construct a morphism \(G\colon\SH[2](S)\to\Mot[2](S)\)
  under \((\Cat{QProj}_{S})^{\op}\).
  Recall from \cref{ker_p}
  that
  \(\SH[2](S)\) is defined
  to be \(\Mod_{\SH}(\Fun(\Span(\Cat{QProj}_{S},\Cat{Pr}_{\st})))\).
  Hence,
  it is equivalent
  to constructing a symmetric monoidal functor
  \(\Span(\Cat{QProj}_{S})\to\Mot[2](S)\)
  extending \([\X]\colon\Cat{QProj}_{S}^{\op}\to\Mot[2](S)\)
  and a morphism
  \(\alpha\colon\SH\to\Hom_{\Mot[2](S)}(\unit,[\X])\)
  in \(\CAlg(\Fun(\Span(\Cat{QProj}_{S}),\Cat{Pr}_{\st}))\).

  We first construct
  the desired extension of~\([\X]\)
  by restricting
  a functor from \(\Span[2]_{P,J}(\Cat{QProj}_{S})\),
  where \(J\) and~\(P\) consist of open immersions
  and projective morphisms, respectively.
  We check the \((J,P)\)-biadjointability condition in \cref{cll}.
  Base change for~\(J\) follows from \cref{i:a_sm}.
  For~\(P\),
  base change follows from \cref{aminus}.
  The double Beck–Chevalley condition
  can be verified by considering the complementary closed immersion
  and using \cref{i:a_exc}.
  Since \((\Cat{QProj}_{S})^{\op}\to\Mot[2](S)\)
  is symmetric monoidal, the projection formula
  automatically
  follows from base change.

  We now construct~\(\alpha\).
  We write \(D=\Hom_{\Mot[2]}(\unit,[\X])\).
  By applying \cref{drga},
  we obtain a natural transformation
  after restricting to \((\Cat{QProj}_{S})^{\op}\).
  We wish to invoke \cref{cll} again
  and restrict to \(\Span(\Cat{QProj}_{S})\)
  to construct~\(\alpha\).
  This requires verifying the condition in \cref{cll};
  namely,
  we need to show that
  \begin{align*}
    \begin{tikzcd}[ampersand replacement=\&]
      \SH(X)\ar[r,"j^{*}"]\ar[d,"\alpha_{X}"']\& 
      \SH(U)\ar[d,"\alpha_{U}"]\\
      D(X)\ar[r,"j^{*}"]\&
      D(U)\rlap,
    \end{tikzcd}
    &&
    \begin{tikzcd}[ampersand replacement=\&]
      \SH(X)\ar[r,"p^{*}"]\ar[d,"\alpha_{X}"']\& 
      \SH(Y)\ar[d,"\alpha_{Y}"]\\
      D(X)\ar[r,"p^{*}"]\&
      D(Y)
    \end{tikzcd}
  \end{align*}
  for any open immersion \(j\colon U\to X\)
  and projective morphism \(p\colon Y\to X\)
  are left and right adjointable, respectively.
  The former follows,
  since it is compatible with smooth base change
  by construction.
  The latter follows from \cref{aplus}.

  It remains to show that they are mutually inverse.
  By construction,
  \(GF\) is homotopic to~\(\id\)
  after composing with \([\X]\colon\Cat{QProj}^{\op}\to\Mot[2](S)\),
  and hence \(GF\) is homotopic to~\(\id\)
  by universality.
Hence,
  it suffices to prove that \(FG\)
  is homotopic to~\(\id\).
  By construction,
  it is homotopic to~\(\id\) under \((\Cat{QProj}_{S})^{\op}\),
  and by \cref{cll},
  under \(\Span(\Cat{QProj}_{S})\) as well.
  We hence obtain the corresponding
  autoequivalence
  of~\(D\) in \(\CAlg(\Fun(\Span(\Cat{QProj}_{S}),\Cat{Pr}_{\st}))\),
  which is homotopic to~\(\id\) by \cref{drga}.
\end{proof}

\subsection{\texorpdfstring{\(2\)}{2}-motives of pullbacks}\label{ss:mot_pb}

Here,
we present
an axiomatic method to obtain the relative Künneth formula:

\begin{proposition}\label{x15q1p}
Let \(\cat{G}\) be a category with finite limits,
  \(\cat{C}\) a presentably symmetric monoidal \(2\)-category,
  and \([\X]\colon\cat{G}^{\op}\to\cat{C}\) a symmetric monoidal functor.
  Suppose that \(I\) is a wide subcategory
  of~\(\cat{G}\) stable under base change
  and for \(i\in I\),
  the morphism \(i^{*}=[i]\)
  admits a fully faithful right adjoint~\(i_{*}\)
  satisfying proper base change.
  Then for a pullback \(Y'=X'\times_{X}Y\) in~\(\cat{G}\)
  such that the diagonal of~\(X\) is in~\(I\),
  the canonical morphism
  \([X']\otimes_{[X]}[Y]\to[Y']\) in~\(\cat{C}\)
  is an equivalence.
\end{proposition}

\begin{lemma}\label{xadp26}
  Let \(\cat{C}^{\bullet}\) be a cosimplicial object in~\(\Cat{Cat}\)
  such that all cofaces are fully faithful.
  Suppose that
  \begin{equation*}
    \begin{tikzcd}
      \cat{C}^{m}\ar[r,"d^{0}"]\ar[d]&
      \cat{C}^{m+1}\ar[d]\\
      \cat{C}^{n}\ar[r,"d^{0}"]&
      \cat{C}^{n+1}
    \end{tikzcd}
  \end{equation*}
  is right adjointable
  for any \([m]\to[n]\) in~\(\Cat{\Delta}\).
  Then its totalization~\(\cat{C}^{-1}\)
  is identified with the full subcategory
  of~\(\cat{C}^{0}\) spanned by
  those objects~\(C\)
  such that
  \(d^{0}(C)\) is equivalent to \(d^{1}(C)\).
\end{lemma}

\begin{proof}
We write \(F\colon\cat{C}^{-1}\to\cat{C}^{0}\)
  for the tautological functor.
  By computing~\(\cat{C}^{-1}\) cosemisimplicially,
  we see that \(F\) is fully faithful.
  Moreover,
  it is clear that \(C\in\cat{C}^{0}\)
  is in the image of~\(F\),
  then \(d^{0}(C)\simeq d^{1}(C)\) holds.

  We prove the converse.
  Suppose that \(C\in\cat{C}^{0}\)
  satisfies \(d^{0}(C)\simeq d^{1}(C)\).
  We apply~\cite[Theorem~4.7.5.2]{LurieHA}
  to see that
  \begin{equation*}
    \begin{tikzcd}
      \cat{C}^{-1}\ar[r,"F"]\ar[d,"F"']&
      \cat{C}^{0}\ar[d,"d^{1}"]\\
      \cat{C}^{0}\ar[r,"d^{0}"]&
      \cat{C}^{1}
    \end{tikzcd}
  \end{equation*}
  is right adjointable.
  By \(C\simeq d^{0,\R}d^{0}(C)
  \simeq d^{0,\R}d^{1}(C)\simeq FF^{\R}(C)\),
  we see that~\(C\) is in the image of~\(F\).
\end{proof}

\begin{proof}[Proof of \cref{x15q1p}]
  We realize \([X']\otimes_{[X]}[Y]\)
  as the geometric realization
  of \([X'\times X^{\times\bullet}\times Y]\)
  in~\(\cat{C}\).
  Then by considering \(\Hom_{\cat{C}}(\X,C)\)
  for an arbitrary object~\(C\),
  this is reduced to \cref{xadp26}.
\end{proof}

\begin{remark}\label{xzjpy8}
  Another useful technique we can use is
  \(X'\times_{X}Y=X\times_{X\times X}(X'\times Y)\);
  cf.~the proof of~\cref{x925qh}.
\end{remark}

\begin{corollary}\label{x6lx9q}
  In the situation of \cref{main_1},
  for a pullback \(Y'=X'\times_{X}Y\) in \(\Cat{QProj}_{S}\),
  the canonical morphism
  \([X']\otimes_{[X]}[Y]\to[Y']\) is an equivalence
  in~\(\SH[2](S)\).
\end{corollary}

\begin{corollary}\label{xfviz8}
  In the situation of \cref{main_1},
  for \(X\in\Cat{QProj}_{S}\),
  the morphism
  \(\SH[2](S)\to\SH[2](X)\)
  induces an equivalence
  \(\Mod_{[X]}(\SH[2](S))\simeq\SH[2](X)\).
\end{corollary}

\begin{proof}
  This follows from
  \cref{x15q1p}
  and~\cite[Proposition~4.8]{StefanichN}.
\end{proof}

\begin{corollary}\label{x0z85x}
  Let \(X\) be a divisorial noetherian static scheme
  and \(X'\to X\gets Y\) be quasiprojective morphisms.
  Then the canonical morphism
  \begin{equation*}
    \SH[2](X')\otimes_{\SH[2](X)}\SH[2](Y)
    \to
    \SH[2](X'\times_{X}Y)
  \end{equation*}
  is an equivalence.
\end{corollary}

\begin{remark}\label{x61y49}
  Following~\cite[Section~3]{StefanichN},
  we can also obtain
  the presentably symmetric monoidal \(n\)-category
  \(\SH[n](\ZZ)\)
  of \(n\)-motives
  for \(n\geq3\).
  By the result above,
  it is the iterated module \(n\)-category
  over \(\SH[2](\ZZ)\).
  In other words,
  stable motivic homotopy theory is \(2\)-affine.
\end{remark}

\section{Stacks over a \texorpdfstring{\(2\)}{2}-category}\label{s:stk}

In this section,
we fix a presentably symmetric monoidal \(2\)-category~\(\cat{C}\)
and study the geometry of the following:

\begin{definition}\label{xz4a0a}
  We call \(\CAlg(\cat{C})^{\op}\)
  the category of \emph{stacks} over~\(\cat{C}\).
  For a stack~\(X\), we write~\([X]\) for the corresponding object
  in \(\CAlg(\cat{C})\).
\end{definition}

\begin{remark}\label{x8fuv1}
  In the language of Scholze–Stefanich (cf.~\cite{Gestalten}),
  up to size restrictions,
  \cref{xz4a0a}
  yields
  the category of \(1\)-affine gestalts over \(\Gest(\cat{C})\).
\end{remark}

The following examples are particularly important:

\begin{example}\label{xm6f2s}
  Let \(X\) be a locale.
  Then \(\Shv(X)\) is a \(0\)-truncated stack
  over \(\Cat{Pr}\);
  this follows from~\cite[Corollary~1.10]{ttg-sm}
  (it is more straightforward in this particular case).
  When no confusion can arise,
  we write \(X\) for this.
\end{example}

\begin{example}\label{xu4vl1}
  Let \(X\) be an accessible presheaf
  on~\(\Cat{CAlg}^{\op}\),
  the opposite category of \(\E_{\infty}\)-rings.
  Then \(\D(X)=\projlim_{\Spec(A)\to X}\D(A)\)
  is a stack over \(\Cat{Pr}_{\st}\).
\end{example}

In \cref{ss:stk_suave},
we introduce some classes of morphisms of stacks.
In \cref{ss:six},
we explain a certain way to construct six operations.
In \cref{ss:stk_ring},
we explain the notion of ring stacks
used in \cref{main_2,main_3,main_4}.
In \cref{ss:stk_abs},
we explain the notion of absolute values
on ring stacks
used in \cref{main_4}.

\subsection{Suave and prim maps}\label{ss:stk_suave}

See~\cite[Remark~4.4.2]{HeyerMann}
for the origin
of these terminologies.

\begin{definition}\label{x5gcg1}
  We consider a morphism of stacks
  \(Y\to X\).
  We call it \emph{weakly suave}
  when \([X]\to[Y]\) admits an \([X]\)-linear left adjoint.
  We call it \emph{suave}
  if furthermore \([Y]\) is dualizable over~\([X]\).
  We call it \emph{prim}
  when \([X]\to[Y]\) admits an \([X]\)-linear right adjoint.
\end{definition}

\begin{remark}\label{xisa2l}
  In \cref{x5gcg1},
  by \cref{xahxu2} below,
  the additional condition for being suave
  is equivalent to the assumption
  that the morphism \(\Mod_{[X]}(\cat{C})\to\Mod_{[Y]}(\cat{C})\)
  also admits a \(\Mod_{[X]}(\cat{C})\)-linear left adjoint.
\end{remark}

\begin{lemma}\label{xahxu2}
  Consider \(A\in\CAlg(\cat{C})\).
  The following are equivalent:
  \begin{conenum}
    \item\label{i:ac_dbl}
      The underlying object of~\(A\) is dualizable.
    \item\label{i:ac_left}
      The functor
      \(\cat{C}\to\Mod_{A}(\cat{C})\)
      admits a \(\cat{C}\)-linear left adjoint.
    \item\label{i:ac_right}
      The forgetful functor
      \(\Mod_{A}(\cat{C})\to\cat{C}\)
      admits a \(\cat{C}\)-linear right adjoint.
  \end{conenum}
\end{lemma}

\begin{proof}
  The equivalence~\(\text{\cref{i:ac_dbl}}\Leftrightarrow\text{\cref{i:ac_right}}\)
  is well known.
  The equivalence~\(\text{\cref{i:ac_left}}\Leftrightarrow\text{\cref{i:ac_right}}\)
  follows from taking the duals of both sides as \(\cat{C}\)-modules.
\end{proof}

\begin{example}\label{x2m9yp}
  Let \(\cat{G}\) be a category with finite limits.
  Suppose that \(\cat{C}\) is the presentable \(2\)-category of kernels
  associated with a lax symmetric monoidal functor
  \(D\colon\Span(\cat{G})\to\Cat{Pr}_{\st}\).
  Then,
  to check that \(Y\to X\)
  in \(\cat{G}\) determines a suave (or prim) morphism
  over~\(\cat{C}\),
  it suffices to check that
  \([X]\to[Y]\)
  has an \([X]\)-linear left (or right) adjoint.
\end{example}

\begin{example}\label{et_sua}
  Any étale geometric morphism of toposes
  \(\cat{Y}\to\cat{X}\),
  i.e., a morphism of the form \(\cat{X}_{/X}\to\cat{X}\)
  for some object \(X\in\cat{X}\),
  determines a weakly suave morphism of stacks over~\(\Cat{Pr}\);
  see~\cite[Remark~6.3.5.12]{LurieHTT}.
\end{example}

\begin{example}\label{xf141x}
  Not every proper morphism \(\cat{Y}\to\cat{X}\)
  of toposes
  (see~\cite[Definition~7.3.1.4]{LurieHTT})
  determines a prim morphism
  of stacks over \(\Cat{Pr}\).
  However,
  \(\Shv(\cat{X};\Cat{Sp})\to\Shv(\cat{Y};\Cat{Sp})\)
  is a prim morphism of stacks over \(\Cat{Pr}_{\st}\);
  see, e.g.,~\cite[Lemma~6.7]{verdier-asc}.
\end{example}

\begin{example}\label{x9fu9c}
  Let \(Y\to X\) be a morphism
  of quasicompact quasiseparated schemes.
  Then the corresponding morphism of stacks
  over \(\Cat{Pr}_{\st}\) is prim.
\end{example}

\begin{definition}\label{x4q61y}
  An \emph{open immersion}
  is a suave monomorphism.
  A \emph{closed immersion}
  is a prim monomorphism.
\end{definition}

\begin{lemma}\label{xuw75d}
  Let \(f\colon Y\to X\) be a morphism of stacks.
  It is an open (or closed) immersion
  if and only if \([X]\to[Y]\) admits
  a fully faithful \([X]\)-linear left (or right) adjoint.
\end{lemma}

\begin{proof}
  We only treat the open immersion case;
  the closed immersion case is easier.

  We first prove the “only if” direction.
  Since \(f\) is weakly suave,
  it admits a linear left adjoint~\(f_{\natural}\).
  The left adjointability
  of the square
  \begin{equation*}
    \begin{tikzcd}
      {[X]}\ar[r]\ar[d]&
      {[Y]}\ar[d]\\
      {[Y]}\ar[r]&
      {[Y]}
    \end{tikzcd}
  \end{equation*}
  implies that the unit \(\id\to f^{*}f_{\natural}\)
  is an equivalence.
  
We then prove the “if” direction.
  In this situation,
  \([Y]\to[Y]\otimes_{[X]}[Y]\)
  admits a fully faithful left adjoint,
  and since
  \begin{equation*}
    \begin{tikzcd}
      {[X]}\ar[r]\ar[d]&
      {[Y]}\ar[d]\\
      {[Y]}\ar[r]&
      {[Y]\otimes_{[X]}[Y]}
    \end{tikzcd}
  \end{equation*}
  is left adjointable,
  it must be an equivalence.
  Therefore, it is a monomorphism.
  By definition, it is weakly suave.
  We see that \([Y]\) is self-dual over~\([X]\)
  in this situation,
  so it is suave.
\end{proof}

\begin{definition}\label{xuzdhm}
  We consider a static morphism of stacks.
  We call a static map of stacks \emph{unramified}
  if its diagonal is an open immersion.
  We call it \emph{étale}
  if it is suave and unramified.
  We call it \emph{separated}
  if its diagonal is a closed immersion.
  We call it \emph{proper}
  if it is prim and separated.
\end{definition}

\subsection{Six operations for open stacks}\label{ss:six}

We need to construct this in a situation
where we do not have access to compactification.
We use \cref{cll} to construct this.

\begin{definition}\label{xm14yl}
  We say a morphism of stacks \(Y\to X\)
  is \emph{exceptional} if
  it is \(n\)-truncated for some~\(n\)
  and
  the object~\([Y]\)
  is dualizable over \([Y^{\times_{X}S^{n-1}}]
  =[Y]^{\otimes_{[X]}S^{n-1}}\)
  for any \(n\geq0\).
\end{definition}

\begin{proposition}\label{x5xzh9}
  Exceptional morphisms are closed under base change and composition
  and satisfy cancellation.
\end{proposition}

\begin{lemma}\label{xda67d}
  For morphisms \(A\to B\to C\)
  in \(\CAlg(\cat{C})\),
  when \(B\) is dualizable over~\(A\) and \(C\) is dualizable over~\(B\),
  then \(C\) is dualizable over~\(A\).
\end{lemma}

\begin{proof}
  This follows from \cref{xahxu2}.
\end{proof}

\begin{proof}[Proof of \cref{x5xzh9}]
  Base change is straightforward.
We consider morphisms of stacks \(Z\to Y\to X\).
  Composition follows from
  \cref{xda67d} and
  the diagram
  \begin{equation*}
    \begin{tikzcd}
      Z\ar[r]&
      Z^{\times_{Y}S^{n-1}}\ar[r]\ar[d]&
      Z^{\times_{X}S^{n-1}}\ar[d]\\
      {}&
      Y\ar[r]&
      Y^{\times_{X}S^{n-1}}\rlap,
    \end{tikzcd}
  \end{equation*}
  where the square is a pullback.
  The cancellation property
  follows from
  \cref{xda67d} and
  the diagram
  \begin{equation*}
    \begin{tikzcd}
      Z\ar[r]\ar[d]&
      Z^{\times_{Y\times_{X}Z}S^{n-1}}\ar[r,equal]\ar[d]&
      Z^{\times_{Y\times_{X}Z}S^{n-1}}\ar[r]\ar[d]&
      Z^{\times_{Y}S^{n-1}}\ar[d]\\
      Y\ar[r]&
      Y^{\times_{X}S^{n}}&
      Z\ar[r]&
      Z^{\times_{X}S^{n-1}}\rlap,
    \end{tikzcd}
  \end{equation*}
  where both squares are pullbacks.
\end{proof}

Consequently,
the relevant span category is well defined.
We prove the following:

\begin{theorem}\label{xc8fbc}
  We write~\(\cat{G}\)
  for the category of stacks
  and~\(E\) for the class of exceptional morphisms.
  Then we have a canonical symmetric monoidal functor
  \(\Span_{E}(\cat{G})\to\cat{C}\)
  extending
  \(\cat{G}^{\op}\to\cat{C}\).
\end{theorem}

\begin{proof}
  We consider \(
    \cat{G}^{\op}\to\Mod_{\cat{C}}(2\Cat{Pr})
  \) given by \(X\mapsto\Mod_{[X]}(\cat{C})\).
  We claim that
  this satisfies \((E,E)\)-biadjointability condition of \cref{cll}.
  If this is the case,
  we obtain a canonical extension
  \(
  \Span[2]_{E;E,E}(\cat{G})\to\Mod_{\cat{C}}(2\Cat{Pr})
  \) by \cref{cll},
  and by looking at \(\End(\unit)\),
  we obtain the desired functor.

  We need to see that the biadjointability condition
  is satisfied;
  the projection formula is automatic by symmetric monoidality.
  The Beck–Chevalley condition for \(P=E\)
  is clear.
  For~\(J=E\),
  when \(Y\to X\) is exceptional,
  we first see that \(\Mod_{[X]}(\cat{C})\to\Mod_{[Y]}(\cat{C})\)
  admits a left adjoint
  given by \([Y]^{\vee}\otimes_{[Y]}\X\),
  where \([Y]^{\vee}\) denotes the dual of~\([Y]\) over~\([X]\).
  By considering the right adjoints of all morphisms
  appearing in the Beck–Chevalley transformation,
  the Beck–Chevalley condition
  is reduced to the Beck–Chevalley condition for~\(P\).
To check the double Beck–Chevalley condition,
  again we can consider the right adjoints.
\end{proof}

\begin{example}\label{xxxpuy}
  Consider \(\cat{C}=\Cat{Pr}_{\st}\).
  Let \(\cat{G}\) be the category
  of locally compact Hausdorff spaces
  of countable weight
  that are disjoint unions of open subsets of~\(\RR^{n}\)
  for varying~\(n\geq0\).
  Once we verify that the morphisms
  \(\RR\to{*}\) and \({*}\xrightarrow{0}\RR\) are exceptional,
  \cref{xc8fbc} yields six operations,
  where \(E\) consists of all morphisms.
  This agrees with the classical one.
\end{example}

\subsection{Ring stacks}\label{ss:stk_ring}

We define these notions
using Lawvere theories.

\begin{definition}\label{xwvxhj}
  We write \(\Cat{Lat}^{+}\), \(\Cat{Lat}\), and \(\Cat{Pol}\)
  for the full subcategories of abelian monoids, abelian groups, and commutative rings,
  respectively, spanned by~\(\NN^{n}\), \(\ZZ^{n}\),
  and \(\ZZ[T_{1},\dotsc,T_{n}]\) for \(n\geq 0\).
\end{definition}

\begin{definition}\label{xwh85x}
  A \emph{monoid stack}, \emph{group stack}, and \emph{ring stack}
  are functors \(\Cat{Lat}^{+}\to\CAlg(\cat{C})\), \(\Cat{Lat}\to\CAlg(\cat{C})\),
  and \(\Cat{Pol}\to\CAlg(\cat{C})\),
  respectively, preserving finite coproducts.
\end{definition}

\begin{remark}\label{x2vzh2}
  In \cref{xwh85x},
  it is more precise
  to say that they are strict or animated,
  since they are not the free ones
  in the categorical sense.
  In this paper,
  we restrict our attention to such situations.
\end{remark}

\begin{example}\label{x0qjpa}
  The left adjoint
  of the “multiplicative” forgetful functor
  from commutative rings to abelian monoids
  is restricted to \(\Cat{Lat}^{+}\to\Cat{Pol}\).
  By composing this with a ring stack,
  we obtain the monoid stack of the underlying multiplicative monoid.
  Similarly,
  we have the underlying additive group stack
  by composing with \(\Cat{Lat}\to\Cat{Pol}\).

  Moreover,
  we also obtain the multiplicative group stack
  of invertible elements
  by considering
  \(\Cat{Lat}\to\PShv_{\Sigma}(\Cat{Pol})\)
  given by \(\ZZ^{n}\mapsto\ZZ[\ZZ^{n}]\).
\end{example}

\begin{example}\label{abs_t}
  When \(X\) is a monoid locale,
  then \(\Shv(X)\) naturally becomes
  a monoid stack in \(2\Cat{Pr}\).

  Two examples are particularly important.
  First, the Sierpiński space~\(\{s,\eta\}\),
  where~\(s\) is special and~\(\eta\) is generic,
  with multiplication given by \(s\cdot s=s\), \(s\cdot\eta=s\),
  and \(\eta\cdot\eta=\eta=1\)
  is a monoid locale.
  Second,
  \([0,\infty)\) with the usual multiplication
  is also a monoid locale.
  Note that \(\{s,\eta\}\) is 
  a quotient of \([0,\infty)\)
  as a monoid locale
  by \(0\mapsto s\) and \(r\mapsto\eta\) for \(r>0\).
\end{example}

\begin{definition}\label{x924i9}
  We call a ring stack \emph{sutured} if
  the diagram
  \(R^{\times}\to R\xleftarrow{0}{*}\)
  of stacks
  induces a recollement in~\(\Mod_{[R]}(\cat{C})\)
  (cf.~\cref{xibj1p}).
\end{definition}

\begin{lemma}\label{xc95y9}
  The diagonal of a sutured ring stack is a closed immersion.
  In particular, any sutured ring stack is static.
\end{lemma}

\begin{proof}
By considering the shearing map \((x,y)\mapsto(x,x-y)\),
  the diagonal \(R\to R\times R\)
  is isomorphic to
  \({\id}\times0\colon R\times{*}\to R\times R\).
  Since \(0\colon{*}\to R\) is a base change
  of the inclusion of the closed point into the Sierpiński space,
  the desired result follows.
\end{proof}

\begin{definition}\label{xvbdgl}
  We call a sutured weakly suave ring stack \emph{stable}
  if \(f_{\natural}s_{*}\) is invertible,
  where \(s\colon{*}\to R\) is the zero section.
\end{definition}

\begin{lemma}\label{xujpr0}
  A stable sutured weakly suave ring stack
  is suave and exceptional.
\end{lemma}

\begin{proof}
  To prove this,
  we need an elementary fact about duality.
  Consider
  morphisms
  \begin{align*}
    \eta&\colon\unit\to C\otimes D,&
    \epsilon&\colon D\otimes C\to\unit
  \end{align*}
  in a monoidal \((1,1)\)-category.
  By definition,
  \(C\) and~\(D\) are dual to each other
  when the composites
  \begin{align*}
    C&\xrightarrow{\eta\otimes\id_{C}}C\otimes D\otimes C
    \xrightarrow{\id_{C}\otimes\epsilon}C,&
    D&\xrightarrow{\id_{D}\otimes\eta}D\otimes C\otimes D
    \xrightarrow{\epsilon\otimes\id_{D}}D
  \end{align*}
  are the identities.
  However,
  even when the composites are just isomorphisms,
  instead of the identities,
  we can still conclude that~\(C\) and~\(D\) are dual to each other
  by twisting~\(\eta\)
  to obtain a duality datum;
  see, e.g.,~\cite[Lemma~5.11]{SixFunctors}.

  We show that \([R]\) is self-dual
  using the observation above.
  We write \(f\colon R\to{*}\) for the tautological map
  and \(d\colon R\to R\times R\) for the diagonal,
  which is a closed immersion by \cref{xc95y9}.
  We claim that
  \begin{align*}
    \eta&\colon[*]\xrightarrow{f^{*}}[R]\xrightarrow{d_{*}}[R\times R],&
    \epsilon&\colon[R\times R]\xrightarrow{d^{*}}[R]\xrightarrow{f_{\natural}}[*]
  \end{align*}
  fit into the situation above.
  To see this,
  we consider the diagram
  \begin{equation*}
    \begin{tikzcd}
      R\ar[r,"d"]\ar[d,"d"']&
      R\times R\ar[r,"\pr_{2}"]\ar[d,"{\id}\times d"]&
      R\\
      R\times R\ar[r,"d\times{\id}"]\ar[d,"\pr_{1}"']&
      R\times R\times R&
      {}\\
      R\rlap,&
      {}&
      {}
    \end{tikzcd}
  \end{equation*}
  where the square is a pullback.
  By proper base change,
  we see that the first composite is equivalent to
  \begin{equation*}
    [R]
    \xrightarrow{\pr_{2}^{*}}[R\times R]
    \xrightarrow{d^{*}}[R]
    \xrightarrow{d_{*}}[R\times R]
    \xrightarrow{\pr_{1,\natural}}[R].
  \end{equation*}
  By shearing as in the proof of \cref{xc95y9},
  we see that the stability implies that
  this is an equivalence
  (but still not equivalent to the identity).
  The other composite is treated similarly.

  Since \([R]\) is also self-dual over \([R]\otimes[R]\)
  by \cref{xc95y9},
  it is exceptional.
\end{proof}

\begin{definition}\label{x4vjro}
  We call a sutured ring stack
  \emph{smooth}
  if it is suave
  (and then it is automatically exceptional by \cref{xc95y9},
  hence the six operations make sense)
  and \(f^{!}\unit\) is invertible
  in \(\Hom_{\cat{C}}(\unit,[R])\).
\end{definition}

\begin{proposition}\label{st_sm}
  For a sutured ring stack,
  it is smooth
  if and only if it is weakly suave and stable.
\end{proposition}

\begin{proof}
  We write~\(\cat{G}\) for the category
  of exceptional stacks
  and consider
  \begin{equation*}
    D=\Hom_{\cat{C}}(\unit,[\X])\colon\Span(\cat{G})\to\Cat{Pr}
  \end{equation*}
  obtained from \cref{xc8fbc}.
  Note that in both situations,
  \(R=\GG_{a}\) is exceptional
  by \cref{xc95y9,xujpr0}.
  By the extension result
  of Mann~\cite[Section~A.5]{Mann}
  (see also~\cite[Section~3.4]{HeyerMann}),
  we extend this to the situation where
  the tautological morphism \(t\colon{*}\to B\GG_{a}\) is shriekable.
  We consider the diagram
  \begin{equation*}
    \begin{tikzcd}
      {*}\ar[r,"s"]&
      \GG_{a}\ar[r,"f"]\ar[d,"f"']&
      {*}\ar[d,"t"]\\
      {}&
      {*}\ar[r,"t"]&
      B\GG_{a}\rlap.
    \end{tikzcd}
  \end{equation*}
  By \(f^{!}\unit\simeq f^{!}t^{*}\unit\simeq f^{*}t^{!}\unit\),
  we see
  that \(R\) is smooth
  if and only if \(L=t^{!}\unit\) is invertible.
  Since \(f_{!}s_{*}\simeq f_{!}s_{!}\simeq{\id}\),
  we have
  \begin{equation*}
    L\simeq
    s^{*}f^{!}\unit
    \simeq
    f_{!}s_{*}s^{*}f^{!}\unit
    \simeq
    f_{!}(f^{!}\unit\otimes s_{*}\unit)
    \simeq
    f_{\natural}s_{*}\unit.
  \end{equation*}
  Therefore,
  stability is equivalent to the invertibility of~\(L\),
  and hence the desired result follows.
\end{proof}

\subsection{Absolute values}\label{ss:stk_abs}

\begin{definition}\label{x2m7c9}
  An \emph{absolute semivalue} on a ring stack~\(R\)
  over~\(\cat{C}\) is a multiplicative
  map \(N\colon R\to[0,\infty)\),
  i.e., a \(2\)-cell
  \begin{equation*}
    \begin{tikzcd}
      \Cat{Lat}^{+}\ar[rd,"{\Shv([0,\infty))}",""'{name=a}]\ar[d]&
      {}\\
      \Cat{Pol}\ar[r,""{name=b},"R"']&
      \CAlg(\cat{C})\rlap,
      \ar[from=a,to=b,Rightarrow]
    \end{tikzcd}
  \end{equation*}
  where we use \cref{x0qjpa,abs_t},
  satisfying the following conditions:
  \begin{itemize}
    \item
      We have \(N(0)\leq0\);
      i.e., the morphism
      \(0\colon{*}\to R\)
      factors through
      \(\{0\}\to[0,\infty)\),
      which is a monomorphism.
    \item
      It satisfies the triangle inequality
      \(N(x+y)\leq N(x)+N(y)\);
      i.e.,
      we have
      \(N(N^{-1}([0,r])+N^{-1}([0,s]))\subset[0,r+s]\)
      for \(r\),~\(s\in[0,\infty)\).
  \end{itemize}
  We call it an \emph{absolute value}
  if it moreover satisfies the following:
  \begin{itemize}
    \item
      We have \(N^{-1}(\{0\})=\{0\}\).
    \item
      We have \(N^{-1}((0,\infty))=R^{\times}\).
  \end{itemize}
\end{definition}

\begin{remark}\label{xljgmw}
  A ring stack with an absolute value
  is automatically sutured.
  In fact,
  a sutured ring stack can be formulated similarly
  to \cref{x2m7c9}
  using the Sierpiński space~\(S\)
  in place of \([0,\infty)\)
  (see \cref{abs_t}).
\end{remark}

\begin{remark}\label{xraf07}
  In \cref{x2m7c9},
  certain completeness or overconvergence
  is already inherent
  when we consider a morphism \(R\to[0,\infty)\).
  This is because \([0,\infty)\)
  automatically has the overconvergence property,
  and it is hence inherited by~\(R\);
  e.g., for any~\(r\),
  the disk of radius~\(r\)
  (as a substack of~\(R\))
  is the intersection
  of the disks of radius~\(r'\) for \(r'>r\).
\end{remark}

We resume the study of this notion in \cref{s:main_ber},
where we prove \cref{main_4}.
We conclude this section by describing what it means to have
a multiplicative map \(R\to[0,\infty)\)
concretely:

\begin{lemma}\label{xpbyl2}
  Let \(X\) be a stack.
  The morphism
  \(X\to[0,\infty)\)
  is a specification of idempotent algebras
  \(\dotsb\to D_{r}\to\dotsb\to D_{0}\) for 
  \(r\in[0,\infty)\)
  and \(\unit=E_{0}\to\dotsb\to E_{r}\to\dotsb\)
  for \(r\in[0,\infty)\)
  in \(\Hom_{\cat{C}}(\unit,[X])\)
  satisfying the following:
  \begin{conenum}
    \item\label{i:int_d}
      We have \(\injlim_{r'>r}D_{r'}=D_{r}\).
    \item\label{i:int_e}
      We have \(\injlim_{r'<r}E_{r'}=E_{r}\).
    \item\label{i:int_de1}
      and \(D_{r}\vee E_{r}=\unit\) for~\(r\geq0\).
    \item\label{i:int_de0}
      We have \(D_{r}\otimes E_{s}=0\) for \(0\leq r<s\).
  \end{conenum}
\end{lemma}

\begin{proof}
By~\cite{ttg-sm},
  we have to check
  that the universal frame
  described by these conditions
  coincides with \([0,\infty)\).

  First,~\cref{i:int_d} describes
  the topology on \([0,\infty)\)
  such that \((r,\infty)\) is open for \(r\geq0\).
  Then~\cref{i:int_e} describes
  the topology on \([0,\infty)\)
  such that \([0,r)\) is open for \(r\geq0\).
  This means that
  the universal locale
  is a sublocale of their product.
  The conditions~\cref{i:int_de1,i:int_de0}
  restrict this to
  the diagonal part of this space,
  which is precisely the usual topology on \([0,\infty)\).
\end{proof}

\begin{lemma}\label{xs6zl2}
  Let \(M\) be a monoid stack.
  A multiplicative map
  \(M\to[0,\infty)\)
  is a datum described in \cref{xpbyl2}
  satisfying
  \begin{align*}
    \{1\}&\subset D_{1},&
    D_{r}\cdot D_{s}&\subset D_{rs},&
    \{1\}&\subset E_{1},&
    E_{r}\cdot E_{s}&\subset E_{rs}
  \end{align*}
  for \(r\) and~\(s\in[0,\infty)\),
  where we identify idempotent algebras with the corresponding
  closed substacks;
  e.g.,
  \(D_{r}\cdot D_{s}\subset D_{rs}\) means that
  the action of \([M]\) on
  \(\Mod_{D_{r}}[M]\otimes\Mod_{D_{s}}[M]\)
  coming from the binary multiplication on~\(M\)
  factors through \(\Mod_{D_{rs}}[M]\).
\end{lemma}

\begin{proof}
By the \(0\)-truncatedness of \([0,\infty)\),
  multiplicativity is a condition,
  which is precisely the commutativity
  (which is a condition) of the diagrams
  \begin{align*}
    \begin{tikzcd}[ampersand replacement=\&]
      {*}\ar[r]\ar[d,"1"']\&
      {*}\ar[d,"1"]\\
      M\ar[r]\&
      {[0,\infty)}\rlap,
    \end{tikzcd}
    &&
    \begin{tikzcd}[ampersand replacement=\&]
      M^{2}\ar[r]\ar[d,"m"']\&
      {[0,\infty)^{2}}\ar[d,"m"]\\
      M\ar[r]\&
      {[0,\infty)}\rlap,
    \end{tikzcd}
  \end{align*}
  where \(m\) is the multiplication map.
  Since \(1\colon{*}\to[0,\infty)\)
  corresponds to \(D_{1}\otimes E_{1}\)
  in the universal case,
  we see that the commutativity of the first square
  is equivalent to \(\{1\}\subset D_{1}\) and \(\{1\}\subset E_{1}\).
  We consider the second square.
Similarly,
  we can match the rest
  since
  \begin{align*}
    m^{*}(D_{t})&=\bigvee_{rs\leq t}D_{r}\boxtimes D_{s},&
    m^{*}(E_{t})&=\bigvee_{rs\leq t}E_{r}\boxtimes E_{s}
  \end{align*}
  holds in the universal case.
\end{proof}

\begin{lemma}\label{xyvgm4}
  Let \(R\) be a ring stack.
  Then an absolute semivalue on~\(R\)
  is a datum as in \cref{xs6zl2}
  satisfying
  \(\{0\}\subset D_{0}\) and \(D_{r}\cdot D_{s}\subset D_{r+s}\).
  It is an absolute value
  if furthermore \(\{0\}=R\setminus R^{\times}=D_{0}\)
  is satisfied.
\end{lemma}

\begin{proof}
  This is straightforward.
\end{proof}

\section{Descent}\label{s:des}

We develop an axiomatic theory of descent.
In this paper,
it serves not only as a prerequisite for the statements of \cref{main_3},
but also as a powerful ingredient in the proof of \cref{main_2}.

After introducing the notion in \cref{ss:des},
we establish its basic properties in \cref{ss:d_basic}.
In \cref{ss:d_pr}, we show that, in the prim case,
the condition admits a simpler formulation,
which connects to the notion of descendability,
which we discuss in \cref{ss:m_des}.
In \cref{ss:d_sp} we return to the theory of stacks to provide some applications.
We see many examples
of descent in \cref{s:d_ex}.

\subsection{Descent}\label{ss:des}

We first introduce the following condition:

\begin{definition}\label{xdsxpl}
  In a presentably symmetric monoidal \(n\)-category,
  we say that
  an \(\E_{0}\)-coalgebra,
  i.e., a morphism \(C\to\unit\),
  is \emph{faithful}
  if the augmented semisimplicial object \(C^{\otimes\bullet+1}\)
  is a colimit diagram.
\end{definition}

\begin{lemma}\label{xpx5dm}
In the situation of \cref{xdsxpl},
  when \(C\) is an \(\E_{1}\)-coalgebra,
  the geometric realization of~\(C^{\otimes\bullet+1}\)
  is an idempotent coalgebra.
\end{lemma}

\begin{proof}
  This comes from the fact
  that \(C^{\otimes\bullet+2}\)
  is a split augmented simplicial object.
\end{proof}

\begin{proposition}\label{xq4y2d}
  Let \(A\) be a commutative algebra object
  in a presentably symmetric monoidal \(n\)-category~\(\cat{C}\).
  We consider the following conditions:
  \begin{conenum}
    \item\label{i:hicone}
      The morphism \(\lvert\Mod_{A^{\otimes\bullet+1}}(\cat{C})\rvert\to\cat{C}\)
      (where transitions are forgetful functors)
      admits a fully faithful right adjoint in \(\Mod_{\cat{C}}(n\Cat{Pr})\).
    \item\label{i:hico}
      The morphism in \cref{i:hicone} is an equivalence,
      i.e., \(\Mod_{A}(\cat{C})\to\cat{C}\)
      is faithful in \(\Mod_{\cat{C}}(n\Cat{Pr})\)
      in the sense of~\cref{xdsxpl}.
  \end{conenum}
  Then they are equivalent.
  If \(A\) underlies a dualizable object in~\(\cat{C}\),
  they are furthermore equivalent to the following:
  \begin{conenum}[resume]
    \item\label{i:co}
      The dual \(A^{\vee}\) is faithful in the sense of~\cref{xdsxpl}.
  \end{conenum}
\end{proposition}

\begin{proof}
  Note that \(\text{\cref{i:hico}}\Rightarrow\text{\cref{i:hicone}}\)
  is trivial. We consider the converse.
By \cref{xpx5dm},
  the colimit, which we write~\(\cat{D}\)
  is a coidempotent algebra in \(\Mod_{\cat{C}}(n\Cat{Pr})\).
  We claim that \(F\colon\cat{D}\to\cat{C}\) is
  an equivalence under~\cref{i:hicone}.
  Let \(G\) be the right adjoint,
  which is fully faithful by assumption.
  It suffices to show that
  \({\id}\to GF\) is an equivalence.
  After tensoring with~\(\cat{D}\) over~\(\cat{C}\),
  this is an equivalence, since \(\cat{D}\) is coidempotent.
  However,
  this morphism is equivalent to the original one,
  since \(\cat{D}\) is coidempotent.

  We now assume that \(A\) is dualizable as an object of~\(\cat{C}\).
  The condition~\cref{i:hico}
  asks the morphism
  \(\lvert\Mod_{A^{\otimes\bullet+1}}(\cat{C})^{\vee}\rvert\to\cat{C}\)
  is an equivalence.
  By the self-duality of \(\Mod_{B}(\cat{C})\)
  for any commutative algebra object~\(B\),
  this is equivalent to asking
  \(L\colon\lvert\Mod_{A^{\otimes\bullet+1}}(\cat{C})\rvert\to\unit\)
  is an equivalence,
  where the transitions are forgetful functors.
  Now, we forget degeneracies to consider the semisimplicial diagram instead.
  Since \(A\) is dualizable,
  every transition admits a right adjoint
  in \(\Mod_{\cat{C}}(n\Cat{Pr})\);
  here, forgetting degeneracies is crucial.
  This means that
  the source of~\(L\) can be computed as a limit
  and
  \(L\) admits a right adjoint~\(R\) in \(\Mod_{\cat{C}}(n\Cat{Pr})\),
  which can be described as \(C\mapsto((A^{\vee})^{\otimes\bullet+1}\otimes C)\).
  With this limit description,
  \(L\) can be expressed as
  \((M_{\bullet})\mapsto\lvert M_{\bullet}\rvert\).
  We now consider the condition
  that~\(\epsilon\colon LR\to{\id}\) is an equivalence,
  which is exactly~\cref{i:hicone}.
  This means that
  \(\lvert(A^{\vee})^{\otimes\bullet+1}\otimes C\rvert\to C\)
  is an equivalence
  for any \(C\in\cat{C}\),
  and this condition is equivalent to~\cref{i:co}.
\end{proof}

\begin{definition}\label{x8uduj}
  Let \(\cat{C}\) be a presentably symmetric monoidal \(n\)-category
  and \(A\) a commutative algebra object.
  We say that \(A\) satisfies \emph{descent}
  if~\cref{i:hico} of \cref{xq4y2d} holds.
  We say that a morphism \(A\to B\)
  satisfies \emph{descent}
  if \(B\) satisfies descent as a commutative algebra object
  of \(\Mod_{A}(\cat{C})\).
\end{definition}

We record the following
from \cref{xq4y2d}
and its proof:

\begin{corollary}\label{x6ruic}
  In the situation of \cref{x8uduj},
  the following are equivalent:
  \begin{conenum}
    \item
      The commutative algebra object~\(A\) satisfies descent.
    \item
      The commutative algebra object~\(\Mod_{A}(\cat{C})\)
      in \(\Mod_{\cat{C}}(n\Cat{Pr})\)
      satisfies descent.
    \item
      The functor
      \(\Mod_{\cat{C}}(n\Cat{Pr})\to\Tot(\Mod_{\Mod_{A^{\otimes\bullet+1}}(\cat{C})}(n\Cat{Pr}))\)
      is fully faithful.
    \item
      The functor
      \(\Mod_{\cat{C}}(n\Cat{Pr})\to\Tot(\Mod_{\Mod_{A^{\otimes\bullet+1}}(\cat{C})}(n\Cat{Pr}))\)
      is an equivalence.
    \item
      For any \(\cat{C}\)-module~\(\cat{M}\),
      the morphism
      \(
      \cat{M}\to\Tot(\Mod_{A^{\otimes\bullet+1}}(\cat{M}))
      \)
      is an equivalence.
  \end{conenum}
  Note that all totalizations above exist
  for any~\(A\).
\end{corollary}

We conclude this section with the following criterion
of descent
for suave maps;
cf.~\cite[Proposition~6.19]{SixFunctors}.
We treat prim descent in \cref{ss:d_pr} below.

\begin{proposition}\label{d_sua}
  Let \(\cat{C}\) be a presentably symmetric monoidal \((n+1)\)-category
  such that \(\End_{\cat{C}}(\unit)\) is presentable.
  Let \(A\) be a commutative algebra object
  in~\(\cat{C}\) such that
  \(f^{*}\colon\unit\to A\) admits a left adjoint~\(f_{\natural}\)
  and \(A\) is dualizable.
  Then \(A^{\vee}\) is faithful if and only if
  \((f^{*})^{\vee}(f_{\natural})^{\vee}\)
  is faithful in \(\End_{\cat{C}}(\unit)\).
\end{proposition}

\begin{proof}
By arguing as in the proof of \(\text{\cref{i:hicone}}\Rightarrow\text{\cref{i:co}}\)
  in \cref{xq4y2d},
  we obtain the “only if” direction.

  We prove the “if” direction.
  By the dualizability of~\(A\),
  we obtain a cocommutative coalgebra~\(A^{\vee}\).
  Therefore,
  the faithfulness condition
  is about the augmented simplicial object
  \((A^{\vee})^{\otimes\bullet+1}\)
  being a colimit diagram.
  By arguing as in the proof of
  \(\text{\cref{i:co}}\Rightarrow\text{\cref{i:hico}}\)
  in \cref{xq4y2d},
  we obtain the desired result.
\end{proof}

Since the condition of descent
involves higher categories,
it is in general not true that
we can impose it universally.
Still, note the following,
which is the case we need in \cref{main_3}:

\begin{lemma}\label{d_okay}
  Let \(\cat{C}\) be a presentably symmetric monoidal \(n\)-category
  and \(A\to B\) a morphism of commutative algebra objects.
  Then there exists a universal morphism \(\cat{C}\to\cat{D}\)
  that makes this morphism satisfy descent.
\end{lemma}

\begin{proof}
  This directly follows from \cref{xq4y2d}:
  Imposing that \(B^{\vee}_{A}\) is faithful
  in~\(\Mod_{A}(\cat{C})\)
  is equivalent to imposing that
  \((B^{\vee}_{A})^{\otimes\bullet+1}\)
  is a colimit diagram in~\(\cat{C}\).
  Hence,
  we can localize~\(\cat{C}\) along this condition.
\end{proof}

\begin{remark}\label{x0jjry}
  In general,
  if we wish to impose descent,
  one must sacrifice affineness:
  When we work in eventually affine presentable categorical spectra
  (see~\cite[Remark~2.9]{n-pr}),
  we can always impose descent
  by losing affineness by~\(1\).
\end{remark}

\subsection{Basic properties}\label{ss:d_basic}

We then record some basic permanence properties:

\begin{lemma}\label{xcpiim}
  Let \(\cat{C}\to\cat{D}\) be a morphism
  of presentably symmetric monoidal \(n\)-categories.
  Then it maps a commutative algebra object satisfying descent
  to another such object.
\end{lemma}

\begin{proof}
  This directly follows from the definition.
\end{proof}

\begin{lemma}\label{xhrdzm}
  Let \(\cat{C}\) be a presentably symmetric monoidal \(n\)-category
  and \(A\) and~\(B\) commutative algebra objects.
  \begin{enumerate}
    \item\label{i:d_bc}
      When \(A\) satisfies descent,
      so does \(B\to A\otimes B\).
    \item\label{i:d_cbc}
      When \(B\) and \(B\to A\otimes B\) satisfy descent,
      so does~\(A\).
  \end{enumerate}
\end{lemma}

\begin{proof}
  First,
  \cref{i:d_bc} follows from \cref{xcpiim}.

  We consider~\cref{i:d_cbc}.
  We wish to show that \(\Mod_{A}(\cat{C})\) is faithful
  in \(\Mod_{\cat{C}}(n\Cat{Pr})\).
  Since \(B\) satisfies descent,
  it suffices to show that
  \(\Mod_{A\otimes B}(\cat{C})\)
  is faithful in \(\Mod_{\Mod_{B}(\cat{C})}(n\Cat{Pr})\),
  which follows from the assumption
  that \(B\to A\otimes B\) satisfies descent.
\end{proof}

\begin{lemma}\label{x4xcqq}
  Let \(\cat{C}\) be a presentably symmetric monoidal \(n\)-category
  and \(A\to B\) a morphism of commutative algebra objects.
  \begin{enumerate}
    \item\label{i:d_comp}
      If \(A\) and \(A\to B\) satisfy descent,
      then so does~\(B\).
    \item\label{i:d_canc}
      When \(B\) satisfies descent,
      so does~\(A\).
  \end{enumerate}
\end{lemma}

\begin{proof}
  We first prove \cref{i:d_canc}.
  By~\cref{i:d_cbc} of \cref{xhrdzm},
  it suffices to prove that
  \(B\to B\otimes A\) satisfies descent.
  This morphism admits a retraction
  and hence satisfies descent by \cref{x059gc} below.

  We then prove \cref{i:d_comp}.
  By~\cref{i:d_cbc} of \cref{xhrdzm},
  we need to show that \(A\to A\otimes B\) satisfies descent.
  Since the composite \(A\to A\otimes B\to B\)
  satisfies descent,
  the desired result follows from \cref{i:d_canc}.
\end{proof}

For the following,
we recall the notion of \emph{\(n\)-retract}
in an \(\infty\)-category,
which was introduced as \emph{\(n\)-condensation}
by Gaiotto–Johnson-Freyd~\cite{GJF}
(see also~\cite{ReutterZetto}):
Consider a pair of morphisms \(s\colon C\to D\)
and \(r\colon D\to C\).
We say that \(s\) is a \emph{\(0\)-section}
and \(r\) is a \emph{\(0\)-retraction}
if they are equivalences.
For \(n\geq0\),
we say that \(s\) is an \emph{\((n+1)\)-section}
and \(r\) is an \emph{\((n+1)\)-retraction}
if there is an \(n\)-section \({\id}\to rs\)
and an \(n\)-retraction \(rs\to{\id}\).

\begin{lemma}\label{x059gc}
  Let \(\cat{C}\) be a presentably symmetric monoidal \(n\)-category
  and \(A\) a commutative algebra object.
  When the underlying morphism of \(\unit\to A\)
  in~\(\cat{C}\)
  admits an \(n\)-retraction,
  it satisfies descent.
\end{lemma}

\begin{proof}
  Note that
  the morphism \(\cat{C}\to\Mod_{A}(\cat{C})\)
  in \(\Mod_{\cat{C}}(n\Cat{Pr})\)
  admits an \((n+1)\)-retraction.
  Hence,
  by replacing \(A\) with \(\Mod_{A}\cat{C}\),
  we assume that \(A\) is rigid.
  Now, we need to check
  that the augmented simplicial diagram \(A^{\otimes\bullet+1}\)
  is a colimit diagram.
  This is an \(n\)-retract
  of \(A\otimes A^{\otimes\bullet+1}\),
  which is a colimit diagram.
  Therefore,
  the desired claim follows.
\end{proof}

\subsection{Prim descent}\label{ss:d_pr}

By definition,
to check descent,
we need to go higher.
Here,
we consider how to go lower.

\begin{lemma}\label{xaf8qg}
  Let \(\cat{C}\) be a presentably symmetric monoidal \((n+1)\)-category
  such that \(\End_{\cat{C}}(\unit)\) is presentable.
  Consider a commutative algebra object~\(A\) in \(\End_{\cat{C}}(\unit)\).
  Then the following are equivalent:
  \begin{conenum}
    \item\label{i:dm_low}
      The commutative algebra~\(A\)
      satisfies descent in \(\End_{\cat{C}}(\unit)\).
    \item\label{i:dm_high}
      The commutative algebra \(\Mod_{A}(\unit)\)
      satisfies descent in~\(\cat{C}\).
  \end{conenum}
\end{lemma}

\begin{proof}
  Consider the full faithful embedding
  \(\Mod_{\End_{\cat{C}}(\unit)}(n\Cat{Pr})
  \to\cat{C}\).
  Then \cref{i:dm_low} is equivalent to
  the condition that
  the augmented simplicial diagram
  \(\Mod_{A^{\bullet+1}}(\cat{C})\)
  is a colimit diagram in the source.
  Then \cref{i:dm_high}
  is equivalent to the same condition
  in the target.
\end{proof}

\begin{theorem}\label{xpbfpq}
  Let \(\cat{C}\) be a presentably symmetric monoidal \((n+1)\)-category
  such that \(\End_{\cat{C}}(\unit)\) is presentable.
  Consider a commutative algebra object~\(A\)
  such that the unit \(u\colon\unit\to A\)
  admits a right adjoint~\(u^{\R}\).
  Then the following are equivalent:
  \begin{conenum}
    \item\label{i:d_ori}
      The commutative algebra~\(A\) in~\(\cat{C}\)
      satisfies descent.
    \item\label{i:d_low}
      The commutative algebra~\(u^{\R}u\)
      in \(\End_{\cat{C}}(\unit)\)
      satisfies descent.
  \end{conenum}
\end{theorem}

\begin{proof}
  We decompose~\(u\) as
  \(\unit\to\Mod_{u^{\R}u}(\unit)\to A\).
  By \cref{xaf8qg},
  the condition~\cref{i:d_low}
  is equivalent to the condition that
  the first map satisfies descent.
  Hence,
  by \cref{x4xcqq},
  it suffices to observe that
  \(\Mod_{u^{\R}u}(\unit)\to A\) satisfies descent,
  which follows from \cref{x059gc}.
\end{proof}

\begin{example}\label{xyu2uy}
  Let \(f\colon Y\to X\) be a morphism
  of quasicompact quasiseparated (spectral) schemes.
  Then \(\D(X)\to\D(Y)\) satisfies descent in \(\Cat{Pr}_{\st}\)
  if and only if \(f_{*}(\unit_{Y})\) satisfies descent
  in~\(\D(X)\).
\end{example}

\subsection{Descendability}\label{ss:m_des}

As we see in \cref{xpbfpq},
in a nice situation,
we can reduce the descent condition
to something lower categorical.
The following notion of Mathew~\cite{Mathew16}
(see also~\cite{Balmer16})
is useful in such a situation:

\begin{definition}\label{xwlisj}
  Let \(\cat{C}\) be a stable presentably symmetric monoidal category
  and \(A\) a commutative algebra object.
  We say that \(A\) is \emph{descendable}
  if the \(\Pro\)-object
  associated to the Adams tower
  is equivalent to the constant \(\Pro\)-object~\(\unit\).
  By~\cite[Lemma~11.20]{BhattScholze17},
  it is equivalent to
  the condition that
  the tautological map
  \(I^{\otimes d}\to\unit\) is zero
  for some \(d\geq0\),
  where \(I=\fib(\unit\to A)\).
  In this case,
  we say that \(A\)
  is descendable of index \(\leq d\).
\end{definition}

First,
we observe the following (cf.~\cite[Corollary~3.42]{Mathew16}):

\begin{proposition}\label{x670j7}
  In the situation of \cref{xwlisj},
  if \(A\) is descendable,
  it satisfies descent.
\end{proposition}

\begin{proof}
  We write \(I\) for \(\fib(\unit\to A)\).
  By \cref{x6ruic},
  it suffices to show that
  for any \(\cat{C}\)-module~\(\cat{M}\),
  the morphism
  \begin{equation*}
    \cat{M}
    \to\Tot\bigl(\Mod_{A^{\otimes\bullet+1}}(\cat{M})\bigr),
  \end{equation*}
  where the transitions are base change,
  is an equivalence.
  By~\cite[Corollary~4.7.5.3]{LurieHA},
  it suffices to see that
  \(F\colon\cat{M}\to\Mod_{A}(\cat{M})\)
  is conservative and preserves \(F\)-split totalizations.

  First,
  we prove conservativity.
  Suppose that \(A\otimes M\simeq0\)
  for \(M\in\cat{M}\).
  Then \(I\otimes M\to M\) is an equivalence,
  but then \(I^{\otimes n}\otimes M\to M\)
  is zero and an equivalence for \(n\gg0\)
  and therefore \(M\simeq0\).

  We then consider an \(F\)-split
  cosimplicial object \(M\colon\Cat{\Delta}\to\cat{M}\).
  We claim that
  the \(\Pro\)-object associated to~\(M\) is constant.
  Similarly,
  in this case,
  the cofiber of \(I^{n}\otimes M\to M\)
  satisfies that condition for \(n\geq0\).
  For \(n\gg0\),
  this contains~\(M\) as a direct summand.
\end{proof}

The converse is false:

\begin{example}\label{xxb52c}
  Let \(R_{n}\to S_{n}\)
  be a morphism of (static) rings
  that is descendable
  but \(\fib(R_{n}\to S_{n})^{\otimes n}\to R_{n}\)
  is nonzero;
  see~\cite{Zelich} or~\cite{n-stone} for a construction.
  We consider nonquasicompact schemes \(X=\coprod_{n}\Spec R_{n}\)
  and \(Y=\coprod_{n}\Spec S_{n}\)
  and the induced morphism \(f\colon Y\to X\).
  Then
  by \cref{x670j7}, \(f_{*}(\unit_{Y_{n})}\in\D(X_{n})\) satisfies
  descent for each~\(n\),
  and hence so does \(f_{*}(\unit_{Y})\in\D(X)\).
  However, it is not descendable by assumption.
\end{example}

\begin{example}\label{xg2rlr}
  Fix a prime~\(p\)
  and consider \(\Fun(BC_{p},\D(\FF_{p}))\).
  Then the regular representation \(\FF_{p}[C_{p}]\)
  is not descendable.
  This is because of a morphism
  \(\FF_{p}[-1]\to\fib(\FF_{p}\to\FF_{p}[C_{p}])\)
  corresponding to the generator
  of the cohomology ring \(H^{*}(BC_{p};\FF_{p})\),
  whose power does not vanish.
  However,
  by \cref{xapqpz},
  it still satisfies descent.
\end{example}

The units in both examples above
are not compact.
In fact, this is the only obstruction:

\begin{proposition}\label{x765rh}
  In the situation of \cref{xwlisj},
  suppose that \(\unit\in\cat{C}\) is compact\footnote{Or more generally,
    it is sufficient to assume that
    it is sequentially compact,
    i.e.,
    \(\Map(\unit,\X)\) commutes with sequential colimits.
  }.
  Then if \(A\) satisfies descent, it is descendable.
\end{proposition}

\begin{proof}
We write~\((T^{n})_{n}\) for the Adams tower of~\(A\),
  and we wish to prove its constancy.
  We write \([\X,\X]\)
  for the internal mapping object in~\(\cat{C}\).
  For \(C\in\cat{C}\),
  we have
  \begin{equation*}
    \injlim_{n}\Map_{\cat{C}}(T^{n},C)
    \simeq
    \injlim_{n}\Map_{\cat{C}}(\unit,[T^{n},C])
    \simeq
    \Map_{\cat{C}}(\unit,\injlim_{n}[T^{n},C])
    \simeq
    \Map_{\cat{C}}(\unit,\lvert[A^{\otimes\bullet+1},C]\rvert)
  \end{equation*}
  by the (sequential) compactness of~\(C\).
  Hence,
  it suffices to show
  that the morphism
  \begin{equation}
    \label{e:90yld}
    \lvert[A^{\otimes\bullet+1},C]\rvert
    \to C
  \end{equation}
  is an equivalence
  for any \(C\in\cat{C}\).

  We take a regular cardinal~\(\kappa\)
  satisfying the following:
  \begin{itemize}
    \item
      The presentably symmetric monoidal category~\(\cat{C}\) 
      belongs to \(\CAlg(\Cat{Pr}^{\kappa})\).
    \item
      The functor \(C\mapsto[A,C]\)
      preserves \(\kappa\)-filtered colimits.
  \end{itemize}
  Under this,
  it suffices to verify that \cref{e:90yld}
  is an equivalence for \(C\in\cat{C}_{\kappa}\).

  We now take a regular cardinal \(\lambda\geq\kappa\)
  satisfying the following:
  \begin{itemize}
    \item
      The full subcategory~\(\cat{C}_{\lambda}\)
      is closed under \(\kappa\)-small limits.
    \item
      When \(C\) is \(\kappa\)-compact
      and \(D\) is \(\lambda\)-compact,
      \([C,D]\) is \(\lambda\)-compact.
  \end{itemize}
  We then consider \(\cat{D}=\Ind_{\kappa}((\cat{C}_{\lambda})^{\op})\)
  as a \(\cat{C}\)-module
  such that the action of \(\cat{C}_{\kappa}\) on~\(\cat{D}\)
  restricts to \((\cat{C}_{\lambda})^{\op}\)
  and is given by \((C,D)\mapsto[C,D]\).
  Since \(\cat{D}\to\Tot(\Mod_{A^{\otimes\bullet+1}}(\cat{D}))\)
  is an equivalence by assumption,
  we see that \(\lvert A^{\otimes\bullet+1}\otimes D\rvert\to D\)
  is an equivalence for \(D\in\cat{D}\).
  Unwinding this for \(D\in\cat{D}_{\kappa}=(\cat{C}_{\lambda})^{\op}\),
  we obtain the desired result.
\end{proof}

We conclude this section
by collecting some facts
about descendability:

\begin{lemma}\label{x50j16}
  Let \(\cat{C}\) be a stable presentably symmetric monoidal category
  and \(A\to B\) a morphism of commutative algebra objects.
  \begin{enumerate}
    \item\label{i:md_e}
      When \(B\) is descendable of index~\(\leq e\),
      then so is~\(A\).
    \item\label{i:md_de}
      When \(A\) is descendable of index~\(\leq d\)
      and \(B\) is descendable of index~\(\leq e\) in \(\Mod_{A}(\cat{C})\),
      then \(B\) is descendable of index~\(\leq de\).
  \end{enumerate}
\end{lemma}

\begin{proof}
  We write \(I\), \(J\), and~\(K\)
  for the fibers
  of \(\unit\to A\), \(\unit\to B\),
  and \(A\to B\), respectively.
  We have a cofiber sequence \(I\to J\to K\).
  Since \(I\to\unit\) factors through \(J\to\unit\),
  we obtain~\cref{i:md_e}.
Under the assumption of~\cref{i:md_de},
  we see that \(J^{\otimes e}\to\unit\to A\) is zero,
  which implies \(J^{\otimes e}\to\unit\) factors through~\(I\).
  Therefore, we see the desired claim.
\end{proof}

The following is from~\cite[Lemma~11.22]{BhattScholze17}:

\begin{lemma}\label{md_seq}
  Let \(\cat{C}\) be a stable presentably symmetric monoidal category
  and \(A_{0}\to\dotsb\)
  a sequence of commutative algebra objects.
  If \(A_{n}\) is descendable of index~\(\leq d\)
  for any~\(n\),
  then \(\injlim_{n}A_{n}\) is descendable of index~\(\leq 2d\).
\end{lemma}

\begin{proof}
  We write \(I_{n}\) for \(\fib(\unit\to A_{n})\).
  By the Milnor sequence
  \begin{equation*}
    0
    \to\sideset{}{^{1}}\projlim_{n}\pi_{1}\Map(I_{n}^{\otimes d},A)
    \to\pi_{0}\Map(I^{\otimes d},A)
    \to\projlim_{n}\pi_{0}\Map(I_{n}^{\otimes d},A)
    \to0,
  \end{equation*}
  the obstruction lies in \(\projlim_{n}^{1}\pi_{1}\Map(I_{n}^{\otimes d},A)\),
  which squares to~\(0\).
\end{proof}

The following is from~\cite[Lemma~5.4.6]{ABLBRCS}:

\begin{lemma}\label{md_str}
  Let \(\cat{C}\) be a stable presentably symmetric monoidal category
  and \(A\) a commutative algebra object.
  Consider an idempotent coalgebra~\(\unit_{U}\)
  with the complementary idempotent algebra~\(\unit_{Z}\),
  defining \(\cat{C}_{U}\) and \(\cat{C}_{Z}\).
  Suppose that \(A_{U}=A\otimes\unit_{U}\) and \(A_{Z}=A\otimes\unit_{Z}\)
  are descendable of index~\(\leq d_{U}\) and~\(\leq d_{Z}\)
  in~\(\cat{C}_{U}\) and~\(\cat{C}_{Z}\), respectively.
  Then \(A\) is descendable of index~\(\leq d_{U}+d_{Z}\).
\end{lemma}

\begin{proof}
  Since the composition
  \(F^{\otimes d_{Z}}\to\unit\to\unit_{Z}\) is zero,
  \(F^{\otimes d_{Z}}\to\unit\)
  factors through
  \(\unit_{U}\).
  Therefore,
  \(F^{\otimes d_{U}+d_{Z}}\to\unit\) is zero,
  since \(F^{\otimes d_{U}}_{U}\to\unit_{U}\) is zero.
\end{proof}

\begin{remark}\label{x7psvi}
  We often do not know
  whether
  the category of interest has compact unit,
  precluding the use of \cref{md_str}
  when we have descent on each component.
  Still,
  we can use \cref{d_sua}
  in some situations,
  and then stratify;
  cf.~the proof of \cref{xfrmql}.
\end{remark}

\subsection{Descending suave and prim maps}\label{ss:d_sp}

We go back to the context of \cref{s:stk};
we fix a presentably symmetric monoidal \(2\)-category~\(\cat{C}\).

\begin{definition}\label{x0wijg}
  We call a morphism of stacks \(Y\to X\) over~\(\cat{C}\) a \emph{cover}
  if \([X]\to[Y]\) satisfies descent.
\end{definition}

The following is an immediate consequence
of descent for (iterated) module categories:

\begin{proposition}\label{x9vxyu}
  Let \(Y\to X\) be a morphism of stacks.
  Let \(X'\to X\) be a cover.
  If the base change \(Y'\to X'\)
  is weakly suave,
  then so is \(Y\to X\).
  The same holds when we replace all “weakly suave” with “suave” or “prim.”
\end{proposition}

\begin{proposition}\label{sua_loc}
  Let \(Y\to X\) be a morphism of stacks.
  Let \(Y'\to Y\) be a weakly suave cover.
  If the composite \(Y'\to X\) is weakly suave,
  then \(Y\to X\) is weakly suave.
\end{proposition}

\begin{proof}
In this situation,
  \([Y]\) can be computed as the totalization of~\([Y'^{\times_{Y}\bullet+1}]\),
  which is also the \emph{semisimplicial} colimit along the left adjoint.
  This description gives us the desired left adjoint.
\end{proof}

\section{Examples of descent}\label{s:d_ex}

In this section,
we explain some examples of descent.
In \cref{ss:open,ss:closed},
we treat covers by open and closed subsets.
In \cref{ss:lch},
we consider surjections between compact Hausdorff spaces.
In \cref{ss:clausen},
we explain Clausen’s theorem,
which we use in the proof of \cref{main_4}.

\subsection{Open covers}\label{ss:open}

We note that this descent works even in the unstable case:

\begin{proposition}\label{x8jzva}
  Let \(X\) be a locale
  and \((U_{i})_{i\in I}\) be an open cover.
  Then the morphism
  \(\Shv(X)\to\prod_{i\in I}\Shv(U_{i})\)
  satisfies descent in~\(\Cat{Pr}\).
\end{proposition}

Combining this with \cref{xhrdzm},
we obtain a stronger claim:

\begin{corollary}\label{xw7hrq}
  Let \(Y\to X\) be a morphism of locales
  that admits a local section.
  Then \(\Shv(X)\to\Shv(Y)\) satisfies descent in \(\Cat{Pr}\).
\end{corollary}

\begin{proof}[Proof of \cref{x8jzva}]
  We write~\(Y\) for \(\coprod_{i\in I}U_{i}\),
  which admits a tautological map to~\(X\).
  Since \(\Shv(Y)\) is self-dual
  over \(\Shv(X)\) and \(\Shv(Y)^{\otimes_{\Shv(X)}2}
  \simeq\Shv(Y^{\times_{X}2})\)
  (see~\cite[Corollary~1.10]{ttg-sm}),
  it suffices to show that
  \(\Shv(Y)^{\otimes_{n}\bullet+1}
  \simeq\Shv(Y^{\times_{X}\bullet+1})\)
  is a colimit diagram in~\(\Cat{Pr}\).
  Since \(\Shv(Y^{\times_{X}n})\simeq\Shv(X)_{/Y^{\times_{X}n}}\),
  this follows from the fact
  that \(Y\to X\) is an effective epimorphism
  in the topos \(\Shv(X)\)
  and~\cite[Proposition~6.3.5.14]{LurieHTT}.
\end{proof}

\subsection{Finite closed covers}\label{ss:closed}

We recall the following:

\begin{definition}\label{xsjci8}
  Let \((Z_{i})_{i\in I}\)
  be a family of subsets of a set~\(X\).
  We say that it is of \emph{order} \(\leq d\),
  if for any subset \(I_{0}\subset I\) with cardinality \(>d\),
  we have \(\bigcap_{i\in I_{0}}Z_{i}=\emptyset\).
\end{definition}

\begin{proposition}\label{md_ci}
  Let \(X\) be a locale
  and \((Z_{i})_{i\in I}\) be a closed cover of order~\(\leq d\)
  then the commutative algebra \(\prod_{i\in I}\SS_{Z_{i}}\)
  in \(\Shv(X;\Cat{Sp})\) is descendable of index~\(\leq d\).
\end{proposition}

\begin{proof}
  We consider a filtration of~\(X\)
  by the closed subsets
  \(\emptyset=X_{0}\subset X_{1}\subset\dotsb\subset X_{d}=X\)
  such that \(X_{n}\) is 
  the subspace of points that are covered
  by at most~\(n\) of~\(Z_{i}\).
  Then over \(Z_{n}\setminus Z_{n-1}\),
  for \(1\leq n\leq d\),
  the commutative algebra admits a splitting.
  Therefore, the desired result follows from \cref{md_str}.
\end{proof}

\subsection{Surjections of compact Hausdorff spaces}\label{ss:lch}

We recall here that
any surjection between
finite-dimensional compact Hausdorff spaces
of countable weight satisfies descent.

\begin{proposition}\label{x496fa}
  The pushforward along the dyadic expansion
  \begin{align*}
    \{0,1\}^{\NN}&\to[0,1];&
    (x_{n})_{n}&\mapsto\sum_{n}\frac{x_{n}}{2^{n+1}}
  \end{align*}
  determines a descendable commutative algebra
  of index~\(\leq4\)
  in \(\Shv([0,1];\Cat{Sp})\).
\end{proposition}

\begin{proof}
  It is the limit of
  \begin{equation*}
    \coprod_{k=0}^{2^{n}-1}\biggl[\frac{k}{2^{n}},\frac{k+1}{2^{n}}\biggr]\to[0,1].
  \end{equation*}
  Since this is a closed cover of order~\(2\)
  by \cref{md_ci},
  the desired result follows from \cref{md_seq}.
\end{proof}

\begin{corollary}\label{x8izhs}
  Any nonempty compact Hausdorff space
  of countable weight
  of dimension \(\leq d\)
  admits a cover by a Cantor set
  that is of index \(2^{4d+2}\).
\end{corollary}

\begin{proof}
By embedding the space into~\([0,1]^{2d+1}\),
  the result follows from \cref{x496fa}.
\end{proof}

\begin{corollary}\label{x1lqb7}
  Consider a surjection \(Y\to X\)
  from a light compact Hausdorff space
  to a light compact Hausdorff space of dimension \(\leq d\).
  Then it is descendable of index \(\leq 2^{4d+3}\).
\end{corollary}

\begin{proof}
  By \cref{x8izhs},
  we obtain a map \(C\to X\) from a Cantor set~\(C\).
  Then we consider a surjection
  \(D\to C\times_{X}Y\) from a Cantor set~\(D\).
  By~\cref{i:md_e} of \cref{x50j16},
  it suffices to bound
  the descendability of \(D\to X\).
  The map \(D\to C\)
  is a sequential limit of surjections between finite sets.
  By \cref{md_seq}, this is descendable of index \(\leq2\).
\end{proof}

\begin{remark}\label{xzu7rw}
  The bound in \cref{x1lqb7} is suboptimal;
  e.g., one can construct
  a surjection from the Cantor set to~\([0,1]^{2}\)
  with descendability index \(\leq 6\).
\end{remark}

\begin{remark}\label{x65y5q}
  In \cref{x1lqb7},
  the finite-dimensionality assumption on~\(X\) cannot be dropped.
  When \(X=[0,1]^{\NN}\),
  we cannot find a cover from a profinite set,
  since even the usual descent fails.
\end{remark}

\begin{remark}\label{xzad7y}
  In \cref{x1lqb7},
  even in the zero-dimensional case,
  the size restriction cannot be dropped;
  in general, a surjection
  between profinite sets
  is not a cover
  over~\(\Cat{Pr}_{\st}\),
  as shown in~\cite{n-stone}.
\end{remark}

\subsection{Condensed animas}\label{ss:clausen}

We review Clausen’s pointwise descendability criterion
in~\cite[Section~6]{Clausen}.
First we need to recall the following definition of
Clausen–Scholze~\cite{AnaSta}:

\begin{definition}\label{clasch}
  We call a profinite set \emph{light}
  if it is of countable weight.
  We write \(\Cat{PFin}_{\lgt}\) for the category
  of light profinite sets.
  We consider the Grothendieck topology on it
  generated by finite jointly surjective families.
  A \emph{(light) condensed anima}
  is a hypersheaf on \(\Cat{PFin}_{\lgt}\).
  We write \(\Cat{ConAni}_{\lgt}\) for this category.
\end{definition}

\begin{definition}\label{x8746d}
  Let \(X\) be a condensed anima.
  We write \(\Shv^{\wedge}(X;\Cat{Sp})\)
  for what is obtained by applying
  descent to \(X\mapsto\Shv(X;\Cat{Sp})\)
  for light profinite sets~\(X\).
  We define \(\Shv^{\wedge}(X;\D(\ZZ))\) similarly.
\end{definition}

\begin{example}\label{xdhy4g}
  By~\cite[Corollary~2.8]{Haine},
  when \(X\) is a light compact Hausdorff space,
  this is the category of Postnikov-complete sheaves,
  hence the notation.
\end{example}

We recall the following:

\begin{definition}\label{x1vxz5}
  Let \(X\) be a condensed anima.
  Its \emph{cohomological dimension}
  \(\dim_{\ZZ}X\) is
  the smallest integer~\(d\) (if it exists)
  such that for any \(M\in\Shv^{\wedge}(X;\D(\ZZ))\)
  concentrated in degree~\(0\),
  the cohomology \(\Gamma(X;M)\) belongs to \(\D(\ZZ)_{\geq-d}\).
  For a profinite group~\(G\),
  we simply call \(\dim_{\ZZ}(BG)\)
  the cohomological dimension of~\(G\).
\end{definition}

\begin{remark}\label{xi27fj}
  What is defined in \cref{x1vxz5}
  is often called the \emph{strict cohomological dimension}
  in the literature on profinite groups,
  where cohomological dimension
  can mean the dimension with respect to torsion sheaves.
  By \cref{xpyglx} below,
  we can relate these notions.
\end{remark}

\begin{lemma}\label{xpyglx}
  When \(\Gamma(X;\X)\colon\Shv^{\wedge}(X;\D(\ZZ))\to\D(\ZZ)\)
  preserves filtered colimits,
  \(\dim_{\ZZ}X\leq\max(\dim_{\QQ}X,\sup_{p}(\dim_{\FF_{p}}X)+1)\)
  holds.
\end{lemma}

\begin{proof}
  We assume that the right-hand side is finite
  and write it as~\(d\).
  We suppose the contrary so that there is \(i>d\)
  and~\(M\) such that \(H^{i}(X;M)\neq0\).
  We take a prime~\(p\),
  and we write \(pM\) for the image
  of \(p\colon M\to M\).
  We then consider short exact sequences
  \(0\to M[p]\to M\to pM\to 0\)
  and \(0\to pM\to M\to M/p\to 0\).
  Now \(M[p]\) and \(M/p\) are \(\FF_{p}\)-vector spaces,
  and hence
  for \(i>d\),
  the morphism \(H^{i}(X;M)\to H^{i}(X;pM)\)
  and \(H^{i}(X;pM)\to H^{i}(X;M)\)
  are injective.
  Therefore, \(p\) acts injectively on \(H^{i}(X;M)\).
  Since this holds for any prime~\(p\),
  we see \(H^{i}(X;M)\otimes\QQ=H^{i}(X;M\otimes\QQ)\neq0\),
  which is a contradiction.
\end{proof}

\begin{example}\label{xhcvbd}
  Let \(G\) be a free profinite group.
  By freeness,
  the torsion cohomology of~\(G\) vanishes in degree \(\geq2\).
  Hence, by \cref{xi27fj},
  we obtain \(\dim_{\ZZ}(BG)\leq2\).
  This inequality is sharp except for the trivial case,
  since \(H^{2}(\hat{\ZZ};\ZZ)=\QQ/\ZZ\neq0\).
\end{example}

\begin{theorem}[Clausen]\label{clausen}
  Let \(X\) be a \(1\)-truncated light profinite anima
  of finite cohomological dimension;
  i.e., the cohomological dimension of~\(G_{x}\) is uniformly bounded.
  A commutative algebra object~\(A\)
  in \(\Shv^{\wedge}(X;\Cat{Sp})\)
  is descendable
  if and only if there is \(d\) such that
  \(x^{*}(A)\) is descendable of index~\(\leq d\) in \(\Cat{Sp}\)
  for any \(x\colon{*}\to X\).
\end{theorem}

\begin{proof}
  This is a special case of~\cite[Theorem~6.21]{Clausen}.
  Indeed,
  its assumptions
  are satisfied by~\cite[Examples~6.14.2 and 6.19.3]{Clausen}.
\end{proof}

\Cref{clausen} is only useful
in the situation where descendability works.
Here we note the following easy example
(cf.~\cref{xg2rlr}):

\begin{proposition}\label{xapqpz}
  Let \(G\) be a finite group.
  Then \(\Fun(BG;\Cat{Sp})=\Shv^{\wedge}(BG;\Cat{Sp})\to\Cat{Sp}\)
  satisfies descent in \(2\Cat{Pr}_{\st}\).
\end{proposition}

\begin{proof}
We can apply \cref{d_sua},
  since
  both
  \(\Fun(BG,\Cat{Sp})\to\Cat{Sp}\)
  and
  \(\Fun(G,\Cat{Sp})\to\Cat{Sp}\)
  admit left adjoints.
  By unwinding the definition,
  it is reduced to showing that
  the regular representation is faithful in \(\Fun(BG,\Cat{Sp})\).
  We can compute the colimit in~\(\Cat{Sp}\) instead,
  and there we have a splitting.
\end{proof}

\section{Algebraic \texorpdfstring{\(2\)}{2}-motives and ring stacks}\label{s:main_alg}

In this section,
we prove \cref{main_2},
which characterizes \(\SH[2](\ZZ)\)
in terms of ring stacks.
In \cref{ss:aff},
we first translate \cref{main_1}
in a form suitable for our use.
In \cref{ss:sh_six},
we first construct six operations
in the universal recipient.
In \cref{ss:before},
we perform an axiomatic argument
showing that we get a class of smooth morphisms
in the universal target.
In \cref{ss:cc_alg}, we prove \cref{main_1}.
In \cref{ss:cdh},
we prove that cdh~descent holds for algebraic \(2\)-motives.

\subsection{\texorpdfstring{\(2\)}{2}-motives of rings}\label{ss:aff}

To prove \cref{main_2},
we need to have a version of \cref{main_1}
that uses (static) rings in place
of quasiprojective static varieties:

\begin{theorem}\label{sh_ring}
  We write \(\Cat{Ring}^{\heartsuit}\) for the category of static rings.
  The morphism
  \([\X]\colon(\Cat{Ring}^{\heartsuit})_{\aleph_{0}}
  \to\SH[2](\ZZ)\)
  defined as \([\Spec(\X)]\)
  is the universal symmetric monoidal functor
  satisfying the following:
  \begin{conenum}
    \item\label{i:s_sm}
      When \(A\to B\) is smooth,
      \(f^{*}\colon[A]\to[B]\)
      has an \([A]\)-linear left adjoint \(f_{\natural}\)
      satisfying base change.
    \item\label{i:s_exc}
      For \(A\) and an element \(a\in A\),
      the \([A[a^{-1}]]\gets[A]\to[A/a]\)
      is a recollement.
    \item\label{i:s_hi}
      The morphism
      \({\id}\to f^{*}f_{\natural}\)
      is an equivalence for \(f\colon A\to A[T]\)
      for any~\(A\).
    \item\label{i:s_tate}
      For \(f\colon A\to A[T]\)
      and \(s\colon A[T]\to A\) given by \(T\mapsto0\),
      the morphism
      \(f_{\natural}s_{*}\)
      is an equivalence.
  \end{conenum}
\end{theorem}

We write \(\SH[2]'(\ZZ)\) for the universal target of \cref{sh_ring}.
By \cref{main_1},
we obtain a morphism \(F\colon{\SH[2]'(\ZZ)}\to{\SH[2](\ZZ)}\).

\begin{lemma}\label{xdp5rb}
  For \(A'\gets A\to B\) in \((\Cat{Ring}^{\heartsuit})_{\aleph_{0}}\),
  the morphism \([A']\otimes_{[A]}[B]\to[A'\otimes_{A}B]\)
  is an equivalence in \(\SH[2]'(\ZZ)\).
\end{lemma}

\begin{proof}
  This follows from \cref{x15q1p};
  any closed immersion
  is a finite composition of principal closed immersions.
\end{proof}

\begin{lemma}\label{x02mj4}
  A Zariski cover \(A\to B\) in \((\Cat{Ring}^{\heartsuit})_{\aleph_{0}}\)
  determines a cover of stacks over \(\SH[2]'(\ZZ)\).
\end{lemma}

\begin{proof}
  We need to consider
  \(A\to\prod_{i=1}^{n}A[\sfrac1{a_{i}}]\)
  with \(a_{i}\) generating the unit ideal.
  Moreover, it is reduced to the cases \(n=0\) and~\(2\).
  The case \(n=0\) is clear,
  so we consider the case \(n=2\).

  By using \cref{xwog77} and \cref{i:s_exc},
  we see that
  the morphism
  \([A]\to[A[\sfrac1{a_{1}}]]\times_{[A[\sfrac1{a_{1}a_{2}}]]}[A[\sfrac1{a_{2}}]]\)
  is an equivalence.

  We then use \cref{d_sua}.
  By the limit formula above,
  it suffices to see that
  it is a cover on \(A[\sfrac1{a_{1}}]\)
  and on \(A[\sfrac1{a_{2}}]\),
  which is clear.
\end{proof}

\begin{proof}[Proof of \cref{sh_ring}]
  We construct \(G\colon\SH[2](\ZZ)\to\SH[2]'(\ZZ)\).
  We right Kan extend
  \(
  (\Cat{Ring}^{\heartsuit})_{\aleph_{0}}
  \to\SH[2]'(\ZZ)\),
  which is possible by~\cite[Theorem~D]{n-pr}\footnote{We can avoid using this,
    since it suffices to extend to~\(\Cat{QProj}\),
    and we can see the existence of limits concretely.
  },
  to obtain
  \(\PShv(((\Cat{Ring}^{\heartsuit})_{\aleph_{0}})^{\op})
  \to\SH[2]'(\ZZ)\).
  By \cref{x02mj4},
  this factors through
  the category of Zariski sheaves.
  We verify that
  its restriction to \(\Cat{QProj}\)
  satisfies the conditions in \cref{main_1}
  when \(S=\Spec\ZZ\).

  We prove that
  in the situation \(U\to X\gets Z\)
  as in \cref{i:a_exc} of \cref{main_1},
  we have an \([X]\)-linear recollement
  and
  smooth base change for \(U\hookrightarrow X\)
  (hence proper base change for \(Z\hookrightarrow X\)).
  By assumption,
  we only know this
  when \(X\) is affine and \(U\) is a principal open.

  We first see this
  when \(X\) is affine
  and \(U\) is quasiaffine inside~\(X\).
  This can be seen by induction
  on the number of affines needed to cover~\(U\).

  We then see this in general.
  This is done by writing~\(X\)
  as a colimit of \(\Spec A\)
  (as a Zariski sheaf)
  and using smooth base change for open immersions
  and proper base change for closed immersions.

  Therefore,
  we obtain \cref{i:a_exc}.
  By running the same argument as in the proof of \cref{x02mj4},
  we see Zariski descent.
  In particular,
  for~\(X\),
  there is \(\Spec A\to X\)
  such that \([X]\to[A]\) satisfies descent.
  This implies
  \cref{i:a_hi,i:a_tate}.
  For~\cref{i:a_sm},
  by the relative Künneth formula,
  which follows from \cref{x15q1p},
  we need to show that
  any smooth morphism
  \(Y\to\Spec A\) induces a weakly suave map.
  By \cref{sua_loc},
  we can reduce to the case when \(Y\) is affine too.

  Now we have constructed~\(G\).
  By construction,
  \(GF\) is homotopic to~\(\id\).
  The other equivalence
  follows from
  the fact that \(\Cat{QProj}^{\op}\to\SH[2](\ZZ)\)
  is a Zariski sheaf,
  which follows from the argument of \cref{x02mj4}.
\end{proof}

\subsection{Six operations}\label{ss:sh_six}

We now come back to proving \cref{main_2}.

\begin{definition}\label{xjxm6b}
  We write \(\Mot[2](\ZZ)\)
  for the universal target.
  By definition,
  it comes with a colimit-preserving functor
  \(\Cat{Ring}\to\CAlg(\Mot[2](\ZZ))\).
  We write~\([\X]\)
  for this.
\end{definition}

\begin{lemma}\label{x2mgqh}
  The functor in \cref{xjxm6b}
  factors through \(\Cat{Ring}^{\heartsuit}\).
\end{lemma}

\begin{proof}
  Let \(\ZZ[T_{d}]\)
  be the free animated ring
  generated by an element in degree~\(d\geq0\).
  We first prove that the tautological map \(\unit\simeq[\ZZ]\to[\ZZ[T_{d}]]\) 
  is an equivalence for \(d\geq1\).
  We use
  the relation \(\ZZ\otimes_{\ZZ[T_{d-1}]}\ZZ\simeq\ZZ[T_{d}]\)
  and induction to prove this.
  The base case \(d=1\) follows
  from the suturedness.

We go back to the statement.
  We prove that for~\(A\),
  the morphism
  \(A\to\pi_{0}(A)\) induces an equivalence.
  We write~\(A\) as a colimit
  of a sequence
  \begin{equation*}
    \ZZ=A_{-1}\to A_{0}\to A_{1}\to\dotsb,
  \end{equation*}
  where~\(A_{n}\)
  is obtained from~\(A_{n-1}\)
  by attaching \(n\)-cells.
  By the previous paragraph,
  we see that \([A_{1}]\to\dotsb\)
  are equivalences.
  From~\(A_{1}\),
  considering another sequence
  \(A_{1}\to A_{2}'\to\dotsb\)
  killing higher cells,
  we can also obtain~\(\pi_{0}(A)\)
  as a colimit.
  This shows the desired result.
\end{proof}

The key in our proof of \cref{main_2}
is deducing that smooth maps of rings
determine suave morphisms of stacks over \(\Mot[2](\ZZ)\).
It is convenient to use six operations to show that.
We use
the material developed in \cref{ss:six}
to construct such:

\begin{proposition}\label{x9v27f}
  Every map \(A\to B\)
  in \((\Cat{Ring}^{\heartsuit})_{\aleph_{0}}\)
  determines an exceptional morphism.
  Therefore,
  we have a functor
  \(\Span(((\Cat{Ring}^{\heartsuit})_{\aleph_{0}})^{\op})\to\Mot[2](\ZZ)\)
  by \cref{xc8fbc}.
\end{proposition}

\begin{proof}
  We write~\(R\) for the universal ring stack
  over~\(\Mot[2](\ZZ)\).
  By \cref{x5xzh9},
  it suffices to show this
  for the zero section \(s\colon{*}\to R\)
  and the tautological map \(f\colon R\to{*}\).
  For~\(s\),
  it is clear.
  For~\(f\),
  this follows from \cref{xujpr0}.
\end{proof}

\subsection{Smooth implies suave}\label{ss:before}

\begin{theorem}\label{xaivv8}
  Let \(A\to B\) be a smooth morphism of static rings.
  Then the corresponding morphism of stacks
  over \(\Mot[2](\ZZ)\) is suave.
\end{theorem}

We first explain the reduction
to the following specific case:

\begin{proposition}\label{xic097}
Let \(A\) be a static ring of finite type
  and \(P\) be a monic polynomial.
  We consider
  \(A\to A[T][(P')^{-1}]/P\).
  Then the corresponding morphism of stacks
  over \(\Mot[2](\ZZ)\)
  is étale (see \cref{xuzdhm};
  we already know that it is static).
\end{proposition}

\begin{lemma}\label{xzspgi}
  A Zariski cover \(A\to B\) in \((\Cat{Ring}^{\heartsuit})_{\aleph_{0}}\)
  determines a cover of stacks over \(\Mot[2](\ZZ)\).
\end{lemma}

\begin{proof}
  This follows
  from the same argument as in \cref{x02mj4}.
\end{proof}

\begin{proof}[Proof of \cref{xaivv8}]
  By \cref{xzspgi},
  suaveness can be checked Zariski locally
  on the source and target by \cref{sua_loc}.
  Therefore,
  we assume that \(Y\to X\)
  factors as \(Y\to\AA^{n}_{X}\to X\)
  where the first morphism is étale.
  The second morphism is suave,
  so it suffices to prove that étale morphisms are suave.

  We can again argue Zariski locally on the source and target,
  and by~\cite[Tag~00UE]{SP},
  it suffices to consider
  \(A\to B=A[T][Q^{-1}]/P\)
  such that \(P\) is monic and \(P'\) is invertible in~\(B\).
  This factors as
  \(A\to A[T][(P')^{-1}]/P\to B\),
  where the second morphism is suave.
  Hence,
  the desired claim follows from \cref{xic097}.
\end{proof}

We move on to the proof of \cref{xic097}.
We write \(f\colon Y\to X\)
for the corresponding morphism of stacks.
Consider the diagram
\begin{equation*}
  \begin{tikzcd}
    Y\ar[r,"d"]&
    Y\times_{X}Y\ar[r,"p"]\ar[d,"q"]&
    Y\ar[d,"f"]\\
    {}&
    Y\ar[r,"f"]&
    X\rlap.
  \end{tikzcd}
\end{equation*}
We know that \(f\) is static and unramified,
so we have a canonical identification
of~\(d_{!}\) as~\(d_{\natural}\).
To prove that~\(f\) is étale,
we need to show that
the morphism
\begin{equation*}
  {\id}_{[Y]}
  \simeq q_{!}d_{!}d^{*}p^{*}
  \simeq q_{!}d_{\natural}d^{*}p^{*}
  \to q_{!}p^{*}
  \simeq f^{*}f_{!}
\end{equation*}
is a unit of an adjunction.
However,
we need to check fewer conditions
by~\cite[Proposition~6.13]{SixFunctors};
see also~\cite[Lemma~4.6.4]{HeyerMann}:

\begin{lemma}\label{xgygp8}
  In the situation above,
  to show that \(f\) is étale,
  it suffices to show that
  \begin{equation*}
    f^{!}\unit_{X}
    \simeq
    d^{*}p^{*}f^{!}\unit_{X}
    \to
    d^{*}q^{!}f^{*}\unit_{X}
    \simeq
    d^{!}q^{!}f^{*}\unit_{X}
    \simeq
    f^{*}\unit_{X}\simeq\unit_{Y}
  \end{equation*}
  in \(\Hom_{\Mot[2](\ZZ)}(\unit,[Y])\)
  is an equivalence.
\end{lemma}

So from now on,
we only consider
the lax symmetric monoidal functor
\begin{equation*}
  D=\Hom_{\Mod_{[A]}(\Mot[2](\ZZ))}([A],\X)
  \colon\Span(\Cat{Aff}_{A})\to\Cat{Pr}_{\st},
\end{equation*}
where \(\Cat{Aff}_{A}\)
is the category of affine schemes of finite type
over~\(A\).
We now conclude
by repeating
the argument in~\cite[Section~10]{SixFunctors}.
We use \(\Spec A\) as the implicit base.

\begin{lemma}\label{xm90y7}
  In the situation above,
  let \(U\) and~\(V\)
  be quasicompact open affine subschemes
  of~\(\AA^{m}\) and~\(\AA^{n}\),
  respectively.
  For any morphism
  \(g\colon V\to U\),
  the object \(g^{!}\unit\) is equivalent to \(\unit(n-m)[2n-2m]\).
  This formation is compatible with base change from the point,
  i.e., for any quasicompact open affine subscheme~\(W\)
  of~\(\AA^{l}\),
  the canonical morphism
  \(p^{*}g^{!}\unit\to({\id}_{W}\times g)^{!}\unit\)
  is an equivalence,
  where \(p\colon W\times V\to V\)
  is the projection.
\end{lemma}

\begin{proof}
  We write \(f\colon U\to{*}\)
  and \(h\colon V\to{*}\)
  for the tautological morphisms.
  By \cref{st_sm},
  we see \(f^{!}\unit=\unit(m)[2m]\)
  and \(h^{!}\unit=\unit(n)[2n]\).
  By \(g^{!}f^{!}\simeq h^{!}\),
  \begin{equation*}
    g^{!}\unit
    \simeq
    g^{!}f^{!}\unit(-m)[-2m]
    \simeq
    h^{!}\unit(-m)[-2m]
    \simeq
    \unit(n-m)[2n-2m]
  \end{equation*}
  holds.
  The desired compatibility of the base change
  also follows from this proof.
\end{proof}

\begin{proof}[Proof of \cref{xic097}]
  We consider the diagram
  \begin{equation*}
    \begin{tikzcd}
      X\ar[r]\ar[d]&
      U\ar[r,hookrightarrow]\ar[d,"g\rvert_{U}"']&
      \AA^{1}\ar[dl,"g"]\\
      {*}\ar[r,"0"]&
      \AA^{1}\rlap,&
      {}
    \end{tikzcd}
  \end{equation*}
  where \(g\) is determined by the polynomial~\(P\).
  We need to prove that \(X\) is étale.
  Instead,
  we wish to prove that \(g\rvert_{U}\) is étale.
  For that,
  we need to prove that the morphism
  \((g\rvert_{U})^{!}\unit\to\unit\)
  constructed in \cref{xgygp8}
  is an equivalence.
  By \cref{xm90y7},
  this is an endomorphism of~\(\unit\).

By base changing the entire diagram
  along \(U\to{*}\),
  we are now in the situation
  where there is a section \(s\colon{*}\to U\),
  depicted as
  \begin{equation*}
    \begin{tikzcd}
      {*}\ar[r,"s"]&
      U\ar[r,hookrightarrow]\ar[d,"g\rvert_{U}"']&
      \AA^{1}\ar[dl,"g"]\\
      {}&
      \AA^{1}\rlap.&
      {}
    \end{tikzcd}
  \end{equation*}
  By the second part of \cref{xm90y7},
  it suffices to show that
  \((g\rvert_{U})^{!}\unit\to\unit\)
  is an equivalence after applying~\(s^{*}\).

  By shifting,
  we assume that the composite of the top arrows is zero.
  This means that \(P=T^{n}+\dotsb+a_{1}T+a_{0}\)
  with \(a_{1}\in A^{\times}\).
  Now, we consider
  \(Q=a_{0}+a_{1}T+S(a_{2}T^{2}+\dotsb+T^{n})\in A[S,T]\)
  and the diagram
  \begin{equation*}
    \begin{tikzcd}
      \AA^{1}_{S}\ar[r,"s"]&
      V\ar[r,hookrightarrow]\ar[d,"h\rvert_{V}"']&
      \AA^{1}_{S}\times\AA^{1}\ar[dl,"h"]\\
      {}&
      \AA^{1}_{S}\times\AA^{1}\rlap,&
      {}
    \end{tikzcd}
  \end{equation*}
  where \(h\) is given by the projection to~\(S\)
  and the polynomial~\(Q\),
  and \(V\) is the étale locus.
  If we prove that \((h\rvert_{V})^{!}\unit\to\unit\)
  is an equivalence after~\(s^{*}\),
  we obtain the desired result
  by pulling back along \(S=1\),
  which is possible
  by the second part of \cref{xm90y7}.
  However,
  by the first part of \cref{xm90y7},
  it is a map between objects coming from the point,
  and by \(\AA^{1}\)-invariance,
  it suffices to prove this after
  pulling back along \(S=0\).
  There,
  by the second part of \cref{xm90y7} again,
  we are in the original situation with a section,
  where \(P\) has degree one.
  In this case, \(g\) is an isomorphism,
  and hence the desired result follows.
\end{proof}

\subsection{Algebraic \texorpdfstring{\(2\)}{2}-motives and ring stacks}\label{ss:cc_alg}

We now have all the necessary components to complete the proof:

\begin{proof}[Proof of \cref{main_2}]
  By \cref{sh_ring},
  both are defined
  via the universal property
  characterizing maps from \(\Cat{Ring}\).
  Therefore, it suffices to construct morphisms
  in both directions
  compatible with the universal maps from \(\Cat{Ring}\).

  The morphism \(\Mot[2](\ZZ)\to\SH[2](\ZZ)\)
  can be readily constructed.
  We construct a morphism in the other direction.
  We have a functor
  \(\Cat{Ring}\to\Mot[2](\ZZ)\).
  We prove that
  its restriction
  to \((\Cat{Ring}^{\heartsuit})_{\aleph_{0}}\)
  satisfies the conditions in \cref{sh_ring}.
  We have already seen~\cref{i:s_sm}
  as \cref{xaivv8}.
  The conditions~\cref{i:s_exc,i:s_hi,i:s_tate}
  follow from base changing
  the universal situation.
\end{proof}

\subsection{Cdh~descent for algebraic \texorpdfstring{\(2\)}{2}-motives}\label{ss:cdh}

We observe that
the classical descent result for~\(\SH\)
extends to~\(\SH[2]\):

\begin{theorem}\label{cdh}
  A cdh~cover of static affine schemes of finite type
  determines a cover
  over \(\SH[2](\ZZ)\).
  Therefore,
  \(\Mod_{[\X]}(\SH[2](\ZZ))\colon\Cat{Ring}\to2\Cat{Pr}_{\st}\)
  is a cdh~sheaf.
\end{theorem}

\begin{proof}
  We first treat Nisnevich descent.
  We consider the Nisnevich square
  \begin{equation*}
    \begin{tikzcd}
      V\ar[r]\ar[d]&
      Y\ar[d]\\
      U\ar[r]&
      X
    \end{tikzcd}
  \end{equation*}
  of affine schemes of finite type.
  By \cref{d_sua},
  we need to prove
  that \([U]\oplus[Y]\to\unit\)
  in \(\SH(X)\) is faithful.
  This can be verified
  by pulling back to~\(\SH(U)\)
  and~\(\SH(Y)\).

  We then prove cdp~descent.
  For that,
  it is convenient to use the \(2\)-motives
  of quasiprojective schemes.
  With this,
  it suffices to show that
  for a (concrete) blowup square
  \begin{equation*}
    \begin{tikzcd}
      E\ar[r]\ar[d]&
      Y\ar[d]\\
      Z\ar[r]&
      X\rlap,
    \end{tikzcd}
  \end{equation*}
  where \(X\) is an affine scheme
  of finite type,
  the morphism
  \([X]\to[Y]\times[Z]\)
  satisfies descent.
  Now,
  by using \cref{xpbfpq},
  it reduces to checking that
  the pushforward of
  \(\unit_{Y\amalg Z}\)
  in \(\SH(X)\)
  satisfies descent.
  We can verify its descendability
  using \cref{md_str}.
\end{proof}

\section{Étale \texorpdfstring{\(2\)}{2}-motives}\label{s:main_et}

In this section,
we prove \cref{main_3}.
We write \(\SH[2]_{\KAS}(\ZZ)\)
for the universal target of \cref{main_3},
which exists by \cref{d_okay}.
We wish to show
that this is equivalent to \(\SH[2]_{\et}(\ZZ)\).

In \cref{ss:shet}, we review étale motivic spectra.
In \cref{ss:vfc}, we review the theory
of motivic Euler characteristics.
In \cref{ss:gal},
we prove Galois descent for the axiomatic target of \cref{main_3}.
In \cref{ss:cc_et},
we prove \cref{main_3}.
In \cref{ss:h},
we prove that h~descent holds for étale \(2\)-motives.

\subsection{A review of étale motivic spectra}\label{ss:shet}

First,
we recall the definition of
étale motivic spectra:

\begin{definition}\label{xq7xf4}
  For a ring~\(A\),
  we define \(\SH_{\et}(A)\)
  to be the étale sheafified version of \(\SH(A)\);
  i.e.,
  for any étale cover \(Y\to X\)
  of smooth schemes of finite presentation over~\(A\),
  we impose that \([Y^{\times_{X}\bullet+1}]\)
  is a colimit diagram.
  Since Nisnevich descent is already imposed,
  it suffices to do this for finite Galois covers between smooth affines.
\end{definition}

The hypercompleted version
is usually considered
in the literature.
The advantage of \cref{xq7xf4} is the following:

\begin{lemma}\label{x5iwcy}
  The functor
  \begin{equation*}
    {\SH_{\et}}\colon\Cat{Ring}\to\CAlg(\Cat{Pr}_{\st})
  \end{equation*}
  commutes with filtered colimits.
\end{lemma}

\begin{proof}
  This follows from
  the same observation for~\(\SH\)
  and the fact that
  any finite Galois cover \(Y\to X\)
  between smooth affines over~\(A\)
  arises from the same situation
  over a ring of finite presentation.
\end{proof}

Nevertheless,
it was proven in~\cite[Theorem~6.29]{BachmannBurklundXu}
that these categories coincide under mild assumptions:

\begin{theorem}[Bachmann–Burklund–Xu]\label{xmqcoq}
  Let \(A\) be a finite-dimensional noetherian ring
  such that the virtual cohomological dimension of residue fields
  is uniformly bounded
  (e.g., \(A\) is of finite type over~\(\ZZ\)).
  Then
  \(\SH_{\et}(A)\) coincides with the hypercomplete version.
\end{theorem}

We later use the following variant:

\begin{corollary}\label{bbx_more}
  In \cref{xmqcoq},
  the same holds
  if all local rings of~\(A\)
  at maximal ideals
  satisfy the hypothesis of \cref{xmqcoq}.
\end{corollary}

\begin{proof}
  It suffices to show that
  the functors \(\SH_{\et}(A)\to\SH_{\et}(A_{\idl{m}})\)
  are jointly conservative,
  where \(\idl{m}\) runs over maximal ideals of~\(A\).
  If an object~\(M\) maps to~\(0\) over~\(A_{\idl{m}}\),
  by \cref{x5iwcy},
  we can take \(f\notin\idl{m}\)
  such that it is~\(0\) over~\(A[f^{-1}]\).
  This covers~\(A\) and hence \(M\simeq0\).
\end{proof}

\begin{remark}\label{x927y8}
  In~\cite[Corollary~5.7]{Bachmann21},
  it is shown at least
  under the assumption of \cref{xmqcoq}
  that the category \(\SH_{\et}(A)\) is compactly generated
  by \(\Sigma^{\infty}_{+}X\)
  when \(X\) is a smooth scheme of finite presentation over~\(A\)
  with finite cohomological dimension.
  From this,
  we see that our category \(\SH_{\et}(A)\)
  is always compactly generated
  by using \cref{x5iwcy}.

  A subtlety is that
  \(\unit\in\SH_{\et}(A)\) is not necessarily compact.
  However,
  we can always cover~\(A\)
  by \(A[\sfrac12][\zeta_{4}]\times A[\sfrac13][\zeta_{6}]\)
  to ensure the unit is compact:
  To see this,
  again by \cref{x5iwcy},
  we can consider the case when \(A\) is of finite type.
  Then, the result follows from
  the description
  of a set of compact generators above.
\end{remark}

Ayoub constructed six operations for~\(\SH_{\et}\)
for quasiprojective schemes over~\(\ZZ\);
namely,
we can apply \cref{ayoub} in this situation
(note that~\cref{i:y_exc}
was proven in~\cite[Corollaire~4.5.47]{Ayoub07b}).
We here consider the étale variant
of \(\SH[2](\ZZ)\):

\begin{definition}\label{xihdgv}
  We write \(\SH[2]_{\et}(\ZZ)\)
  for the presentable \(2\)-category of kernels
  (see \cref{ker_p})
  for \({\SH}_{\et}\colon\Span(\Cat{QProj})\to\Cat{Pr}_{\st}\).
\end{definition}

\begin{lemma}\label{xobwqg}
  There is a canonical\footnote{This morphism is necessarily unique
    by~\cite{DauserKuijper}.
  }
  morphism
  \({\SH}\to{\SH_{\et}}\)
  in \(\CAlg(\Fun(\Span(\Cat{QProj}),\Cat{Pr}_{\st}))\)
  extending
  the tautological morphism
  in \(\CAlg(\Fun(\Cat{QProj}^{\op},\Cat{Pr}_{\st}))\).
\end{lemma}

We can just use \cref{aplus} for this,
but we can also use \cref{main_1}
as follows:

\begin{proof}
  The morphism
  \([\X]\colon\Cat{QProj}^{\op}\to\SH[2]_{\et}(\ZZ)\)
  satisfies the hypothesis of \cref{main_1}
  so that
  we obtain a morphism
  \(\SH[2](\ZZ)\to\SH[2]_{\et}(\ZZ)\).
  The composition of this with
  \(\Span[2]_{\all;P,J}(\Cat{QProj})\to\SH[2](\ZZ)\),
  where \(J\) consists of open immersions
  and \(P\) of projective morphisms,
  coincides with the tautological morphism.
  By restricting this to \(\Span(\Cat{QProj})\),
  we obtain the desired morphism.
\end{proof}

\subsection{A review of the Euler characteristic}\label{ss:vfc}

We need the following:

\begin{proposition}\label{tab}
  Let \(A\) be a static ring containing either~\(\zeta_{4}\) or~\(\zeta_{6}\)
  and \(P\) a monic polynomial of degree~\(d\).
  We consider \(f\colon A\to A[T]/P\).
  Then there is a morphism \(\epsilon\colon f_{*}f^{*}\unit\to\unit\)
  and an invertible morphism \(v\colon\unit\to\unit\)
  in \(\SH(A)\)
  such that the composite
  \begin{equation*}
    \unit\to f_{*}f^{*}\unit\xrightarrow{\epsilon}\unit\xrightarrow{v}\unit\xrightarrow{d}\unit
  \end{equation*}
  is homotopic to~\(d^{2}\).
\end{proposition}

For the proof,
we recall some motivic homotopy theory.
Consider
the motivic \(J\)-homomorphism
\(\Omega^{\infty}K(A)\to\SH(A)\),
where \(K(A)\) denotes the algebraic K-theory spectrum of~\(A\).
Since it maps~\(0\) to~\(\unit\),
we obtain a morphism \(\Omega^{\infty+1}K(A)\to\End_{\SH(A)}(\unit)\).
By composing this with \(A^{\times}\to K_{1}(A)\),
we obtain a morphism \(A^{\times}\to\pi_{0}\End_{\SH(A)}(\unit)\),
for which we write \(\langle\X\rangle\).
For \(d\geq0\),
we write \(d_{\epsilon}\)
for \(\langle-1\rangle^{0}+\dotsb+\langle-1\rangle^{d-1}\).

\begin{lemma}\label{x6mvu2}
  Let \(A\) be a ring.
  When \(\zeta_{4}\in A\),
  we have \(\langle-1\rangle=1\) in \(\pi_{0}\End_{\SH(A)}(\unit)\).
  When \(\zeta_{6}\in A\),
  we have \(2\langle-1\rangle=2\) in
  \(\pi_{0}\End_{\SH(A)}(\unit)\).
\end{lemma}

\begin{proof}
  Let \(l\) be either~\(2\) or~\(3\).
We assume \(A=\ZZ[\zeta_{2l}]\).
  By~\cite{Druzhinin21},
  we have \(\langle u^{2}\rangle=1\)
  for \(u\in A^{\times}\)
  and \(\langle u\rangle+\langle v\rangle
  =\langle u+v\rangle+\langle(u+v)uv\rangle\)
  for \(u\), \(v\), \(u+v\in A^{\times}\).
  When \(l=2\), the first one implies \(\langle-1\rangle=1\).
  When \(l=3\),
  we have
  \(
  2
  =1+1
  =\langle\zeta_{6}^{2}\rangle+\langle\zeta_{3}^{2}\rangle
  =2\langle\zeta_{6}^{2}+\zeta_{3}^{2}\rangle
  =2\langle-1\rangle\).
\end{proof}

\begin{lemma}\label{xx6bhm}
  Let \(A\) be a ring containing
  either~\(\zeta_{4}\) or~\(\zeta_{6}\).
  For \(d\geq0\),
  the element \(dd_{\epsilon}\in\pi_{0}\End_{\SH(A)}(\unit)\)
  is~\(d^{2}\) up to unit.
\end{lemma}

\begin{proof}
  If \(A\) contains~\(\zeta_{4}\),
  by \cref{x6mvu2},
  we have \(d_{\epsilon}=d\).

  If \(A\) contains~\(\zeta_{6}\),
  by \cref{x6mvu2},
  we have \(2(\langle-1\rangle-1)=0\).
  We do case-by-case analysis
  modulo~\(4\).
  When \(d\equiv0\) or~\(1\), we have~\(d_{\epsilon}=d\).
  When \(d\equiv3\),
  we have \(d_{\epsilon}=d\langle-1\rangle\).
  So we are left to treat the case \(d\equiv2\),
  but since \(d\) is even in this case,
  \(2d_{\epsilon}=2d\) is sufficient for the conclusion.
\end{proof}

\begin{proof}[Proof of \cref{tab}]
We apply~\cite[Proposition~B.1.4]{EHKSY21}
  (see~\cite[Proposition~2.2.5]{ElmantoKhan20}
  for a more relevant formulation)
  to obtain a morphism \(f_{*}f^{*}\unit\to\unit\)
  such that its composite with \(\unit\to f_{*}f^{*}\unit\)
  is~\(d_{\epsilon}\).
  By \cref{xx6bhm}, we obtain the desired result.
\end{proof}

\subsection{Étale descent for Kummer–Artin–Schreier \texorpdfstring{\(2\)}{2}-motives}\label{ss:gal}

The following is the key in the proof of \cref{main_3}:

\begin{theorem}\label{xx45iz}
  Let \(G\) be a finite group
  and \(A\to B\) a \(G\)-Galois morphism of rings.
  Then it induces a cover of stacks over \(\SH[2]_{\KAS}(\ZZ)\).
\end{theorem}

We first
simplify the Artin–Schreier condition.
We use the following theorem from~\cite{Bachmann22}:

\begin{theorem}[Bachmann]\label{x9qb4b}
  The tautological morphism
  \(\Cat{Sp}\to\SH(\FF_{p})_{(p)}\)
  in \(\CAlg(\Cat{Pr}_{\st})\)
  canonically
  factors through \(\D(\ZZ)\).
\end{theorem}

\begin{proposition}\label{as_fp}
  We write \(E(\FF_{p})=\Hom_{\SH[2]_{\KAS}(\ZZ)}(\unit,[\FF_{p}])\).
  Then \(p\in\End_{E(\FF_{p})_{(p)}}(\unit)^{\times}\).
\end{proposition}

\begin{remark}\label{xbcpm6}
  In \cref{as_fp},
  a~posteriori
  (after \cref{main_3} is proven),
  we see \(p\in\End_{E(\FF_{p})}(\unit)^{\times}\).
\end{remark}

\begin{proof}
  We consider the composite
  \(\Shv_{\Nis}(\Cat{Sm}_{\FF_{p}};\Cat{Sp})
  \to\SH(\FF_{p})\to E(\FF_{p})\).
  By the Artin–Schreier condition,
  \(\FF_{p}\to\GG_{a}\xrightarrow{T\mapsto T^{p}-T}\GG_{a}\)
  (with the tautological nullhomotopy)
  becomes a cofiber sequence in~\(E(\FF_{p})\).
Since the second morphism
  is equivalent to~\({\id}_{\unit}\) in \(\SH(\FF_{p})\),
  this implies that
  \(\FF_{p}\) becomes~\(0\) in \(E(\FF_{p})\).
This amounts to saying that
  \(\Cat{Sp}\to E(\FF_{p})\)
  carries~\(\FF_{p}\) to~\(0\).

Now we use \cref{x9qb4b}
  to obtain \(\D(\ZZ)\to E(\FF_{p})_{(p)}\).
  The above paragraph shows that
  this morphism kills
  \(\ZZ\otimes_{\SS}\FF_{p}\),
  where we consider the \(\ZZ\)-linear structure
  coming from the left factor.
  Since \(\FF_{p}\)
  is a retract of \(\FF_{p}\otimes_{\SS}\FF_{p}\)
  in \(\D(\ZZ)\),
  it means that \(\FF_{p}=\ZZ\otimes_{\SS}\SS/p\)
  is also killed.
  Therefore, the desired result follows.
\end{proof}

We then prove several cases:

\begin{lemma}\label{xl1f4o}
  The morphism of stacks over \(\SH[2]_{\KAS}(\ZZ)\)
  induced by the ring map
  \(\ZZ[\sfrac1n]\to\ZZ[\sfrac1n][\zeta_{2n}]\)
  is a cover
  for \(n\geq1\).
\end{lemma}

\begin{proof}
  This is the base change of the Kummer map along
  the morphism
  \(\ZZ[\sfrac1n][T^{\pm}]\to\ZZ[\sfrac1n]\)
  mapping~\(T\) to~\(-1\).
\end{proof}

\begin{lemma}\label{xurrdp}
  The morphism of stacks over \(\SH[2]_{\KAS}(\ZZ)\)
  induced by the ring map
  \(\ZZ\to\ZZ[\sfrac12][\zeta_{4}]\times\ZZ[\sfrac13][\zeta_{6}]\) is a cover.
\end{lemma}

\begin{proof}
  First,
  \(\ZZ\to\ZZ[\sfrac12]\times\ZZ[\sfrac13]\) is a Zariski cover
  of rings
  and hence induces a cover by \cref{cdh}.
  We obtain the desired result by combining this with \cref{xl1f4o}.
\end{proof}

\begin{lemma}\label{x5ybjo}
  Let \(A\) be a ring
  and \(P\) be a monic polynomial of degree~\(d\).
  If \(A\to A[T]/P=B\) is étale,
  then the morphism
  \([A][\sfrac1{d}]\to[B][\sfrac1{d}]\)\footnote{Do not confuse this with \([A[\sfrac1{d}]]\to[B[\sfrac1{d}]]\).
  }
  satisfies descent.
\end{lemma}

\begin{proof}
  By base changing along the map in \cref{xurrdp},
  the desired result follows from \cref{tab}.
\end{proof}

\begin{lemma}\label{kummer}
  For \(n\geq1\) invertible in a ring~\(A\),
  any \(C_{n}\)-Galois cover
  \(A\to B\)
  determines a cover.
\end{lemma}

\begin{proof}
By \cref{xl1f4o},
  we assume that \(A\) contains~\(\zeta_{n}\).
  By Kummer theory,
  as long as \(\Pic(A)\) vanishes,
  such a cover is a base change of the Kummer map.
  Since \(\Pic\) vanishes Zariski locally,
  the desired result follows
  from \cref{cdh}.
\end{proof}

\begin{lemma}\label{xfrmql}
  For \(e\geq0\),
  the morphism of stacks over \(\SH[2]_{\KAS}(\ZZ)\)
  induced by
  any \(C_{p^{e}}\)-Galois cover
  \(A\to B\) is a cover.
\end{lemma}

\begin{proof}
  We consider \(A\to B\).
  Since it is proper,
  by \cref{xpbfpq},
  it suffices to show that \(f_{*}\unit\)
  in \(\Hom_{\SH[2]_{\KAS}(\ZZ)}(\unit,[A])\)
  satisfies descent.
  Moreover,
  since it is finite étale,
  it is dualizable
  with a dualizable diagonal.
  Hence, by \cref{xq4y2d},
  we need to see that
  the (\(E_{0}\)-)coalgebra
  \((f_{*}\unit)^{\vee}\) is faithful.
  This can be checked after
  mapping to
  \(\Hom_{\SH[2]_{\KAS}(\ZZ)}(\unit,[A[\sfrac1p]])\)
  and \(\Hom_{\SH[2]_{\KAS}(\ZZ)}(\unit,[A/p])\).
  The former follows from \cref{kummer},
and the latter follows from \cref{tab,as_fp}.
\end{proof}

\begin{proof}[Proof of \cref{xx45iz}]
  Let \(p_{1}\), \dots,~\(p_{n}\)
  be the primes smaller than the cardinality of~\(G\).
  We cover the coefficients~\(\SS\)
  by~\(\SS[\sfrac{p_{i}}{p_{1}\dotsb p_{n}}]\)
  for \(i=1\), \dots,~\(n\).
  By \cref{md_ci}\footnote{Note that quasicompact opens in the Zariski spectrum
    are closed categorically.
  }, the resulting cover satisfies descent.
We then need to show that
  \([A][\sfrac{p_{i}}{p_{1}\dotsb p_{n}}]
  \to[B][\sfrac{p_{i}}{p_{1}\dotsb p_{n}}]\)
  satisfies descent
  for each~\(i\).
  We write \(p=p_{i}\)
  and choose a \(p\)-Sylow group~\(P\) of~\(G\).
  Since \(P\) is solvable,
  we can write \(A\to B\) as
  a sequence of extensions
  \(A\to B^{P}=B_{0}\to\dotsb\to B_{m}=B\)
  where \(B_{j-1}\to B_{j}\)
  for \(j=1\), \dots, \(m\)
  is a \(C_{p^{e_{j}}}\)-Galois extension
  for some~\(e_{j}\geq1\),
  which, by \cref{xfrmql}, induces a cover.
  Hence,
  it suffices to see that
  \(A\to B_{0}\) induces a cover.
  Since this is Nisnevich locally monogenic
  by the primitive element theorem,
  the desired result follows from \cref{cdh,x5ybjo}.
\end{proof}

\subsection{Étale \texorpdfstring{\(2\)}{2}-motives and ring stacks}\label{ss:cc_et}

Here, we prove \cref{main_3}.

\begin{proposition}\label{x155oy}
  Let \(G\) be a finite group
  and \(A\to B\) a \(G\)-Galois morphism of rings.
  Then it determines a cover
  of stacks over \(\SH[2]_{\et}(\ZZ)\).
\end{proposition}

\begin{proof}
  By considering \(G/H\mapsto A^{H}\),
  we obtain a square
  \begin{equation*}
    \begin{tikzcd}
      \Fun(BG;\Cat{Sp})\ar[r]\ar[d]&
      \Cat{Sp}\ar[d]\\
      \SH_{\et}(A)\ar[r]&
      \SH_{\et}(B)
    \end{tikzcd}
  \end{equation*}
  in \(\CAlg(\Cat{Pr}_{\st})\)
  and correspondingly, we obtain a pushout square
  \begin{equation*}
    \begin{tikzcd}
      \Fun(BG;\Cat{Sp})\otimes\unit\ar[r]\ar[d]&
      \unit\ar[d]\\
      {[A]}\ar[r]&
      {[B]}
    \end{tikzcd}
  \end{equation*}
  in \(\CAlg(\SH[2]_{\et}(\ZZ))\).
  The desired result follows from \cref{xapqpz}.
\end{proof}

\begin{remark}\label{x7c1hf}
  For any finite étale map \(A\to B\),
  what is used in the proof of \cref{x155oy}
  generalizes to the pushout square
  \begin{equation*}
    \begin{tikzcd}
      \Shv_{\fet}(A)\otimes\unit\ar[r]\ar[d]&
      \Shv_{\fet}(B)\otimes\unit\ar[d]\\
      {[A]}\ar[r]&
      {[B]}\rlap,
    \end{tikzcd}
  \end{equation*}
  in \(\SH[2]_{\et}(\ZZ)\).
  Moreover,
  when \(A\) and \(B\) satisfy
  the assumption of \cref{xmqcoq},
  we can replace \(\Shv_{\fet}\)
  with its hypercompleted version,
  as hypercompletion is smashing
  by~\cite[Corollary~4.40]{ClausenMathew21}.
\end{remark}

\begin{proof}[Proof of \cref{main_3}]
  By \cref{x155oy},
  we obtain a morphism
  \(F\colon\SH[2]_{\KAS}(\ZZ)\to\SH[2]_{\et}(\ZZ)\).
  We construct
  a morphism \(G\colon\SH[2]_{\et}(\ZZ)\to\SH[2]_{\KAS}(\ZZ)\)
  as in the proof of \cref{main_1}.
  We already have a symmetric monoidal functor
  \(\Span(\Cat{QProj})\to\SH[2]_{\KAS}(\ZZ)\).
  Therefore,
  we wish to construct
  a morphism
  \({\SH_{\et}}\to\Hom_{\SH[2]_{\KAS}(\ZZ)}(\unit,[\X])\)
  in \(\CAlg(\Fun(\Span(\Cat{QProj}),\Cat{Pr}_{\st}))\).
  By \cref{xx45iz},
  we obtain such a morphism
  when restricted to \(\Cat{QProj}^{\op}\).
  By \cref{cll},
  to promote this to the desired morphism,
  we need to check the compatibility condition
  described in the statement of \cref{cll};
  namely, that the squares
  \begin{align*}
    \begin{tikzcd}[ampersand replacement=\&]
      \SH_{\et}(X)\ar[r]\ar[d]\&
      \SH_{\et}(U)\ar[d]\\
      D(X)\ar[r]\&
      D(U)\rlap,
    \end{tikzcd}
    &&
    \begin{tikzcd}[ampersand replacement=\&]
      \SH_{\et}(X)\ar[r]\ar[d]\&
      \SH_{\et}(Y)\ar[d]\\
      D(X)\ar[r]\&
      D(Y)
    \end{tikzcd}
  \end{align*}
  for any open immersion \(U\to X\)
  and projective morphism \(Y\to X\)
  are left and right adjointable, respectively,
  where \(D=\Hom_{\SH[2]_{\KAS}(\ZZ)}(\unit,[\X])\).
  By \cref{xobwqg}
  and the fact that \(\SH_{\et}(\X)\)
  is a localization of~\(\SH(\X)\),
  it suffices to verify this
  when \(\SH_{\et}\) is replaced with~\(\SH\),
  which follows from the argument
  in \cref{xobwqg}.

  We need to see that this construction
  induces an equivalence.
  By construction,
  \(GF\) is homotopic to~\(\id\).
  For \(FG\),
  we argue as in the proof of \cref{main_1}.
\end{proof}

\subsection{H~descent for étale \texorpdfstring{\(2\)}{2}-motives}\label{ss:h}

We observe that
the h~descent result for~\(\SH_{\et}\)
upgrades to~\(\SH[2]_{\et}\):

\begin{theorem}\label{x81mo0}
  An h~cover of static affine schemes of finite type
  determines a cover
  over \(\SH[2]_{\et}(\ZZ)\).
  Therefore,
  \(\Mod_{[\X]}(\SH[2]_{\et}(\ZZ))
  \colon\Cat{Ring}\to2\Cat{Pr}_{\st}\)
  is an h~sheaf.
\end{theorem}

We prove this
by reducing it to the following:

\begin{proposition}\label{xrf8aw}
  Any finite universal homeomorphism
  between static affine schemes of finite type
  determines a cover over \(\SH[2]_{\et}(\ZZ)\).
\end{proposition}

The reduction is based on the following standard observations:

\begin{lemma}\label{x8wf5z}
  Let \(S\) be a quasicompact quasiseparated static scheme.
  Then the h~topology
  on the category
  of static schemes of finite presentation over~\(S\)
  is generated by the cdp and Zariski~topologies
  and by finite locally free surjections.
\end{lemma}

\begin{proof}
  By~\cite[Theorem~2.9]{BhattScholze17},
  the h~topology
  is generated by the cdp and fppf topologies.
  By~\cite[Tag~05WN]{SP},
  the fppf topology
  is generated by the Zariski topology
  and finite locally free surjections.
\end{proof}

\begin{lemma}\label{xw0mdd}
  Let \(Y\to X\) be a finite locally free morphism
  between quasicompact quasiseparated static schemes.
  Then it admits a sequence of
  closed subschemes of finite presentation
  \(\emptyset=Z_{0}\subset\dotsb\subset Z_{n}=X\)
  such that over each stratum~\(Z_{i+1}\setminus Z_{i}\),
  it factors as \(Y\to X'\to X\),
  where \(Y\to X'\) is a universal homeomorphism
  and \(X'\to X\) is a finite étale map.
\end{lemma}

\begin{proof}
  By approximation,
  we assume \(X\) is of finite type over~\(\ZZ\).
  By noetherian induction,
  it suffices to find a nonempty open subscheme
  over which \(Y\to X\) admits the desired decomposition.
  The desired result follows from \cref{xocbq9} below.
\end{proof}

\begin{lemma}\label{xocbq9}
  Let \(k\) be a field and \(B\) an artinian \(k\)-algebra.
  Then it admits a factorization
  \(k\to A\to B\)
  such that \(k\to A\) is étale
  and \(A\to B\) is a universal homeomorphism.
\end{lemma}

\begin{proof}
  By treating each connected component,
  we assume \(B\) is local;
  we write~\(l\) for its residue field.
  We let \(A\) be the subring of~\(B\)
  consisting of separable elements over~\(k\).
  It suffices to show that \(A\to B\),
  equivalently, \(A\to B\to l\)
  is a universal homeomorphism.
  To prove this,
  it suffices to show that any separable element of~\(l\) over~\(k\)
  lifts to~\(B\).
  This follows from the henselian property of~\(B\).
\end{proof}

\begin{proof}[Proof of \cref{x81mo0}]
  Note that cdh~descent follows from \cref{cdh}.
  By \cref{x8wf5z},
  it suffices to show descent
  for finite flat maps.

  Then we use \cref{xw0mdd}.
  By using \cref{x927y8},
  we are in the descendability situation.
  Therefore, we can use \cref{md_str}
  to reduce to the case
  when the map is a universal homeomorphism,
  which is \cref{xrf8aw}.
\end{proof}

We now move on to the proof of \cref{xrf8aw}.
We recall the following notion
from~\cite[Appendix~B]{Rydh10}.
The formulation here is taken from~\cite[Tag~0EUR]{SP}:

\begin{definition}\label{xf4da9}
  Let \(A\) be a static ring.
  The \emph{absolute weak normalization} of~\(A\)
  is the final object in the category of 
  \(A\)-algebras whose structure morphisms
  are universal homeomorphisms.
  We denote this map by \(A\to A^{\awn}\).
  For a general ring,
  its absolute weak normalization
  is that of~\(\pi_{0}\).
\end{definition}

\begin{lemma}\label{xvpzpa}
  In the situation of \cref{xf4da9},
  the absolute weak normalization~\(A^{\awn}\)
  can be expressed as a filtered colimit of \(A\)-algebras
  whose transitions
  are composites of finitely many
  base changes of the following morphisms:
  \begin{enumerate}
    \item\label{i:uh_wsn}
      The morphism \(\ZZ[S,T]/(S^{3}-T^{2})\to\ZZ[U]\)
      mapping~\(S\) and~\(T\) to \(U^{2}\) and~\(U^{3}\), respectively.
    \item\label{i:uh_pe}
      The morphism
      \(\ZZ[S,T]/(p^{p}S-T^{p})\to\ZZ[U]\)
      mapping~\(S\) and~\(T\) to \(U^{p}\) and~\(pU\), respectively.
    \item\label{i:uh_pu}
      The morphism \(\ZZ[S,T]/(S^{p}-T^{p},pS-pT)\to\ZZ[U]\)
      mapping both~\(S\) and~\(T\) to~\(U\).
  \end{enumerate}
\end{lemma}

\begin{proof}
  By~\cite[Tag~0EUL]{SP},
  a static ring is absolutely weakly normal
  if and only if
  the following conditions hold:
  \begin{itemize}
    \item
      When we have~\(x\) and~\(y\)
      satisfying \(x^{3}=y^{2}\),
      there is \(z\) such that \(x=z^{2}\)
      and \(y=z^{3}\).\footnote{This solution is automatically unique
        for weakly seminormal rings.
      }
    \item
      For each prime~\(p\),
      when we have~\(x\) and~\(y\)
      satisfying \(p^{p}x=y^{p}\),
      there is a unique element~\(z\)
      such that \(x=z^{p}\)
      and \(y=pz\).
  \end{itemize}
  The first condition corresponds to~\cref{i:uh_wsn}
  and the second to~\cref{i:uh_pe,i:uh_pu}.
\end{proof}

\begin{proof}[Proof of \cref{xrf8aw}]
  We base change
  along \(\ZZ\to\ZZ[\sfrac12][\zeta_{4}]\times\ZZ[\sfrac13][\zeta_{6}]\)
  to be in the situation where we can use
  descendability,
  and in particular \cref{md_str}.

  Let \(A\to B\) be a finite universal homeomorphism.
  This induces an isomorphism on the absolute weak normalizations;
  in particular,
  \(B\) lies between \(A\) and its absolute weak normalization.
  Therefore,
  by \cref{xvpzpa},
  it suffices to show
  that each map described there is descendable.

  We consider~\cref{i:uh_wsn}.
  We stratify the base by~\(S\).
  Over the open locus,
  it is an isomorphism.
  Over the closed locus,
  both sides reduce to~\(\ZZ\).

  We consider~\cref{i:uh_pe}.
  We stratify the base by~\(p\).
  Over the open locus, it is
  an isomorphism.
  Over the closed locus,
  it becomes
  \(\FF_{p}[S,T]/T^{p}\to\FF_{p}[U]\).
  By considering the reduction,
  it suffices to show that
  \(\FF_{p}[S]\to\FF_{p}[U]\) given by \(S\mapsto U^{p}\)
  is descendable.
  This follows from \cref{tab},
  since \(p\) is invertible in
  \(\SH_{\et}(\FF_{p})\).

  We then treat~\cref{i:uh_pu}.
  We stratify the base by~\(p\).
  Over the open locus,
  it is an isomorphism.
  Over the closed locus,
  it becomes \(\FF_{p}[S,T]/((S-T)^{p})\to\FF_{p}[U]\),
  whose reduction is an isomorphism.
\end{proof}

\begin{corollary}\label{xm89d9}
  The functor
  \(\Cat{Ring}\to\SH[2](\ZZ)\)
  factors through the category of absolute weakly normal rings
  via \((\X)^{\awn}\).
\end{corollary}

\begin{proof}
  By \cref{xvpzpa},
  it suffices to show that
  any finite universal homeomorphism
  induces an isomorphism in~\(\SH[2]_{\et}(\ZZ)\).
  Let \(Y\to X\) be such a map.
  By \cref{xrf8aw},
  we see \([X]\simeq\Tot[Y^{\times_{X}\bullet+1}]\).
  But since the reduction of \(Y^{\times_{X}\bullet+1}\)
  coincides with the reduction of~\(Y\),
  the cosimplicial object is constant.
\end{proof}

\section{Some Berkovich geometry}\label{s:ber}

This section reviews some key facts about Berkovich geometry.
We introduce seminormed rings and uniform Banach rings
in \cref{ss:snring}
and consider their Berkovich spectra in \cref{ss:ber}.
In \cref{ss:std},
we introduce strictly totally disconnecteds,
which are important building blocks in Berkovich geometry.
The arc~topology is discussed in \cref{ss:arc}.
In \cref{ss:tfp},
we study uniform Banach rings
of topologically finite presentation.

\subsection{Seminormed rings and uniform Banach rings}\label{ss:snring}

Here,
we consider the derived variant
for completeness,
though it is not crucial,
as the derived version of normed rings is automatically static.

\begin{definition}\label{xm15vh}
  A \emph{seminorm} on a static ring~\(A\)
  is a function~\(\lvert\X\rvert\colon A\to[0,\infty)\)
  such that
  \(\lvert 0\rvert\leq0\),
  \(\lvert a+b\rvert\leq\lvert a\rvert+\lvert b\rvert\),
  \(\lvert 1\rvert\leq1\),
  and
  \(\lvert ab\rvert\leq\lvert a\rvert\lvert b\rvert\)
  hold for any \(a\) and~\(b\).
  We write \(\Cat{SNRing}^{\heartsuit}\)
  for the category of such objects
  where a morphism \(A\to B\)
  is a ring map \(f\colon A\to B\)
  satisfying \(\lvert f(a)\rvert\leq\lvert a\rvert\)
  for any \(a\in A\).
  We then define the category of (animated) \emph{seminormed rings}
  via the pullback square
  \begin{equation*}
    \begin{tikzcd}
      \Cat{SNRing}\ar[r]\ar[d]&
      \Cat{Ring}\ar[d,"\pi_{0}"]\\
      \Cat{SNRing}^{\heartsuit}\ar[r]&
      \Cat{Ring}^{\heartsuit}\rlap.
    \end{tikzcd}
  \end{equation*}
  in \(\Cat{Pr}\).
  Concretely,
  a \emph{seminorm} on an (animated) ring~\(A\)
  is a seminorm on \(\pi_{0}A\).
\end{definition}

\begin{example}\label{x6zduv}
  The initial object in \(\Cat{SNRing}\) is~\(\ZZ\)
  with the usual absolute value \(\lvert n\rvert=\max(n,-n)\)
  for \(n\in\ZZ\).
\end{example}

\begin{example}\label{xrfuab}
  The coproduct of~\(A\) and~\(B\) in \(\Cat{SNRing}\),
  which we write as~\(A\otimes B\),
  has the underlying ring~\(A\otimes B\).
  Its seminorm is given by
  \begin{equation*}
    \lvert c\rvert=\inf\biggl\{\sum_{i=1}^{n}\lvert a_{i}\rvert\lvert b_{i}\rvert
      \biggm|
      c=\sum_{i=1}^{n}a_{i}\otimes b_{i}
    \biggr\}.
  \end{equation*}
  This is commonly called the \emph{projective tensor product}.
\end{example}

\begin{example}\label{x6s7o2}
  For \(r\in[0,\infty)\),
  we have the seminormed ring
  \(\ZZ[T]_{r}\)
  that corepresents the functor
  \(\Cat{SNRing}\to\Cat{Ani}\)
  that maps \(A\) to the anima of elements~\(a\in A\)
  satisfying \(\lvert a\rvert\leq r\).
  Concretely,
  the underlying ring is~\(\ZZ[T]\)
  and the seminorm is given by
  \(\lvert\sum_{i=0}^{n}a_{i}T^{i}\rvert=\sum_{i=0}^{n}\lvert a_{i}\rvert r^{i}\).
  We have \(\injlim_{r'>r}\ZZ[T]_{r'}=\ZZ[T]_{r}\).

  By taking coproducts,
  we also obtain
  \(\ZZ[T_{1},\dotsc,T_{n}]_{r_{1},\dotsc,r_{n}}\)
  for \(r_{i}\in[0,\infty)\).
\end{example}

The presentable category \(\Cat{SNRing}\)
is not compactly generated,
but it is close to that:

\begin{proposition}\label{xywaym}
  Consider the full subcategory \(\Cat{SNPol}\)
  of \(\Cat{SNRing}\)
  spanned by \(\ZZ[T_{1},\dotsc,T_{n}]_{r_{1},\dotsc,r_{n}}\)
  for \(r_{i}\in[0,\infty)\).
  Then the Yoneda functor
  \begin{equation*}
    \Cat{SNRing}\to\Fun(\Cat{SNPol}^{\op},\Cat{Ani})
  \end{equation*}
  is fully faithful
  and the essential image consists
  of functors~\(F\) that
  carry finite coproducts to products
  and the colimits
  \begin{align*}
    \ZZ[T]_{s}\otimes_{\ZZ[T]_{r}}\ZZ[T]_{s}&\simeq
    \ZZ[T]_{s}&
    \injlim_{r'>r}\ZZ[T]_{r'}&\simeq\ZZ[T]_{r},
  \end{align*}
  where \(r\leq s\),
  to the corresponding limits.
  In particular,
  it is \(\aleph_{1}\)-compactly generated.
\end{proposition}

\begin{proof}
  Note that
  the Yoneda functor
  \(\Cat{Ring}\to\Fun(\Cat{Pol}^{\op},\Cat{Ani})\)
  is fully faithful
  and its essential image consists
  of functors
  that carry finite coproducts to products.

  We first prove that it is fully faithful.
  For a seminormed ring~\(A\),
  it suffices to show
  the colimit of the tautological diagram indexed by
  \(\Cat{SNPol}_{/A}\) is equivalent to~\(A\).
  This index category is sifted.
  We first consider
  \(\Cat{SNPol}_{/A}\to\Cat{Pol}_{/A}\)
  and observe that it is cofinal.
  Therefore,
  the colimit computes the correct underlying ring.
  The identification of the seminorm
  is straightforward
  by using~\(\ZZ[T]_{r}\).

  We then check
  the description of the essential image.
  Suppose \(F\) satisfies the conditions.
  We define~\(F'\)
  to be its left Kan extension along \(\Cat{SNPol}\to\Cat{Pol}\).
  Concretely,
  we have \(F'(\ZZ[T_{1},\dotsc,T_{n}])=
  \injlim_{r}F(\ZZ[T_{1},\dotsc,T_{n}]_{r,\dotsc,r})\).
  Hence,
  it carries finite coproducts to finite products,
  and therefore it is representable by a ring~\(A\).
  For \(a\in\pi_{0}(A)\),
  consider a corresponding element
  in \(F'(\ZZ[T])\).
  Since \(F(\ZZ[T]_{r})\to F'(\ZZ[T])\)
  is a monomorphism by the condition,
  \(F(\ZZ[T]_{r})\times_{F'(\ZZ[T])}{*}\)
  is either final or empty.
  We set \(\lvert a\rvert\) to be
  the infimum of~\(r\) such that
  \(F(\ZZ[T]_{r})\times_{F'(\ZZ[T])}{*}\)
  is nonempty,
  which exists by the condition.
  We can check that this defines a seminorm on~\(A\),
  and that \(F\) is represented by~\(A\).
\end{proof}

\begin{corollary}\label{xwsis1}
  The category
  \(\Cat{SNRing}\)
  is a retract of a compactly generated category
  in~\(\Cat{Pr}\).
\end{corollary}

\begin{proof}
  We consider the full subcategory~\(\cat{C}\)
  of \(\Fun(\Cat{SNPol}^{\op},\Cat{Ani})\)
  spanned by
  functors~\(F\) that carry finite coproducts
  and
  \(F(\ZZ[T]_{s})\to F(\ZZ[T]_{s})\times_{F(\ZZ[T]_{r})}F(\ZZ[T]_{s})\)
  is an equivalence for any~\(r\leq s\).
  This is compactly generated.
  We conclude the proof by
  showing that
  the left adjoint~\(L\) to
  the embedding
  \(\Cat{SNRing}\hookrightarrow\cat{C}\)
  of \cref{xywaym}
  admits a further left adjoint.
  It suffices to show that
  \(L\) preserves limits.
  By \cref{xywaym},
  we see that
  \begin{equation*}
    \Map_{\Cat{SNRing}}(\ZZ[T]_{r},L(F))
    \simeq
    \projlim_{r'>r}F(\ZZ[T]_{r'}),
  \end{equation*}
  which shows the desired result.
\end{proof}

\begin{remark}\label{xfw8zq}
  \Cref{xwsis1} says that
  \(\Cat{SNRing}\) is \emph{continuous}
  in the sense of Johnstone–Joyal~\cite{JohnstoneJoyal82}
  or \emph{compactly assembled}
  in the sense of Lurie~\cite[Section~21.1.2]{LurieSAG}.
  One can also prove a symmetric monoidal refinement of \cref{xwsis1}.
\end{remark}

We then introduce its important full subcategories:

\begin{definition}\label{xhfuha}
  Let \(A\) be a seminormed ring.
  It is called a \emph{normed ring}
  if the anima of elements~\(a\in A\) with \(\lvert a\rvert=0\)
  is trivial.
  This forces the underlying animated ring to be static.
  We furthermore say that
  it is a \emph{Banach ring}
  if every Cauchy sequence converges.
  We say that it is \emph{uniform}\footnote{This is also called \emph{spectral},
    since a Banach ring satisfies this condition
    if and only if 
    \(\lvert a\rvert=\lim_{n\to\infty}\lvert a^{n}\rvert^{\sfrac1n}\) holds
    for any \(a\in A\).
  }
  if \(\lvert a^{2}\rvert=\lvert a\rvert^{2}\) holds
  for any \(a\in A\).
  We write \(\Cat{URing}\) for the full subcategory
  of \(\Cat{SNRing}\) spanned by uniform Banach rings.
\end{definition}

We can characterize these
using lifting properties:

\begin{proposition}\label{xj6u9m}
  Let \(A\) be a seminormed ring.
  \begin{enumerate}
    \item\label{i:li_n}
      It is normed if and only if
      it is local with respect to \(\ZZ[T]_{0}\to\ZZ\).
    \item\label{i:li_b}
      If it is normed,
      then it is Banach
      if and only if it is local with respect to
      \(A\to B\) for \(r>0\),
      where
      \(A\) is \(\ZZ[T_{1},\dotsc]_{r,\dotsc}\)
      with \(\lvert T_{m}-T_{m+k}\rvert\leq\sfrac1m\)
      for \(m\geq1\) and \(k\geq1\)
      and \(B\) is \(A[U]_{r}\)
      with \(\lvert U-T_{m}\rvert\leq\sfrac1m\)
      for \(m\geq1\).
      Formally,
      to construct~\(A\),
      we first add variables \(S_{m,k}\)
      for \(m\geq1\) and \(k\geq1\)
      with radius~\(\sfrac1m\)
      and then take the quotient by \(S_{m,k}-T_{m}-T_{m+k}\).
      Similarly for~\(B\).
    \item\label{i:li_u}
      If it is Banach,
      it is uniform
      if and only if
      it is local with respect to
\begin{equation*}
        \ZZ[T,U]_{r,s}/(T^{2}-U)
        \to
        \ZZ[T]_{\sqrt{s}}
      \end{equation*}
      with \(r^{2}\geq s>0\).
  \end{enumerate}
\end{proposition}

\begin{proof}
  This is straightforward.
\end{proof}

\begin{corollary}\label{x2lbc5}
  All the full subcategories in \cref{xhfuha}
  are \(\aleph_{1}\)-compactly generated
  and localizations live in \(\Cat{Pr}^{\aleph_{1}}\).
\end{corollary}

\begin{proof}
  This is an immediate consequence of \cref{xywaym,xj6u9m}.
\end{proof}

\begin{definition}\label{x625gh}
  We call the left adjoint
  to the inclusion \(\Cat{SNRing}\to\Cat{UBan}\)
  \emph{uniform completion}
  and write it as \(A\mapsto A^{\ucpl}\).
\end{definition}

\begin{example}\label{xhcyqw}
  For a uniform Banach ring~\(A\),
  we write \(A\{T_{1},\dotsc,T_{n}\}_{r_{1},\dotsc,r_{n}}\)
  for the uniform completion of the tensor product
  of~\(A\) with \(\ZZ[T_{1},\dotsc,T_{n}]_{r_{1},\dotsc,r_{n}}\).
\end{example}

We conclude this section
by performing some calculations,
which we use later:

\begin{lemma}\label{ban_mg}
  Let \(A\) be a normed ring
  and \(P=T^{n}+\dotsb+a_{0}\)
  a monic polynomial.
  Then the seminormed ring \(A[T]_{r}/P\)
  is Banach
  provided that
  \begin{equation*}
    r^{n}\geq 2(\lvert a_{n-1}\rvert r^{n-1}+\dotsb+\lvert a_{0}\rvert)
  \end{equation*}
  is satisfied.
  Moreover,
  in this situation,
  if \(A\) is Banach, so is \(A[T]_{r}/P\).
\end{lemma}

\begin{proof}
  In this proof,
  for a polynomial~\(S\),
  we write~\(S_{<n}\) and~\(S_{\geq n}\)
  for the components of degree~\(<n\)
  and \(\geq n\), respectively.

To show the claim, it suffices to show that
  for \(R\in A[T]\) of degree \(<n\),
  we have
  \begin{equation*}
    \lvert PQ+R\rvert\geq\lvert R\rvert
  \end{equation*}
  for any \(Q\in A[T]\),
  where \(\lvert\X\rvert\) denotes the norm on \(A[T]_{r}\).
  By
  \begin{equation*}
    \lvert PQ+R\rvert
    =\lvert(PQ+R)_{<n}\rvert+\lvert(PQ+R)_{\geq n}\rvert
    \geq(\lvert R\rvert-\lvert(PQ)_{<n}\rvert)+\lvert(PQ)_{\geq n}\rvert,
  \end{equation*}
  it suffices to show
  \(\lvert(PQ)_{\geq n}\rvert\geq\lvert(PQ)_{<n}\rvert\).
  We have
  \begin{equation*}
    \lvert(PQ)_{<n}\rvert
    =
    \biggl\lvert
    \sum_{i=0}^{n-1}a_{i}(T^{i}Q)_{<n}
    \biggr\rvert
    \leq
    \sum_{i=0}^{n-1}\lvert a_{i}\rvert\lvert T^{i}Q\rvert
    =
    \sum_{i=0}^{n-1}\lvert a_{i}\rvert r^{i}\lvert Q\rvert.
  \end{equation*}
  On the other hand, we have
  \begin{equation*}
    \lvert(PQ)_{\geq n}\rvert
    \geq
    \lvert T^{n}Q\rvert
    -\sum_{i=0}^{n-1}\lvert a_{i}\rvert\lvert(T^{i}Q)_{\geq n}\rvert
    \geq
    \lvert T^{n}Q\rvert
    -\sum_{i=0}^{n-1}\lvert a_{i}\rvert\lvert T^{i}Q\rvert
    =
    \biggl(
    r^{n}
    -\sum_{i=0}^{n-1}\lvert a_{i}\rvert r^{i}
    \biggr)
    \lvert Q\rvert.
  \end{equation*}
  Therefore, the result follows from the assumption.

  In the Banach case,
  completeness can be checked directly
  by using the description of the norm obtained above.
\end{proof}

\begin{remark}\label{xujhcs}
  Unlike \cref{ban_mg},
  when we want to extend a uniform Banach ring structure,
  different considerations are required;
  e.g., the uniform completion of \(\RR[T]_{r}/T^{2}\)
  is~\(\RR\) for any~\(r\).
  We treat a very special case in \cref{xhll2q}.
\end{remark}

\subsection{Berkovich spectrum}\label{ss:ber}

We introduce the notion of fields
and hence the spectrum:

\begin{definition}\label{xu8dcd}
  A seminorm on a ring~\(A\) is \emph{multiplicative}
  if \(\lvert1\rvert=1\) and \(\lvert ab\rvert=\lvert a\rvert\lvert b\rvert\)
  for any \(a\) and~\(b\in A\).
  A \emph{normed field}
  is a normed ring
  whose underlying ring is a field
  and the seminorm is multiplicative.
  We define a \emph{Banach field} similarly;
  note that a Banach field is automatically uniform.
\end{definition}

\begin{example}\label{x795cv}
  Fix a prime~\(p\) and consider~\(\QQ_{p}\)
  with the norm \(\max(\lvert\X\rvert_{p},\lvert\X\rvert_{p}^{2})\).
  This is a uniform Banach ring,
  but not a Banach field.
\end{example}

\begin{definition}[Berkovich]\label{x82bxa}
  Let \(A\) be a seminormed ring.
  Its \emph{Berkovich spectrum}
  is the closed subspace
  \begin{equation*}
    \Sp(A)
    \subset
    \prod_{a\in A}[0,\lvert a\rvert]
  \end{equation*}
  consisting of
  multiplicative seminorms on~\(A\)
  bounded by the original norm on~\(A\).
\end{definition}

\begin{remark}\label{xwpus5}
  In \cref{x82bxa},
  the underlying set of~\(\Sp(A)\)
  can be identified with
  the set of equivalence classes of morphisms
  \(A\to K\)
  to normed fields~\(K\).
  The equivalence relation
  is generated by
  \((A\to K)\sim(A\to K')\)
  if there is an isometric embedding \(K\to K'\)
  making the diagram commute.
\end{remark}

The following
was proven in~\cite[Proposition~3.2]{ScholzeB}:

\begin{proposition}\label{xfdud8}
  The following holds for the Berkovich spectrum
  of a seminormed ring~\(A\):
  \begin{enumerate}
    \item\label{i:sp_fil}
      When \(A\) is a filtered colimit~\(\injlim_{i}A_{i}\),
      the map \(\Sp(A)\to\projlim_{i}\Sp(A_{i})\)
      is a homeomorphism.
    \item\label{i:sp_pb}
      When \(B'=A'\otimes_{A}B\) is a pushout
      of seminormed rings,
      the map
      \(\Sp(B')\to\Sp(A')\times_{\Sp(A)}\Sp(B)\)
      is surjective.
    \item\label{i:sp_ucpl}
      The morphism \(A\to A^{\ucpl}\) induces
      a homeomorphism \(\Sp(A^{\ucpl})\to\Sp(A)\).
  \end{enumerate}
\end{proposition}

We do not consider general rational domains
in this paper;
Weierstraß and Laurent domains suffice for our purpose:

\begin{lemma}\label{weilau}
Let \(A\) be a seminormed ring.
  The Berkovich spectra
  of
  \(A[T]_{r}/(T-a)\)
  and \(A[T]_{\sfrac1r}/(aT-1)\)
  are identified with
  the base changes of \(\Sp(A)\)
  along the inclusions
  \([0,r]\to[0,\lvert a\rvert]\)
  and \([r,\lvert a\rvert]\to[0,\lvert a\rvert]\),
  respectively.
  In particular,
  \(\Sp\) is compatible with base change
  in this case.
\end{lemma}

\begin{proof}
  Since Berkovich spectra are compact Hausdorff,
  it suffices to show this on the level of underlying sets.
  This follows from \cref{xwpus5}.
\end{proof}

The following was proven in~\cite[Theorem~2.18]{ScholzeB}:

\begin{theorem}[Scholze]\label{ban_idem}
  Let \(A\) be a Banach ring.
  Then clopen subsets of~\(\Sp(A)\)
  correspond bijectively
  to idempotent elements of~\(A\).
\end{theorem}

\begin{corollary}\label{sp_spec}
  Let \(A\) be a Banach ring.
  The canonical map
  \(\Sp(A)\to\Spec(A)\)
  that maps \(\lvert\X\rvert\)
  to its kernel
  is continuous
  and induces a bijection on connected components.
\end{corollary}

\begin{proof}
  The continuity is straightforward.
  The claim about connected components requires \cref{ban_idem}.
\end{proof}

\begin{corollary}\label{td_sp}
  Let \(A\) be a Banach ring
  such that \(\Sp(A)\) is totally disconnected.
  Then the map \(\Sp(A)\to\Spec(A)\)
  is injective
  and exactly hits the closed points.
  Moreover,
  each connected component of \(\Spec(A)\)
  has a unique closed point.
\end{corollary}

\begin{proof}
  This is a direct consequence of \cref{sp_spec}.
\end{proof}

We bound the size of the Berkovich spectrum:

\begin{lemma}\label{xwafwj}
  Let \(\kappa>\aleph_{0}\) be a regular cardinal
  and \(A\) a \(\kappa\)-compact object in \(\Cat{URing}\).
  Then \(\Sp(A)\) has weight \(<\kappa\).
\end{lemma}

\begin{proof}
  It is a quotient
  of the uniform completion
  of the seminormed ring
  \(A=\ZZ[T_{i}\mid i\in I]_{r_{i}}\),
  where \(I\) is a set of cardinality~\(<\kappa\)
  and \(r_{i}\geq0\).
  Hence,
  it is a closed subset of \(\Sp(A^{\ucpl})\).
  By~\cref{i:sp_ucpl} of \cref{xfdud8},
  it suffices to observe that
  \(\Sp(A)\) has weight~\(<\kappa\),
  but this follows from the definition
  and that \(A\) is of cardinality \(<\kappa\),
  since \(\prod_{i\in I}[0,r_{i}]\) has weight~\(<\kappa\).
\end{proof}

We only use the following
notation for Banach rings:

\begin{definition}\label{x5hs0a}
  Let \(A\) be a Banach ring
  and \(x\in\Sp(A)\).
  The (complete) \emph{residue field}~\(K(x)\)
  is the completion of the residue field at \(\ker\lvert\X\rvert_{x}\),
  where \(\lvert\X\rvert_{x}\) is the multiplicative seminorm
  corresponding to~\(x\).
\end{definition}

\subsection{(Strictly) totally disconnecteds}\label{ss:std}

We first recall the following:

\begin{lemma}\label{xu2axd}
  Let \(K\) be a nondiscrete Banach field.
  The completion of its separable closure
  is algebraically closed.
\end{lemma}

\begin{proof}
  The archimedean case is vacuous.
  The nonarchimedean case
  follows from~\cite[Propositions~3.4.1.3 and~3.4.1.6]{BoschGuntzerRemmert84}.
\end{proof}

\begin{definition}[Scholze]\label{x9lxzs}
  We call a uniform Banach ring
  \emph{totally disconnected}
  when the underlying space is totally disconnected
  and all residue fields are nondiscrete.
  We call it \emph{strictly totally disconnected}
  when moreover residue fields are separably closed
  (hence algebraically closed by \cref{xu2axd}).
  We write \(\Cat{TDis}\)
  and \(\Cat{STDis}\) for the full subcategories
  of \(\Cat{URing}\) spanned by them, respectively.
\end{definition}

We recall the following
from usual algebraic geometry:

\begin{definition}\label{x0h0c8}
  Let \(A\) be a totally disconnected.
  For \(x\in\Sp(A)\subset\Spec(A)\)
  (see \cref{td_sp}),
  we consider the uncompleted residue field~\(\kappa(x)\),
  which is just the Zariski residue field
  at~\(\ker\lvert\X\rvert_{x}\)\footnote{For general~\(A\)
    and \(x\in\Sp(A)\),
    this does \emph{not} recover the definition
    of the uncompleted residue field in Berkovich geometry.
    Our totally disconnectedness assumption
    ensures that it does here.
  }.
  Then we define its local ring~\(A_{x}\)
  to be its Zariski local ring.
\end{definition}

\begin{proposition}\label{xibpba}
  In the situation of \cref{x0h0c8},
  the local ring~\(A_{x}\) is henselian.
\end{proposition}

\begin{proof}
It suffices to show that
  \(A_{x}[T]/P\) for any monic polynomial~\(P\)
  is a finite product of local rings.
  By spreading out,
  we assume that \(P\) is defined over~\(A\).
  We take~\(r\gg0\)
  so that \(A[T]_{r}/P\) is Banach
  by \cref{ban_mg}.
  Since \(\Sp(A[T]_{r}/P)\to\Sp(A)\)
  has finite fibers,
  \(A[T]_{r}/P\) has totally disconnected spectrum.
  Then by \cref{td_sp} again,
  \(A_{x}[T]/P\) becomes
  the product of local rings at points
  of \(\Sp(A[T]_{r}/P)\) lying over~\(x\).
\end{proof}

\begin{lemma}\label{td_pu}
  Let \(A\) be a totally disconnected.
  Then there is a morphism
  \(\ZZ\{T,T^{-1}\}_{s,\sfrac1r}\to A\)
  for some \(0<r\leq s<1\).
\end{lemma}

\begin{proof}
  We proceed locally.
  So we have~\(A_{x}\to\kappa(x)\to K(x)\).
  If the completion is nondiscrete,
  then \(\kappa(x)\) is also nondiscrete.
  Therefore, there exists a uniformizer.
  Then we can spread out this and its inverse to
  elements~\(\pi\) and~\(\pi^{-1}\in A\).

  By this,
  we obtain an element \(\pi\in A\)
  such that the map
  whose norm factors through
  \((0,1)\subset[0,\infty]\).
  We choose a subset \([r,s]\) containing its image,
  and the desired result follows.
\end{proof}

We record some facts about
finite étale maps
over totally disconnecteds:

\begin{proposition}\label{xusjrg}
Let \(A\) be a totally disconnected
  and \(x\in A\) a point.
  Consider a finite separable extension~\(L(x)\) over~\(K(x)\).
  Then there is a finite étale map \(A\to B\)
  (cf.~\cref{xhll2q})
  that induces a homeomorphism on spectra
  such that \(L(x)\) is the base change of~\(B\)
  along \(A\to K(x)\).
\end{proposition}

We first
confirm
that
finite étale algebras
over totally disconnecteds
admit a unique structure of a totally disconnected:

\begin{lemma}\label{xhll2q}
  Let \(A\) be a totally disconnected.
  The forgetful functor
  \(\Cat{TDis}_{A/}\to\Cat{Ring}_{A/}\)
  is an equivalence of categories
  when we restrict to the full subcategory
  of finite étale algebras on the target.
\end{lemma}

In the following proof,
we see an explicit formula for the norm:

\begin{proof}
  What we show here
  is that any finite étale algebra~\(B\) over~\(A\)
  admits a unique structure of a totally disconnected
  so that \(A\to B\) is a morphism of totally disconnecteds;
  the desired equivalence
  follows from the proof.

  We first consider a finite separable extension
  of nondiscrete Banach fields \(A\to B\).
  The archimedean case is straightforward.
  The nonarchimedean case
  follows from~\cite[Theorem~3.2.4.2]{BoschGuntzerRemmert84}.
  By uniqueness,
  the norm on~\(B\)
  must be given by \(\lvert\Nm_{B/A}(\X)\rvert^{\sfrac1d}\),
  where \(\Nm_{B/A}\) is the norm map \(B\to A\)
  and \(d\) is the degree.

  We go back to the general statement.
  For \(x\in\Sp(A)\),
  we write \(B(x)\),
  for the (algebraic) tensor product \(K(x)\otimes_{A}B\).
  This is a finite product of finite separable extensions of~\(K(x)\),
  and hence by the argument above,
  admits a unique norm that makes \(K(x)\to B(x)\)
  a morphism of uniform Banach rings.
  The norm on~\(B\) must be spectral,
  so the only possible candidate
  is the map sending \(b\in B\) 
  to~\(\sup_{x\in\Sp(A)}\lvert b\rvert_{x}\),
  where
  we write~\(\lvert\X\rvert_{x}\)
  for the norm on~\(B(x)\).
  The problem is the existence of this supremum,
  but since it coincides
  with \(\lvert\Nm_{B/A}(\X)\rvert^{\sfrac1d}\),
  where \(d\) is the degree (as a locally constant function),
  it is well defined.
  The fact that this function
  gives~\(B\) a structure of a totally disconnected
  is clear from the description via supremum.
\end{proof}

\begin{lemma}\label{xbjwxs}
  Let \(\kappa\to K\)
  be the completion morphism of a normed field~\(\kappa\).
  Then any separable finite extension of~\(K\)
  lifts to a separable finite extension of~\(\kappa\).
\end{lemma}

\begin{proof}
  In the archimedean case,
  we need to consider \(K=\RR\to\CC\) only,
  and the desired result is clear.
  The nonarchimedean case
  follows from~\cite[Proposition~3.4.2.5]{BoschGuntzerRemmert84}.
\end{proof}

\begin{proof}[Proof of \cref{xusjrg}]
  We consider the maps \(A_{x}\to\kappa(x)\to K(x)\).
  By \cref{xibpba,xbjwxs},
  the finite étale algebra
  \(B(x)\) lifts to
  a finite étale algebra over~\(A_{x}\),
  and hence to a finite étale algebra~\(B'\) over~\(A[e^{-1}]\)
  that induces a homeomorphism on spectra
  for some idempotent~\(e\) that becomes invertible in~\(K(x)\).
  Then \(B'\times A/e\) is the desired \(A\)-algebra.
\end{proof}

We introduce
a class of totally disconnecteds:

\begin{definition}\label{xjfa0h}
  We call a (strictly) totally disconnected \emph{light}
  if it is \(\aleph_{1}\)-compact as an object of \(\Cat{URing}\).
  We write \(\Cat{TDis}_{\lgt}\) and \(\Cat{STDis}_{\lgt}\)
  for the corresponding full subcategories.
\end{definition}

The following bound is important later:

\begin{proposition}\label{xxvdc1}
  There are countably many isomorphism classes
  of finite étale algebras
  over a light totally disconnected.
\end{proposition}

\begin{proof}
The question is reduced to the algebraic case by \cref{xhll2q}.
  Each finite étale algebra~\(B\)
  determines a morphism
  \(\Sp(B)\to\Sp(A)\)
  and there is a partition
  of~\(A\) into finitely many components
  such that
  \(B\) is a free module of finite rank
  on each component.
  Moreover,
  it is a finite product of monogenic ones.
  Therefore,
  by \cref{xwafwj},
  it suffices to show that
  there are countably many isomorphism classes
  of étale \(A\)-algebras of the form
  \(A[T]/(T^{n}+\dotsb+a_{0})\).

  We view~\(A\) as a topological space
  by considering the metric induced by the norm.
  Since it is \(\aleph_{1}\)-compact
  as a uniform Banach ring,
  \(A\) contains a countable dense subset.
  Thus, it is Lindelöf.

  We consider
  a separable polynomial \(P=T^{n}+\dotsb+a_{0}\).
  We wish to find a neighborhood
  of~\(P\) in the space of monic polynomials
  of degree~\(n\) (which is isomorphic to~\(A^{n}\),
  which is Lindelöf)
  such that \(A[T]/P\simeq A[T]/Q\)
  when \(Q\) is contained in this neighborhood.
  To construct such an isomorphism,
  we need to find
  an element \(\alpha\in A[T]/P\)
  close enough to~\(T\)
  satisfying \(Q(\alpha)=0\);
  the proximity guarantees 
  the map is an isomorphism.
  This is possible using the Newton method.
\end{proof}

\subsection{Arc~topology}\label{ss:arc}

\begin{definition}[Scholze]\label{x0p547}
  The arc~topology on \(\Cat{URing}^{\op}\)
  is the Grothendieck topology
  generated by surjections on Berkovich spectra
  and by finite disjoint unions.
  This is well defined by~\cref{i:sp_pb} of \cref{xfdud8}.
  An \emph{arc~stack}\footnote{This is called “small arc~stack” in~\cite{ScholzeB}.
    Our arc~stacks are always small.
  } is an accessible hypersheaf on \(\Cat{URing}^{\op}\).
\end{definition}

\begin{example}\label{xzxpzi}
  The morphism
  \(\ZZ\to\ZZ\{T,T^{-1}\}_{r,\sfrac1r}\)
  is an arc~cover for \(r>0\).
  This translates to the fact
  that for any Banach field~\(K\),
  there is an extension \(K\to K'\)
  of Banach fields
  such that \(K'\)
  contains an element~\(a\)
  with \(\lvert a\rvert=r\).
\end{example}

\begin{example}\label{x9up2q}
For an element~\(a\)
  of a uniform Banach ring~\(A\)
  and~\(r\geq0\),
  the uniform completions of
  the rings given in \cref{weilau} determine an arc~cover.
\end{example}

Note that strictly totally disconnecteds
constitute a nice basis:

\begin{theorem}[Scholze]\label{xsdudu}
  For any uniform Banach ring~\(A\),
  there is a strictly totally disconnected~\(B\)
  with an arc~cover \(A\to B\).
\end{theorem}

\begin{theorem}[Scholze]\label{xhgzg6}
  Strictly totally disconnecteds are (hyper)subcanonical
  with respect to the arc~topology.
\end{theorem}

\begin{definition}\label{xavdzj}
  For a uniform Banach ring~\(A\),
  we write \(\Sparc(A)\) for the corresponding arc~stack.
\end{definition}

See~\cite[Theorem~3.13]{ScholzeB} for
the proof of \cref{xhgzg6}.
Since the proof of \cref{xsdudu} is important,
we recall it here.

\begin{remark}\label{xwo641}
  Note that the first proof of \cref{xsdudu}
  given in~\cite{ScholzeB}
  relied on~\cite[Proposition~3.11\,(i)]{ScholzeB},
  which is incorrect as stated:
  Fix a prime~\(p\)
  and consider the product~\(A\)
  of~\(\QQ_{p}\) with norms \(\lvert p\rvert=p^{-n}\)
  for \(n\geq1\).
  Then \(\Sp(A)\) is indeed totally disconnected,
  but \(A\) is not analytic.
  A similar counterexample
  exists for~\cite[Proposition~3.11~(iv)]{ScholzeB} as well.
\end{remark}

We use the following observation:

\begin{lemma}\label{clean_cov}
  Let \(A\) be a ring such
  that every connected component of \(\Spec A\)
  admits a unique closed point
  and the local ring at each closed point is henselian.
  Then there is a filtered colimit~\(B\) of finite étale covers of~\(A\)
  such that
  \(B\) satisfies the same condition as~\(A\)
  and the residue fields at closed points of~\(B\) are separably closed.
\end{lemma}

\begin{proof}
  Let \(\{A_{i}\mid i\in I\}\)
  be a set of representatives of isomorphism classes
  of finite étale covers of~\(A\).
  We take its tensor product
  \(B=\bigotimes_{i\in I}A_{i}
  \simeq\injlim_{I_{0}}\bigotimes_{i\in I_{0}}A_{i}\),
  where \(I_{0}\) runs through finite subsets of~\(I\).
  We claim that this is the desired \(A\)-algebra.

  It is left to verify that
  the residue fields of~\(B\)
  at maximal ideals
  are separably closed.
  Consider a closed point \(y\in\Spec B\).
  This determines
  a compatible family of closed points \(x_{I_{0}}\in
  \Spec\bigl(\bigotimes_{i\in I_{0}}A_{i}\bigr)\)
  for finite subsets \(I_{0}\subset I\).

  Suppose that \(\kappa(y)\) is not separably closed.
  Then we can find a finite separable extension~\(k'\)
  of~\(\kappa(x_{\emptyset})\)
  such that
  there is no homomorphism \(k'\to\kappa(y)\)
  over~\(\kappa(x_{\emptyset})\).
  We then spread this out to a finite étale cover~\(A'\)
  over~\(A\).
  We consider \(i\in I\)
  with \(A_{i}\simeq A'\).
  Since \(A_{i}\otimes_{A}\kappa(x_{\emptyset})\simeq k'\),
  we have \(k'\simeq\kappa(x_{\{i\}})\),
  which contradicts the choice of~\(k'\).
\end{proof}

\begin{remark}\label{xc00rl}
  The topological part
  of the assumption on~\(A\) in \cref{clean_cov}
  is equivalent
  to the condition that
  any element of~\(A\)
  can be written as a sum of an idempotent and an invertible.
  This condition,
  mainly in the context of associative rings,
  is called \emph{clean}, as first introduced by Nicholson~\cite{Nicholson77}.
  As in the local case,
  \(A\) satisfies the full assumption of \cref{clean_cov}
  if and only if every finite étale algebra over~\(A\) is clean.
\end{remark}

\begin{remark}\label{xpqmfj}
  In \cref{clean_cov},
  unlike the local case,
  we need to add closed points:
  Consider
  the subring of
  \(\Cls{C}(\NN\cup\{\infty\};\CC)\)
  consisting of functions~\(f\) with \(f(\infty)\in\RR\).
\end{remark}

\begin{example}\label{xqom5g}
  Let \(A\) be a totally disconnected.
  Then the underlying ring of~\(A\)
  satisfies the assumption of \cref{clean_cov};
  this follows
  from \cref{td_sp,xibpba}.
\end{example}

\begin{proof}[Proof of \cref{xsdudu}]
  First,
  by base changing along
  \(\ZZ\to\ZZ\{T,T^{-1}\}_{\sfrac12,2}\)
  (see \cref{xzxpzi}),
  we assume that \(A\) has nondiscrete residue fields.

  Then
  by considering the Cantor set cover
  described in \cref{x496fa}
  of \([0,\lvert a\rvert]\)
  for each~\(a\),
  we can replace~\(A\)
  with a totally disconnected.
  It is an arc~cover by \cref{x9up2q}.

  By \cref{xqom5g},
  we can apply \cref{clean_cov} to~\(A\)
  to obtain a filtered family
  of finite étale covers.
  By \cref{xhll2q},
  we regard them as a filtered family
  of totally disconnected
  finite étale over~\(A\).
  We take the colimit in \(\Cat{TDis}_{A/}\)
  to obtain the desired strictly totally disconnected~\(B\).
\end{proof}

We have the following light variant:

\begin{proposition}\label{xryxo9}
  For an \(\aleph_{1}\)-compact uniform Banach ring~\(A\),
  there exists a light strictly totally disconnected~\(B\)
  with an arc~cover \(A\to B\).
\end{proposition}

\begin{proof}
  We verify that each step in the proof of \cref{xsdudu}
  can be made to preserve lightness.

  The first step
  is valid,
  since \(\ZZ\{T,T^{-1}\}_{\sfrac12,2}\)
  is light.

  For the second step,
  by \cref{xwafwj},
  we can choose a countable sequence \(a_{0}\), \dots
  such that
  \(\Sp(A)\to\prod_{n=0}^{\infty}[0,\lvert a_{n}\rvert]\)
  is injective.
  We can then do the same process to obtain
  something light.

  For the last step,
  by \cref{xxvdc1},
  the index set~\(I\)
  in the proof of \cref{clean_cov} is countable,
  and hence the resulting ring is light. 
\end{proof}

We observe the following immediate consequence:

\begin{corollary}\label{x11ph0}
  A presheaf \(\Cat{URing}_{\aleph_{1}}\to\Cat{Ani}\)
  is an arc~(hyper)sheaf
  if and only if
  it is right Kan extended
  from~\(\Cat{STDis}_{\aleph_{1}}\to\Cat{Ani}\),
  which is an arc~(hyper)sheaf there.
  In particular,
  we have equivalences
  \begin{align*}
    \Shv((\Cat{STDis}_{\lgt})^{\op})&\simeq
    \Shv((\Cat{URing}_{\aleph_{1}})^{\op}),&
    \Shv^{\hyp}((\Cat{STDis}_{\lgt})^{\op})&\simeq
    \Shv^{\hyp}((\Cat{URing}_{\aleph_{1}})^{\op}),
  \end{align*}
  both given by right Kan extension.
\end{corollary}

\subsection{Topological finite presentation}\label{ss:tfp}

Here we consider
the following:

\begin{definition}\label{xoxb2j}
  Let \(A\) be a uniform Banach ring.
  We call a uniform Banach ring~\(B\) over~\(A\)
  \emph{topologically of finite presentation}
  if it is a quotient of
  \(A\{T_{1},\dotsc,T_{n}\}_{r_{1},\dotsc,r_{n}}\)
  by a finitely generated ideal.
  When \(A=\ZZ\), we simply say that
  \(B\) is topologically of finite presentation.
  We write \(\Cat{URing}_{\tfp}\)
  for the full subcategory of \(\Cat{URing}\).
  Equivalently,
  this is the smallest full subcategory
  containing \(\ZZ\{T\}_{r}\) for \(r>0\)
  and closed under finite colimits.
\end{definition}

We prove a certain finiteness statement
about such rings.
First,
we note the following about
residue fields:

\begin{proposition}\label{xwba2d}
  Let \(O\) be a totally imaginary number ring.
  For any complete residue field~\(K\)
  of \(\ZZ\{T_{1},\dotsc,T_{n}\}_{r_{1},\dotsc,r_{n}}\),
  we have \(\dim_{\ZZ}(\Sparc(K))\leq n+3\).
\end{proposition}

\begin{proof}
  We only have to consider
  finite extensions
  of \(\QQ(T_{1},\dotsc,T_{n})\)
  or \(\FF_{p}(T_{1},\dotsc,T_{n})\);
  namely,
  when\(K\) is the completion of~\(\kappa\),
  the absolute Galois group~\(G_{K}\)
  injects to~\(G_{\kappa}\).
  Now we can use the result from étale cohomology
  (but see~\cref{xi27fj}).
  By~\cite[Tag~0F0S]{SP},
  we are reduced to the case when \(n=0\)
  and this is classical.
\end{proof}

We also need the following later;
see, e.g.,~\cite[Théorème~7.3.7]{LemanissierPoineau24} for a proof:

\begin{proposition}\label{xul75l}
  Let \(O\) be a number ring.
  The covering dimension
  of \(O\{T_{1},\dotsc,T_{n}\}_{r_{1},\dotsc,r_{n}}\)
  for any \(r_{1}\), \dots, \(r_{n}>0\)
  is \(2n+1\).
\end{proposition}

\begin{remark}\label{x90oh5}
  When we bound
  the arc~cohomological dimension
  (see~\cite[Remark~4.11]{ScholzeB})
  of a uniform Banach ring of topological finite presentation
  by bounding
  its topological dimension
  and the cohomological dimension of residue fields separately,
  we obtain a suboptimal bound.
  With a careful argument,
  one can show, e.g., that
  \(\Sparc(O\{T_{1},\dotsc,T_{n}\}_{r_{1},\dotsc,r_{n}})\)
  for a totally imaginary number ring~\(O\),
  has cohomological dimension
  \(2n+4\)
  for any \(r_{1}\), \dots, \(r_{n}>0\).
\end{remark}

We later use the following,
which was proven in~\cite[Théorèmes~9.17 et~10.1]{Poineau13}:

\begin{theorem}[Poineau]\label{poi_noe}
  For a point~\(y\)
  of the Berkovich-analytic affine \(n\)-space over~\(\Sp(\ZZ)\),
  consider the local ring~\(\Cls{O}_{y}\).
  It is an excellent noetherian local ring
  of dimension~\(\dim(\Cls{O}_{\ZZ,x})+n\),
  where \(x\) is the image of~\(y\) in~\(\Sp(\ZZ)\).
\end{theorem}

See also~\cite[Théorème~8.6.10]{LemanissierPoineau24}
for the global noetherianness result,
which we do not need in this paper.

\section{Berkovich \texorpdfstring{\(2\)}{2}-motives}\label{s:main_ber}

In this section,
we prove \cref{main_4},
which characterizes
\(\D[2]_{\mot}(\ZZ)\)
in terms
of ring stacks equipped with an absolute value.
Throughout this section,
we write \(\Mot[2](\ZZ)\)
for the universal target of \cref{main_4}.

In \cref{ss:d_mot},
we review~\(\D_{\mot}\).
In \cref{ss:smot},
we consider the class of a seminormed ring
in \(\Mot[2](\ZZ)\).
In \cref{ss:disk,ss:arc_des},
we prove important descent results.
In \cref{ss:cc_ber},
we complete the proof of \cref{main_4}.

\subsection{A review of Berkovich motivic spectra}\label{ss:d_mot}

In~\cite{ScholzeB},
Scholze considered integral coefficients.
Here, we consider
the same category with spherical coefficients.

\begin{definition}\label{xbdjh9}
Let \(X\) be an accessible presheaf
  on \(\Cat{STDis}^{\op}\) preserving finite products.
  A \emph{finitary arc~sheaf} over~\(X\)
  is a morphism \(Y\to X\) from another such presheaf
such that
  for a cofiltered limit \(T=\projlim_{i}T_{i}\),
  the square
  \begin{equation*}
    \begin{tikzcd}
      \injlim_{i}Y(T_{i})\ar[r]\ar[d]&
      Y(T)\ar[d]\\
      \injlim_{i}X(T_{i})\ar[r]&
      X(T)
    \end{tikzcd}
  \end{equation*}
  is a pullback.
  It is equivalently
  a finitary arc~sheaf on
  \((\Cat{STDis}^{\op})_{/X}\).
\end{definition}

\begin{definition}\label{x0o7d6}
  Let \(X\) be an arc~stack.
  We define
  \begin{equation*}
    \D_{\mot}(X;\SS)
    =\Shv_{\fin}(X;\Cat{Sp})_{\DD^{1}}[(\Sigma^{\infty}\PP^{1})^{\otimes-1}].
  \end{equation*}
  First,
  \(\Shv_{\fin}(X;\Cat{Sp})\)
  is the category of \(\Cat{Sp}\)-valued finitary arc~sheaves
  on \((\Cat{STDis}^{\op})_{/X}\).
  Then we restrict to \(\DD^{1}\)-invariant sheaves.
  This means that
  for any uniform Banach ring~\(A\),
  the morphism \(A\to A\{T\}_{1}\) induces an equivalence.
  By~\cite[Corollary~5.4]{ScholzeB},
  for finitary sheaves,
  it suffices to check this
  for \(C\to C\{T\}_{1}\) where \(C\) is a nondiscrete algebraically closed Banach field
  (but still note that \(C\{T\}_{1}\) is not totally disconnected).
We omit~\(\SS\) and simply write \(\D_{\mot}(X)\)
  when there is no confusion.
\end{definition}

The following for \(\D_{\mot}(\X;\ZZ)\)
was proven in~\cite[Theorem~9.2]{ScholzeB}
(see also~\cite[Corollary~10.2]{ScholzeB}).
The same proof works in the spherical case:

\begin{theorem}[Scholze]\label{x5sz2s}
  Let \(A\to B\) be a map of uniform Banach rings
  such that \(Y=\Sparc(B)\) embeds
  into a finite-dimensional ball over \(X=\Sparc(A)\)
  and \(\sup_{x\in\Sp(A)}\dim_{\ZZ}(Y_{x})\) is finite.
  Then the functor
  \(f^{*}\colon\D_{\mot}(A)\to\D_{\mot}(B)\)
  admits a \(\D_{\mot}(A)\)-linear right adjoint~\(f_{*}\)
  which is compatible with base change.
\end{theorem}

\begin{remark}\label{xmqjp9}
  One subtle part in the proof of \cref{x5sz2s}
  is the projection formula
  for \(\PP^{1}_{X}\to X\),
  when \(X=\Sparc(C)\) for
  a nondiscrete algebraically closed Banach field~\(C\).
  This can also be deduced using \cref{ayoub}.
\end{remark}

Hence,
we can apply \cref{cll}
to obtain the following:

\begin{corollary}\label{x7y202}
  The lax symmetric monoidal functor
  \({\D_{\mot}}\colon\Cat{URing}_{\tfp}\to\Cat{Pr}_{\st}\)
  extends (uniquely)
  to
  \(\Span[2]_{\all;\all,\iso}((\Cat{URing}_{\tfp})^{\op})\).
\end{corollary}

\begin{definition}\label{x6bbzd}
  We write \(\D[2]_{\mot}(\ZZ)\)
  for the presentable \(2\)-category of kernels
  for the functor
  \(\Span((\Cat{URing}_{\tfp})^{\op})\to\Cat{Pr}_{\st}\),
  which is a restriction of the functor in \cref{x7y202}.
\end{definition}

\subsection{Disk invariance}\label{ss:disk}

We connect various notions of disk invariance:

\begin{proposition}\label{x85jpj}
  Consider a stable presentably symmetric monoidal \(2\)-category~\(\cat{C}\)
  and a stable weakly suave ring stack~\(R\)
  with an absolute value \(N\colon R\to[0,\infty)\)
  over~\(\cat{C}\).
  Assume \(\AA^{1}=R\) and \(\BB^{1}=N^{-1}([0,1))\)
  are homologically trivial.
  \begin{enumerate}
    \item\label{i:b_br}
      Then \(\BB_{r}^{1}=N^{-1}([0,r))\)
      is suave and
      homologically trivial
      for any \(r>0\).
    \item\label{i:b_dr}
      Then \(\DD_{r}^{1}=N^{-1}([0,r])\)
      is proper and
      cohomologically trivial for any \(r\geq0\).
  \end{enumerate}
\end{proposition}

\begin{proof}
  By \cref{main_2},
  the specification of the ring stack
  factors through~\(\SH[2](\ZZ)\).
  In particular,
  \(\PP^{1}=R\amalg_{R^{\times}}R\) is proper.

  We first prove~\cref{i:b_dr} for \(r=1\).
  Hence,
  the extended norm map
  \(\widetilde{N}\colon\PP^{1}=R\amalg_{R^{\times}}R\to[0,\infty]\)
  is proper.
  Therefore,
  \(\DD^{1}=N^{-1}([0,1])=\widetilde{N}^{-1}([0,1])\)
  is proper.
  Its cohomology can be computed by the assumption
  and the recollement induced by
  \(\BB^{1}\to\PP^{1}\gets\DD^{1}\),
  which is the base change of
  \([0,1)\to[0,\infty]\gets[1,\infty]\)
  along~\(\widetilde{N}\).

  We now prove~\cref{i:b_dr} for general~\(r\geq0\).
  Its properness follows from the same argument.
  To compute its cohomology,
  we can use the multiplication map
  \(\DD^{1}\times\DD_{r}^{1}\to\DD_{r}^{1}\),
  which gives a \(\DD^{1}\)-homotopy between~\(\id\) 
  on~\(\DD_{r}^{1}\)
  and the constant map~\(0\).

  Finally,
  we prove~\cref{i:b_br}.
  First, \(\BB_{r}^{1}\) is suave,
  since \([0,r)\to[0,\infty)\) is suave.
  To compute its homology,
  as in the previous argument,
  we realize~\(\BB_{r}^{1}\)
  as the complement of \(\DD_{\sfrac1r}^{1}\)
  inside~\(\PP^{1}\).
  Now the result follows from~\cref{i:b_dr}.
\end{proof}

This proves the existence
of a pseudouniformizer up to descent:

\begin{corollary}\label{xlcq6d}
  In the situation of \cref{x85jpj},
  the morphism \(N^{-1}(\{r\})\to{*}\) is 
  a cover for any \(r>0\).
\end{corollary}

\begin{proof}
  We know that~\(\PP^{1}\) is proper
  and therefore,
  the extended norm map~\(\widetilde{N}\)
  is also proper.
  This implies that
  \(f\colon N^{-1}(\{r\})\to{*}\) is proper
  and therefore by \cref{xpbfpq},
  it suffices to show that
  \(f_{*}\unit\in\End_{\cat{C}}(\unit)\)
  satisfies descent.
  We compute this
  by regarding it as
  the intersection
  of \(N^{-1}([0,r])\cap\widetilde{N}^{-1}([r,\infty])\)
  inside~\(\PP^{1}\).
  By \cref{i:b_dr},
  we can compute it as
  \(\unit\oplus\unit(-1)[-1]\),
  and therefore it admits a section,
  and by \cref{x059gc},
  it satisfies descent.
\end{proof}

\begin{remark}
  Some weaker variants of \cref{xlcq6d}
  remain true without assuming disk invariance:
  Given
  a stable homologically trivial weakly suave ring stack
  with an absolute value~\(N\),
  one can show that \(N^{-1}((r,s))\to{*}\)
  is a cover for any \(0<r<s\).
  Nevertheless,
  \(\AA^{1}\)-invariance
  is strictly weaker than \(\BB^{1}\)-invariance.
\end{remark}

\begin{remark}\label{xm1m2w}
  At least after proving \cref{main_4},
  one can show with extra work that
  the norm map \(N\colon\AA^{1}\to[0,\infty)\) is a cover
  over \(\Mot[2](\ZZ)\).
\end{remark}

\subsection{\texorpdfstring{\(2\)}{2}-motives of seminormed rings}\label{ss:smot}

By \cref{main_3},
we obtain a morphism \(\SH[2]_{\et}(\ZZ)\to\Mot[2](\ZZ)\).
Therefore, we obtain the class for each ring:

\begin{definition}\label{x0qgj9}
  For a ring~\(A\),
  we write \([A]_{\alg}\in\CAlg(\Mot[2](\ZZ))\)
  for the image
  of~\([A]\in\CAlg(\SH[2]_{\et}(\ZZ))\).
\end{definition}

In our situation,
we can define the motive of each seminormed ring
as follows:

\begin{definition}\label{xucjcx}
  For a seminormed ring~\(A\),
  we consider the pushout
  \begin{equation*}
    \begin{tikzcd}
      \Shv([0,\infty]^{\pi_{0}(A)})\otimes\unit\ar[r]\ar[d]&
      {[A]_{\alg}}\ar[d]\\
      \Shv(\Sp(A))\otimes\unit\ar[r]&
      {[A]}
    \end{tikzcd}
  \end{equation*}
  in \(\CAlg(\Mot[2](\ZZ))\).
  Since \([A]_{\alg}\to[A]\) is an epimorphism,
  by \cref{xis10l} below,
  this assembles into a functor
  \([\X]\colon\Cat{SNRing}\to\CAlg(\Mot[2](\ZZ))\).
\end{definition}

\begin{lemma}\label{xis10l}
  Let \(F\colon\cat{C}\to\cat{D}\) be a functor.
  For \(C\in\cat{C}\),
  we have a monomorphism \(F'_{C}\to F(C)\).
  Assume that for any \(C\to C'\),
  the composite
  \(F'_{C}\to F(C)\to F(C')\)
  factors through~\(F'_{C'}\).
  Then this family uniquely assembles into a functor
  \(F'\colon\cat{C}\to\cat{D}\)
  and a natural transformation \(F'\to F\).
\end{lemma}

\begin{proof}
  By Yoneda, we may assume \(\cat{D}=\Cat{Ani}\):
  More precisely,
  we consider \(G\colon\cat{C}\times\cat{D}^{\op}\to\Cat{Ani}\)
  given by \((C,D)\mapsto\Map(D,F(C))\)
  and \(G'_{C,D}=\Map(D,F'_{C})\).
  In this case,
  we can construct \(F'_{0}\colon\cat{C}\to\Cat{Set}\)
  as a subfunctor of \(\pi_{0}F\colon\cat{C}\to\Cat{Set}\).
  The functor \(F'=F'_{0}\times_{\pi_{0}(F)}F\)
  satisfies the desired property.
  The uniqueness also follows from the construction.
\end{proof}

\begin{proposition}\label{x925qh}
  The functor \([\X]\colon\Cat{SNRing}\to\CAlg(\Mot[2](\ZZ))\)
  factors through \(\Cat{SNRing}^{\heartsuit}\)
  and preserves colimits.
\end{proposition}

\begin{proof}
  The first claim follows directly from the definition.
  We prove the second claim about colimits.
  For filtered colimits,
  it follows from the definition and \cref{i:sp_fil}
  of \cref{xfdud8}.

  We first prove that it preserves finite coproducts. 
  For the nullary case,
  we need to show that
  \(\Shv([0,\infty]^{\ZZ})\to\End_{\Mot[2](\ZZ)}(\unit)\)
  factors through \(\Shv(\Sp(\ZZ))\).
  For this,
  we recall that \(\Sp(\ZZ)\subset\prod_{n}[0,\infty]\)
  is cut out by the multiplicative seminorm axioms.
  Therefore,
  this follows from the definition of an absolute semivalue.
Similarly,
  we can see that it preserves binary coproducts.

  We then prove
  that it preserves pushouts
  \(A'\otimes_{A}B\)
  to complete the proof.
  By \(A'\otimes_{A}B=A\otimes_{A\otimes A}(A'\otimes B)\),
  it suffices to treat the case
  when \(A\to A'\) is a surjection.
In this case,
  \(\Sp(A')\to\Sp(A)\)
  is the pullback of
  the diagonal inclusion
  \([0,\infty]^{A'}\to[0,\infty]^{A}\)
  and hence
  \(\Sp(A'\otimes_{A}B)\simeq\Sp(A')\otimes_{\Sp(A)}\Sp(B)\).
  From this,
  we observe the desired result.
\end{proof}

We see some examples:

\begin{example}\label{x099iq}
  We describe the direct relation between this construction
  and the norm \(N\colon R\to[0,\infty)\).
  We claim that \([\ZZ[T]_{r}]\) for \(r\geq0\)
  corresponds to \(N^{-1}([0,r])\).
  To see this,
  we need to see that \([\ZZ[T]_{r}]\)
  inside \([\ZZ[T]]_{\alg}=[R]\)
  is solely cut out from~\(T\).
  This follows from an argument
  similar to the proof of \cref{x925qh};
  all the inequalities are generated by \(\lvert T\rvert\leq r\).
  Similarly,
  \([\ZZ[T,T^{-1}]_{s,\sfrac1r}]\)
  for \(r\leq s\)
  corresponds to \(N^{-1}([r,s])\).
\end{example}

\begin{remark}\label{x0cq9b}
  One can give another definition of
  the \(2\)-motive of a seminormed ring
  via \cref{x099iq}.
  Specifically,
  we define
  \(\Cat{SNPol}\to\Mot[2](\ZZ)\)
  by sending
  \(\ZZ[T_{1},\dotsc,T_{n}]_{r_{1},\dotsc,r_{n}}\)
  to \(N^{-1}([0,r_{1}])\times\dotsb\times N^{-1}([0,r_{n}])\).
  Then we can check
  the relations described in \cref{xywaym}
  are satisfied,
  so that this extends to a functor
  \(\Cat{SNRing}\to\Mot[2](\ZZ)\).
  By \cref{x925qh},
  this definition recovers our definition here.
\end{remark}

\begin{example}\label{x5i50v}
  Let \(l\) be a prime.
  Consider \(\ZZ[\sfrac1l]_{1}\)
  to be the universal one with \(\lvert\sfrac1l\rvert\leq1\).
  The motive of this is
  just the base change of~\(\unit\)
  along the morphism
  \(\Shv(\Sp(\ZZ))\otimes\unit\to\Shv(\Sp(\ZZ[\sfrac1l]_{1}))\otimes\unit\).
\end{example}

\begin{example}\label{x4y9d2}
  Let \(A\to A/I\) be a quotient of (static) seminormed rings.
  In this case,
  \begin{equation*}
    \begin{tikzcd}
      {[A]_{\alg}}\ar[r]\ar[d]&
      {[A/I]_{\alg}}\ar[d]\\
      {[A]}\ar[r]&
      {[A/I]}
    \end{tikzcd}
  \end{equation*}
  is a pushout in \(\CAlg(\Mot[2](\ZZ))\).
  Since both commute with colimits,
  it suffices to show this
  for \(\ZZ[T]_{r}\to\ZZ\)
  for \(r\geq0\),
  which follows from \cref{x099iq}.
\end{example}

\begin{proposition}\label{x2a9p8}
  For any seminormed ring~\(A\),
  the stack corresponding to \([A]\) is static and proper.
\end{proposition}

\begin{proof}
  We may assume \(A\) is static by \cref{x925qh}.
  We write it
  as a filtered colimit
  of algebras of the form
  \(\ZZ[T_{1},\dotsc,T_{n}]_{r_{1},\dotsc,r_{n}}/(P_{1},\dotsc,P_{m})\)
  to reduce to the case of such rings.
  Moreover, by \cref{x4y9d2},
  we assume \(m=0\).
  Now we are reduced to the case of \(\ZZ[T]_{r}\).
  The corresponding stack
  is \(N^{-1}([0,r])\) by \cref{x099iq}.
  We now consider
  the extended norm map
  \(\widetilde{N}\colon\PP^{1}\to[0,\infty]\),
  which is obtained by gluing \(N\) and~\(N^{-1}\).
  Since \(\PP^{1}\) is static and proper already over \(\SH[2](\ZZ)\),
  the desired result follows.
\end{proof}

Note that uniform completion is automatic in this process:

\begin{theorem}\label{xh9vw0}
  The functor
  \(\Cat{SNRing}\to\Mot[2](\ZZ)\)
  in \cref{xucjcx}
  factors through
  the uniform completion functor \(\Cat{SNRing}\to\Cat{URing}\).
\end{theorem}

\begin{proof}
  We prove that it factors
  through the categories
  of normed rings, Banach rings, and subsequently uniform Banach rings.
  To prove that,
  it suffices to show
  that each morphism in \cref{xj6u9m}
  is sent to an equivalence.

We consider~\cref{i:li_n};
  i.e., we wish to show that
  \(\ZZ[T]_{0}\to\ZZ\)
  is sent to an equivalence.
  This follows from \cref{x099iq}.

  Then we consider~\cref{i:li_b}.
  We fix \(r\gg0\).
  The situation is that
  \(A\) is \(\ZZ[T_{1},\dotsc]_{r,\dotsc}\)
  with \(\lvert T_{m}-T_{m+k}\rvert\leq\sfrac1m\)
  for \(m\geq1\) and \(k\geq1\)
  and \(B\) is \(A[U]_{r}\)
  with \(\lvert U-T_{m}\rvert\leq\sfrac1m\)
  for \(m\geq1\).
  We wish to show that
  \([A]\to[B]\) is an equivalence.
  By \cref{x2a9p8}
  and considering its diagonal
  (cf.~our proof of \cref{xm89d9}),
  it suffices to show that the map \([A]\to[B]\)
  satisfies descent.
  By \cref{xpbfpq},
  it reduces to the question
  of \(f_{*}\unit\in\Hom_{\Mot[2](\ZZ)}(\unit,[A])\)
  satisfying descent.
  By \cref{x670j7},
  it suffices to show that it is descendable.
  We first write
  \(A\to B\) as
  a sequential colimit of \(f_{n}\colon A_{n}\to B_{n}\)
  for \(n\geq0\)
  by only considering~\(T_{1}\), \dots,~\(T_{n}\).
  But this map \(A_{n}\to B_{n}\) admits a section
  given by \(U\mapsto T_{n}\),
  and hence
  \(f_{n,*}\unit\in\Hom_{\Mot[2](\ZZ)}(\unit,[A_{n}])\)
  is descendable.
  Therefore, by \cref{md_seq},
  we obtain the desired result.

  For~\cref{i:li_u},
  we have nothing to prove.
\end{proof}

\begin{example}\label{xbi0h6}
  Let \(A\to B\) be a finite étale morphism
  of totally disconnecteds.
  Then the square
  \begin{equation*}
    \begin{tikzcd}
      {[A]_{\alg}}\ar[r]\ar[d]&
      {[B]_{\alg}}\ar[d]\\
      {[A]}\ar[r]&
      {[B]}
    \end{tikzcd}
  \end{equation*}
  is a pushout in \(\CAlg(\Mot[2](\ZZ))\).
\end{example}

\subsection{Arc~descent}\label{ss:arc_des}

We prove two descent results in \(\Mot[2](\ZZ)\).

\begin{theorem}\label{xrp8oy}
  Let \(f\colon A\to B\) be an arc~cover in \(\Cat{STDis}_{\lgt}\).
  Then \(f^{*}\colon[A]\to[B]\) satisfies descent in \(\Mot[2](\ZZ)\).
  In other words,
  the functor \(\Cat{STDis}_{\lgt}\to2\Cat{Pr}\)
  given by \(A\mapsto\Mot[2](A)\)
  is an arc~sheaf.
\end{theorem}

\begin{remark}\label{xi7cdy}
  In \cref{xrp8oy},
  by considering
  uniform Banach rings \(\Cls{C}(X;\CC)\)
  for compact Hausdorff spaces~\(X\),
  we can see from \cref{xzad7y}
  that
  the size assumption cannot be dropped,
  and from \cref{x65y5q}
  that
  we cannot take \(\Cat{UBan}_{\aleph_{1}}\) instead.
  We see in \cref{xif2s7} below
  that descent is still valid
  under the topological finite presentation assumption.
\end{remark}

\begin{proof}
  By \cref{x2a9p8},
  the morphism is prim
  and hence,
  by \cref{xpbfpq},
  it suffices to show that
  \(f_{*}\unit\in\Hom_{\Mot[2](\ZZ)}(\unit,[A])\)
  satisfies descent.
  By \cref{x670j7},
  it suffices to show that it is descendable.

  By~\cite[Lemma~4.8]{ScholzeB},
  we can write~\(B\) as a filtered colimit
  \(\injlim_{i}B_{i}\) of uniform Banach rings over~\(A\)
  such that each \(A\to B_{i}\) admits a splitting.
By the \(\aleph_{1}\)-compactness,
  up to retract,
  we assume that the colimit is sequential.
  Hence,
  the desired result follows from \cref{md_seq}.
\end{proof}

\begin{theorem}\label{x0606p}
  For \(B=\ZZ\{T_{1},\dotsc,T_{n}\}_{r_{1},\dotsc,r_{n}}\),
  there is a map to a light strictly totally disconnected~\(B''\)
  such that \(f^{*}\colon[B]\to[B'']\)
  satisfies descent in \(\Mot[2](\ZZ)\).
\end{theorem}

We need to do a slightly intricate construction for this:

\begin{proof}[Proof of \cref{x0606p}]
  We first replace~\(B\)
  with \(B\{U,U^{-1}\}_{\sfrac12,2}\);
  this substitution is justified by \cref{xlcq6d}.
We then consider
  \(A=\ZZ\{T_{1},\dotsc,T_{n},U,U^{-1}\}^{\dagger}_{r_{1},\dotsc,r_{n},\sfrac12,2}\),
  its overconvergent counterpart.
  We have the completion morphism \(A\to B\).

  We then consider \(B\to B'\) to
  be the map constructed in \cref{xryxo9},
  where we use a sequence of elements in \(\ZZ[T_{1},\dotsc,T_{n},U^{\pm}]\)
  to ensure the existence of an overconvergent counterpart \(A\to A'\).
  With this construction,
  we obtain a pushout square
  \begin{equation*}
    \begin{tikzcd}
      \Shv(\Sp(B))\otimes\unit\ar[r]\ar[d]&
      \Shv(\Sp(B'))\otimes\unit\ar[d]\\
      {[B]}\ar[r]&
      {[B']}
    \end{tikzcd}
  \end{equation*}
  in \(\CAlg(\Mot[2](\ZZ))\),
  and therefore,
  by \cref{xul75l,x1lqb7},
  the morphism \([B]\to[B']\) satisfies descent.

  Note that
  \(A'\) satisfies the assumption of \cref{clean_cov}.
  Moreover,
  \(A'\) also satisfies
  the assumption of \cref{bbx_more}.
  Since its local rings at maximal points
  are local rings of the Berkovich-analytic affine \((n+1)\)-space
  over~\(\ZZ\),
  we need to bound the cohomological dimension of such local rings
  uniformly;
  this is possible,
  but
  our argument here uses the existence
  of a topologically nilpotent unit.
  The noetherianness is from \cref{poi_noe}.
  Now it suffices to bound \(l\)-torsion cohomological dimension
  for prime~\(l\) uniformly.
  When \(l\) is invertible in~\(\Cls{O}_{y}\),
  the affine Lefschetz theorem of Gabber~\cite[Corollaire~XV-1.2.4]{ILO14}
  gives an optimal bound on \(l\)-torsion cohomological dimension.
  So the only case left to bound
  is when the point lies over a closed point of \(\Spec\ZZ\).
  For this,
  we write it as a filtered colimit
  of Banach algebras over~\(\ZZ_{p}\).
Then for each term,
we can use~\cite[Theorem~1.9]{Hansen20},
  which is based on the affinoid comparison theorem of Huber~\cite{Huber96},
  to obtain the desired bound,
  derived from the bound for analytic étale cohomology
  (cf.~\cref{x90oh5});
  here, we use the existence of a topologically nilpotent unit.

  By applying \cref{clean_cov},
  we therefore obtain a morphism \(A'\to A''\).
  By \cref{x7c1hf},
  we obtain a pushout square
  \begin{equation*}
    \begin{tikzcd}
      \Shv_{\fet}^{\wedge}(A')\otimes\unit\ar[r]\ar[d]&
      \Shv_{\fet}^{\wedge}(A'')\otimes\unit\ar[d]\\
      {[A']}\ar[r]&
      {[A'']}
    \end{tikzcd}
  \end{equation*}
  in \(\SH[2]_{\et}(\ZZ)\).
  In this situation,
  by \cref{clausen},
  we see that
  \(\Shv_{\fet}^{\wedge}(A')\to\Shv_{\fet}^{\wedge}(A'')\)
  satisfies descent over \(\Cat{Pr}_{\st}\).
  Therefore,
  the morphism \([A']\to[A'']\)
  satisfies descent over \(\SH[2]_{\et}(\ZZ)\),
  and hence
  the morphism \([A']_{\alg}\to[A'']_{\alg}\)
  satisfies descent.

  We now base change
  \(A'\to A''\) along \(A'\to B'\)
  to obtain a totally disconnected~\(B''\).
  This is possible by \cref{xhll2q}
  and the fact
  that \(A''\) is a filtered colimit of finite étale algebras over~\(A'\).
  We hence obtain a square
  \begin{equation*}
    \begin{tikzcd}
      {[A']_{\alg}}\ar[r]\ar[d]&
      {[A'']_{\alg}}\ar[d]\\
      {[B']}\ar[r]&
      {[B'']}\rlap,
    \end{tikzcd}
  \end{equation*}
  which is a pushout square by \cref{xbi0h6}.
  Hence, the morphism \([B]\to[B'']\) satisfies descent.

  We are now left with proving that
  \(B''\) is a light strictly totally disconnected.
  We first prove that
  it is a light totally disconnected.
  By construction,
  \(B'\) is a light totally disconnected,
  as in the proof of \cref{xryxo9}.
  To prove the same for~\(B''\),
  we need to show that
  \(A''\) is a sequential colimit of
  finite étale algebras over~\(A'\).
  By construction,
  we have to see that
  there are countably many
  finite étale algebras up to isomorphism
  over~\(A'\).
  As in the argument of \cref{xxvdc1},
  we need to show that for any direct summand~\(A'_{0}\) of~\(A'\),
  there are countably many finite étale algebras
  of the form \(A'_{0}[T]/(T^{n}+\dotsb+a_{0})\).
  Since \(A'_{0}\) is a sequential colimit
  of finite products
  of overconvergent algebras 
  on the subspace of
  the Berkovich-analytic affine \((n+1)\)-space
  over~\(\ZZ\) cut out
  by finitely many inequalities,
  it suffices to show the same property for such algebras.
  Each of them is a sequential colimit
  of Banach rings topologically of finite type over~\(\ZZ\).
  For such, the desired property has been observed in
  the proof of \cref{xxvdc1}.
  Therefore, it remains to show that
  all the complete residue fields of~\(B''\) are algebraically closed,
  which we can observe fiberwise on \(\Sp(B')\).
\end{proof}

Combining \cref{xrp8oy,x0606p},
we obtain the following:

\begin{corollary}\label{xif2s7}
  Let \(A\to B\) be an arc~cover in \(\Cat{URing}_{\tfp}\).
  Then \([A]\to[B]\) satisfies descent in \(\Mot[2](\ZZ)\).
\end{corollary}

\begin{proof}
  By \cref{xrp8oy},
  it suffices to show that
  any \(A\in\Cat{URing}_{\tfp}\)
  admits a morphism \(A\to A'\)
  such that \(A'\) is a light strictly totally disconnected.
  We can take the cover in \cref{x0606p}
  and take its quotient by the defining ideal.
\end{proof}

\subsection{Berkovich \texorpdfstring{\(2\)}{2}-motives and normed ring stacks}\label{ss:cc_ber}

We prove \cref{main_4}.
We first introduce the following categories of geometric objects:
We write \(\Cat{Stk}_{\lgt}^{\vee}\)
for the category
of arc~sheaves on \((\Cat{STDis}_{\lgt})^{\op}\);
the notation indicates that hyperdescent is not imposed.
We write \(\Cat{Var}\) for
the essential image of \((\Cat{URing}_{\tfp})^{\op}\)
in \(\Cat{Stk}_{\lgt}^{\vee}\).
We note that
\(\D[2]_{\mot}(\ZZ)\)
is also the presentable category of kernels
for \({\D_{\mot}}\colon\Span(\Cat{Var})\to\Cat{Pr}_{\st}\).

In \cref{ss:smot},
we constructed
a symmetric monoidal functor \(\Cat{SNRing}\to\Mot[2](\ZZ)\).
We first restrict this
to \(\Cat{STDis}_{\aleph_{1}}\).
Its right Kan extension,
whose existence follows from~\cite[Theorem~D]{n-pr},
determines
a symmetric monoidal functor
\((\Cat{Stk}_{\lgt}^{\vee})^{\op}\to\Mot[2](\ZZ)\)
by \cref{xrp8oy}.

We begin with the following:

\begin{proposition}\label{xammc1}
  The morphism
  \((\Cat{Stk}_{\lgt}^{\vee})^{\op}\to\Mot[2](\ZZ)\)
  constructed above
  carries a \(\DD^{1}\)-invariant finitary morphism
  to a weakly suave morphism.
\end{proposition}

We recall the following,
which is useful in various places:

\begin{lemma}\label{fd_co}
  Let \(F\) be a finitary arc~stack over~\(A\).
  Then for any filtered colimit \(B=\injlim_{i}B_{i}\)
  of uniform Banach rings over a strictly totally disconnected~\(A\)
  with uniformly bounded arc~cohomological dimension,
  the morphism
  \begin{equation*}
    \injlim_{i}\Map(\Sparc(B_{i}),F)
    \to
    \Map(\Sparc(B),F)
  \end{equation*}
  is an equivalence.
\end{lemma}

\begin{proof}
  This is the unstable analog
  of~\cite[Remark~4.11]{ScholzeB}:
  When \(F\) is truncated, the claim is straightforward;
  see~\cite[Lemma~4.7]{ScholzeB}.
  In general,
  we write \(F=\varprojlim_{n}\tau_{\leq n}F\)
  and
  observe that for each~\(k\),
  \begin{equation*}
    \bigl(\pi_{k}\Map(\Sparc(B_{i}),\tau_{\leq n}F)\bigr)_{n}
  \end{equation*}
  stabilizes
  uniformly in~\(i\);
  this follows from our assumption
  on the cohomological dimension.
\end{proof}

\begin{lemma}\label{xabply}
  Let \(D\) be a uniform Banach ring
  topologically of finite presentation
  over a strictly totally disconnected~\(A\).
  Then for any finitary arc~sheaf,
  \(B\mapsto F(B\otimes_{A}D)\)
  determines a finitary arc~sheaf.
\end{lemma}

\begin{proof}
  This follows from \cref{fd_co}.
\end{proof}

\begin{lemma}\label{x1drqd}
  Any static finitary arc~stack~\(F\)
  over a strictly totally disconnected~\(A\)
  is a colimit of
  open substacks of~\(\BB_{r,A}^{n}\).
\end{lemma}

\begin{proof}
  By \cref{fd_co},
  we can cover~\(F\)
  by \(U=\coprod_{i}U_{i}\),
  where \(U_{i}\)
  is an open substack of~\(\BB^{n}_{r,A}\).
  We then consider the Čech nerve of \(U\to F\).
  For \(m\geq0\),
  since \(U^{\times_{F}(m+1)}\to U^{\times(m+1)}\)
  is the base change
  of \(F\to F^{\times(m+1)}\),
  which is a finitary monomorphism,
  we see that \(U^{\times_{F}(m+1)}\)
  is a disjoint union of
  open substacks of~\(\BB^{n}_{r,A}\).
\end{proof}

\begin{lemma}\label{xqduiq}
  Let \(F\) be a finitary arc~stack
  over a strictly totally disconnected~\(A\).
  If \(F\) is \((-1)\)-truncated
  over the nonarchimedean locus \(\lvert2\rvert\leq1\),
  it is a colimit
  of open substacks of~\(\Sparc(A)\).
\end{lemma}

\begin{proof}
  We consider the colimit~\(F'\)
  of all open substacks of~\(\Sparc(A)\)
  mapping to~\(F\).
  We wish to show that
  the tautological morphism \(F'\to F\) is an equivalence.
  Since both~\(F\) and~\(F'\) are finitary,
  it suffices to show this after pulling back to
  each~\(x\in\Sp(A)\).
  Since open substacks spread out,
  we assume that \(A=C\)
  is an algebraically closed nondiscrete Banach field.
  When \(A\) is archimedean,
  \(F\) is just an anima.
  When \(A\) is nonarchimedean,
  \(F\) itself is an open substack by assumption.
\end{proof}

\begin{proof}[Proof of \cref{xammc1}]
  By writing the target as a colimit,
  we assume it is \(\Sparc(A)\)
  for a strictly totally disconnected~\(A\).

  For any morphism of finitary arc~stacks \(G\to G'\),
  by \cref{fd_co},
  it is an equivalence
  if and only if \(\Map(U,G)\to\Map(U,G')\)
  is an equivalence
  for any open substack~\(U\) of~\(\BB^{n}_{r,A}\).
  Therefore,
  it suffices to show that
  its \(\DD^{1}\)-homotopification
  \(L_{\DD^{1}}(U)\),
  which is finitary by~\cite[Lemma~5.11]{ScholzeB},
  is weakly suave
  for any such~\(U\).
  Moreover,
  by writing it as a union,
  we assume that
  \(U\) is of the form \(U'+\BB_{\epsilon,A}^{n}\subset\AA^{n}_{A}\)
  for some open substack~\(U'\) of~\(\BB^{n}_{r,A}\)
  and some \(\epsilon>0\).

  We now fix such~\(U\) and \(\epsilon>0\)
  together with the embedding \(U\subset\AA^{n}_{A}\).
  We take a small positive number~\(\epsilon\).
  We consider \(U_{0}=U\)
  and the open substack~\(U_{1}\)
  of~\(U\times U\)
  consisting of pairs whose distance is \(<\epsilon\).
  We take its \(1\)-coskeleton to obtain~\(U_{\bullet}\).
  Now, \(U_{k}\) is an open substack of~\(\BB^{n(k+1)}_{r,A}\),
  and hence weakly suave.
  Therefore, \(F=\lvert U_{\bullet}\rvert\),
  which is finitary by an argument similar to~\cite[Lemma~5.11]{ScholzeB},
  is weakly suave.
  Note that in the nonarchimedean locus \(\lvert2\rvert\leq1\),
  this is equivalent to the action groupoid
  of the translation action of~\(\BB_{\epsilon,A}^{n}\)
  on~\(U\).
  Therefore, it is static there.

  We wish to show that \(L_{\DD^{1}}F\) is weakly suave.
  We write~\(F^{\DD^{k}}\)
  for the stack obtained by \cref{xabply}
  for \(D=A\{T_{1},\dotsc,T_{k}\}\).
  Since \(L_{\DD^{1}}F\) can be computed as
  the geometric realization of \(F^{\DD^{\bullet}}\),
  it suffices to show that
  they are weakly suave.
  As in the proof of \cref{x1drqd},
  we first take a surjection
  from \(V=\coprod_{i}V_{i}\),
  where \(V_{i}\) is an open substack of~\(\BB^{n}_{r,A}\).
  By considering the Čech nerve,
  we are now reduced to showing that
  \(V_{i_{0}}\times_{F^{\DD^{k}}}\dotsb\times_{F^{\DD^{k}}}V_{i_{l}}\)
  is weakly suave.
  It suffices to show that
  the tautological map
  \(V_{i_{0}}\times_{F^{\DD^{k}}}\dotsb\times_{F^{\DD^{k}}}V_{i_{l}}
  \to V_{i_{0}}\times\dotsb\times V_{i_{l}}\),
  since its target is weakly suave.
  This map is a base change of \(F^{\DD^{k}}\to(F^{\DD^{k}})^{l+1}\),
  which is weakly suave by \cref{xqduiq}.
\end{proof}

\begin{proof}[Proof of \cref{main_4}]
We first
  construct a morphism
  \(F\colon\Mot[2](\ZZ)\to\D[2]_{\mot}(\ZZ)\).
  To do so,
  we need to construct a ring stack with an absolute value
  over \(\D[2]_{\mot}(\ZZ)\).
  First,
  by composing \((\X)^{\ucpl}\colon\Cat{SNPol}\to\Cat{URing}_{\tfp}\),
  we obtain a functor
  \(\Cat{SNPol}\to\D[2]_{\mot}(\ZZ)\).
  Then, right Kan extending
  along the forgetful functor \(\Cat{SNPol}\to\Cat{Pol}\)
  (which is possible by~\cite[Theorem~D]{n-pr}),
  we obtain \(\Cat{Pol}\to\D[2]_{\mot}(\ZZ)\),
  which defines a ring stack over \(\D[2]_{\mot}(\ZZ)\).
The norm can then be concretely constructed as
  a symmetric monoidal functor
  \begin{equation*}
    \Shv([0,\infty))
    \to\Hom_{\D[2]_{\mot}(\ZZ)}(\unit,[\AA^{1}])
    \simeq\D_{\mot}(\AA^{1})
    =\projlim_{r}\D_{\mot}(\DD^{1}_{r})
  \end{equation*}
  compatible with the multiplication,
  which is a condition.
  We can use \cref{xyvgm4} for this
  by setting
  \begin{align*}
    D_{r}&=i_{r,*}\unit,&
    E_{r}&=k_{r,*}\unit
  \end{align*}
  where \(i_{r}\colon\DD^{1}_{r}\to\AA^{1}\)
  and \(k_{r}\colon\AA^{1}\setminus\BB^{1}_{r}\to\AA^{1}\),
  where \(\BB^{1}_{r}\) denotes the open disk of radius~\(r\).

  Next, we construct a morphism
  \(G\colon\D[2]_{\mot}(\ZZ)\to\Mot[2](\ZZ)\)
  as in the proof of \cref{main_1}.
  First,
  we have \({\Cat{Stk}_{\lgt}^{\vee}}^{\op}\to\Mot[2](\ZZ)\)
  by \cref{xrp8oy}.
  Since any object in \(\Cat{Var}\)
  determines a proper stack over \(\Mot[2](\ZZ)\)
  by \cref{xif2s7,x2a9p8},
  we apply \cref{cll}
  to obtain
  a symmetric monoidal functor
  \(\Span[2]_{\all;\all,\iso}(\Cat{Var})\to\Mot[2](\ZZ)\).
  To construct~\(G\),
  it remains to construct
  \begin{equation*}
    \D_{\mot}(\X)\to\Hom_{\Mot[2](\ZZ)}(\unit,[\X])
  \end{equation*}
  in \(\CAlg(\Fun(\Span(\Cat{Var}),\Cat{Pr}_{\st}))\).
To do this,
  we apply \cref{cll}
  when \(J\) consists of \(\DD^{1}\)-invariant finitary morphisms
  to obtain a symmetric monoidal functor
  \(\Span[2]_{J;\iso,J}(\Cat{Stk}_{\lgt}^{\vee})\to\Mot[2](\ZZ)\).
  By considering \(\Hom(\unit,\X)\)
  on both sides,
  we obtain functors
  \begin{equation*}
    \Shv_{\fin}(X)_{\DD^{1}}\to\Hom_{\Mot[2](\ZZ)}(\unit,[X])
  \end{equation*}
  natural in \(X\in(\Cat{Stk}_{\lgt}^{\vee})^{\op}\).
  This functor preserves colimits
  by construction.
  Moreover,
  this factors through as
  \begin{equation*}
    \D_{\mot}(X)\to\Hom_{\Mot[2](\ZZ)}(\unit,[X])
  \end{equation*}
  by the axioms.
  To construct~\(G\),
  we still need to extend this natural transformation
  to \(\Span(\Cat{Var})\).
  We obtain this as the restriction of the symmetric monoidal functor
  from \(\Span[2]_{\all;\all,\iso}(\Cat{Var})\),
  which we can construct by \cref{cll};
  we check the condition in \cref{finally} below.

We then prove that \(GF\) is homotopic to~\(\id\).
  We see that
  it preserves the universal ring stack;
  formally,
  this means that it is an equivalence after
  composing with~\(\Cat{Pol}\to\Mot[2](\ZZ)\),
  which specifies the ring stack.
  Then we can match the norms to see that
  \(GF\) is homotopic to~\(\id\).
  The fact that the norm is preserved
  is just a condition,
  and can be checked straightforwardly.

For~\(FG\),
  as in the proof of \cref{main_1},
  it suffices to show that
  this is an equivalence
  after composing with \(\Cat{Var}^{\op}\to\D[2]_{\mot}(\ZZ)\).
  This amounts to showing that
  the composite
  \begin{equation*}
    \Cat{URing}_{\tfp}
    \to
    \Cat{Var}^{\op}
    \hookrightarrow
    (\Cat{Stk}_{\lgt}^{\vee})^{\op}
    \to\Mot[2](\ZZ)
    \xrightarrow{F}\D[2]_{\mot}(\ZZ)
  \end{equation*}
  is equivalent to the structure morphism.
  This is equivalent to the composite
  \begin{equation*}
    \Cat{URing}_{\tfp}
    \hookrightarrow
    \Cat{URing}
    \to
    \Mot[2](\ZZ)
    \xrightarrow{F}\D[2]_{\mot}(\ZZ).
  \end{equation*}
  by \cref{xrp8oy,x0606p}.
  By the construction of~\(F\),
  when we compose this
  with \(\Cat{SNPol}\to\Cat{URing}_{\tfp}\),
  we obtain the structure morphism.
  Since the structure morphism
  \(\Cat{URing}_{\tfp}\to\D[2]_{\mot}(\ZZ)\)
  preserves finite colimits
  by \cref{x15q1p},
  we obtain the desired result.
\end{proof}

In the proof above,
we postponed the following technical step:

\begin{lemma}\label{finally}
  For any \(f\colon Y\to X\) in~\(\Cat{Var}\),
  the square
  \begin{equation}
    \label{e:9rdu6}
    \begin{tikzcd}
      \D_{\mot}(X)\ar[r,"f^{*}"]\ar[d]&
      \D_{\mot}(Y)\ar[d]\\
      \Hom_{\Mot[2](\ZZ)}(\unit,[X])\ar[r,"f^{*}"]&
      \Hom_{\Mot[2](\ZZ)}(\unit,[Y])
    \end{tikzcd}
  \end{equation}
  is right adjointable.
\end{lemma}

\begin{proof}
  By construction,
  the square~\cref{e:9rdu6} is left adjointable
  when \(f\) is a \(\DD^{1}\)-invariant finitary morphism
  in \(\Cat{Stk}_{\lgt}^{\vee}\).

  By considering the complement,
  we see that
  \cref{e:9rdu6} is right adjointable
  for any closed immersion~\(f\)
  in \(\Cat{Var}\).
  Therefore,
  we are reduced to the case
  \(f\colon\PP_{X}^{1}\to X\) for \(X\in\Cat{Var}\).
  By \cref{x0606p},
  we can replace \(X\)
  with \(\Sparc(A)\)
  where \(A\) is a light strictly totally disconnected.

  We check this using \cref{aplus}.
  We first
  consider
  the analytification functor
  \(\Cat{QProj}\to\Cat{Stk}_{\lgt}^{\vee}\),
  which maps~\(X\) to~\(X_{A}\).
  By composing this with
  \(\D_{\mot}\)
  and \(\Hom_{\Mot[2](\ZZ)}(\unit,[\X])\),
  we are in a situation
  where \cref{ayoub} applies,
  as we can observe by applying \cref{main_2}.
  The only remaining condition for applying \cref{aplus}
  is that
  \cref{e:9rdu6}
  is left adjointable
  when \(Y\to X\)
  is the \(A\)-analytification
  of a smooth morphism in \(\Cat{QProj}\).

  Using the left adjointability
  of \cref{e:9rdu6}
  for \(\DD^{1}\)-invariant finitary morphisms,
  we are reduced to showing this
  when \(Y=\AA^{1}_{X}\).
  By compactifying~\(X\)
  and using descent again,
  we are reduced to showing that
  \cref{e:9rdu6}
  is left adjointable
  for \(Y=\AA^{1}_{X}\to X\).
  It suffices to check this left adjointability
  at \([V]\in\D_{\mot}(Y)\),
  where \(V\) is an open subspace of \(\AA^{n}_{Y}\)
  for some~\(n\),
  since they generate \(\D_{\mot}(Y)\)
  under colimits and shifts by~\cite[Proposition~5.13]{ScholzeB}.

  Now we consider
  \(f\colon Y\to X=\Sparc(A)\) to be an open subspace of \(\AA^{n}_{X}\).
  It suffices to show that
  the square \cref{e:9rdu6} is left adjointable
  at \(\unit\in\D_{\mot}(Y)\).
  By recollement,
  it suffices to treat
  the archimedean and nonarchimedean cases
  for~\(A\) separately.

  For the archimedean case,
  we claim that the square
  \begin{equation*}
    \begin{tikzcd}
      \Shv(\Sp(X))\otimes\unit\ar[r]\ar[d]&
      \Shv(\Sp(Y))\otimes\unit\ar[d]\\
      {[X]}\ar[r]&
      {[Y]}
    \end{tikzcd}
  \end{equation*}
  is a pushout in \(\CAlg(\Mot[2](\ZZ))\),
  which implies the desired result.
  To show this,
  by compactifying~\(Y\),
  we can replace it with \(\PP^{n}_{X}\);
  this case then follows from descent.

  We then treat the nonarchimedean case.
  We take \(r>0\) small enough so that
  \(\BB_{r,X}^{n}\) acts on~\(Y\) by translation.
  Then \(g\colon Y/\BB_{r,X}^{n}\to X\)
  is finitary and static.
  Moreover,
  the canonical morphism
  \(f_{\natural}\unit\to g_{\natural}\unit\)
  is an equivalence in \(\D_{\mot}(X)\)
  and also in \(\Hom_{\Mot[2](\ZZ)}(\unit,[X])\).
  Therefore, the desired result
  follows from the left adjointability of \cref{e:9rdu6}
  for~\(g\).
\end{proof}

\bibliographystyle{plain}
\let\AA\oldAA \let\L\oldL \let\SS\oldSS  \newcommand{\yyyy}[1]{}

\end{document}